\documentclass[10pt,reqno,english]{amsart}
\usepackage{amsfonts,amsmath,latexsym,verbatim,amscd,mathrsfs,color,array}
\usepackage[colorlinks=true]{hyperref}

\usepackage{amsmath,amssymb,amsthm,amsfonts,graphicx,color}
\usepackage{amssymb}
\usepackage{pdfsync}
\usepackage{epstopdf}

\newcommand{\ass}{\quad\mbox{as}\quad}

\DeclareMathOperator{\sech}{sech}

\newcommand{\inn}{{\quad\hbox{in } }}

\newcommand{\ttt}{\tilde }

\newcommand{\pp}{ {\partial} }

\newcommand{\E}{\mathcal{E} }

\newcommand{\N}{\mathbb{N}}

\newcommand{\R} {\mathbb R}

\newcommand{\cuad}{{\sqcap\kern-.68em\sqcup}}

\newcommand{\foral}{\quad\mbox{for all}\quad}

\newcommand{\be}{\begin{equation}}
	\newcommand{\ee}{\end{equation}}

\newcommand{\la}{\lambda}

\newcommand{\equ}[1]{(\ref{#1})}

\newtheorem{lemma}{Lemma}[section]
\newtheorem{prop}{Proposition}[section]
\newtheorem{theorem}{Theorem}

\newtheorem{remark}{Remark}[section]
\newcommand{\bremark}{\begin{remark} \em}
	\newcommand{\eremark}{\end{remark} }

\numberwithin{equation}{section}

\begin{document}
	
	\title[Delaunay-like compact equilibria]{Delaunay-like compact equilibria in the liquid drop model }
	
	\author[M. del Pino]{Manuel del Pino}
	\address{\noindent M.dP.:  Department of Mathematical Sciences University of Bath,
		Bath BA2 7AY, United Kingdom.}
	\email{mdp59@bath.ac.uk}

	\author[M. Musso]{Monica Musso}
	\address{\noindent M.M.:  Department of Mathematical Sciences University of Bath,
		Bath BA2 7AY, United Kingdom.}
	\email{mm6823@bath.ac.uk}

	\author[A. Zuniga]{Andres Zuniga}
	\address{\noindent A.Z.: Instituto de Ciencias de la Ingenier\'ia, Universidad de O’Higgins (UOH), Avenida Libertador Bernardo O’Higgins 611, Rancagua, Chile.}
	\email{andres.zuniga@uoh.cl}
	
	\maketitle
	
	\begin{abstract}
		The {\em liquid drop model}  was introduced by Gamow in 1928 and Bohr-Wheeler in 1938 to model atomic nuclei. The model describes the competition between the surface tension,
		which keeps the nuclei together, and the Coulomb force, corresponding to repulsion among protons. More precisely, 
		the problem consists of finding a surface $\Sigma =\partial  \Omega$ in $\R^3$ that is critical for the  energy 
		$$
		\E(\Omega) = {\rm Per\,} (\Omega )  +  \frac 12 \int_\Omega\int_\Omega  \frac {dxdy}{|x-y|}
		$$
		under the volume constraint $|\Omega| = m$. The term ${\rm Per\,} (\Omega ) $ corresponds to the surface area of $\Sigma$.
		The associated Euler-Lagrange equation is
		$$
		H_\Sigma (x) + \int_{\Omega } \frac {dy}{|x-y|} = \lambda \quad \hbox{ for all } x\in \Sigma, \quad 
		$$
		where $H_\Sigma$ stands for the mean curvature of the surface,   and where $\lambda\in\R$ is the Lagrange multiplier associated to the constraint $|\Omega|=m$. Round spheres enclosing balls of volume $m$ are always solutions. They are minimizers for sufficiently small $m$. 
		Since the two terms in the energy compete, finding non-minimizing solutions can be challenging. We find a new class of compact, embedded solutions with large volumes, whose geometry resembles a ``pearl necklace" with an axis located on a large circle, with a shape close to a Delaunay's unduloid surface of constant mean curvature. 
		The existence of such equilibria is not at all obvious, since for the closely related constant mean curvature problem $H_\Sigma = \la$, the only compact embedded solutions are spheres, as stated by the classical Alexandrov result. 
		
	\end{abstract}

	\section{Introduction} 

	For open regions $\Omega \subset \mathbb{R}^3$, we consider the energy functional
	\begin{equation}\label{E}
		\E(\Omega) = \text{Per}(\Omega) + D(\Omega) , \quad D(\Omega) = \frac{1}{2} \int_\Omega \int_\Omega \frac{dx\,dy}{|x-y|},
	\end{equation}
	where $\text{Per}(\Omega)$ denotes the perimeter of $\Omega$, which in the smooth case corresponds to the surface area of its boundary $\partial\Omega$.
	This energy traces back to Gamow’s liquid drop model \cite{bohr,Gamow}, a theoretical model used in nuclear physics to describe the structure of atomic nuclei. It treats the nucleus as a drop of incompressible fluid represented by 
	$\Omega $, where nucleons (protons and neutrons) are bound together by strong nuclear forces, explaining phenomena like nuclear fission and nuclear binding energies.
	
	The nucleons are assumed to be distributed with constant density, implying that the number of nucleons is proportional to $|\Omega|$. The perimeter term in the energy \equ{E} corresponds to surface tension, which holds the nuclei together. The second term represents the electrostatic Coulomb repulsion among protons.
	
	The basic variational problem in search for equilibria of such system is as follows: Given $m > 0$, we aim to find regions $\Omega\subset\R^3$ with smooth boundaries that are critical for the energy $\E$ under the volume constraint $|\Omega| = m$. Initially formulated by Bohr and Wheeler \cite{Bohr Wheeler} to describe the mechanism of nuclear fission, the problem consists in identifying regions $\Omega$ that, for some constant Lagrange multiplier $\lambda$, solve the following equation:
	
	\begin{equation}
		H_{\partial\Omega}(x) + \int_\Omega \frac{dy}{|x-y|} = \lambda \quad \foral x\in \partial\Omega. \label{eq:equation}
	\end{equation}
	Here, $H_{\partial\Omega}(x)$ represents the mean curvature operator of the boundary $\partial\Omega$, and the integral term involves the interaction potential between the boundary point $x$ and points $y$ inside the region $\Omega$.
	
	The first observation is that balls with volume $m$ always constitute solutions to this problem. They are critical for both terms in the formula \eqref{E}. Due to the 
	classical isoperimetric inequality, balls minimize perimeter subject to the volume constraint $|\Omega| = m$; see~\cite{DeGiorgi} and also \cite[Theorem 14.1]{maggi}. On the other hand, the balls instead maximize the Coulomb interaction term, as shown in \cite{Riesz}, and also in the book \cite[Theorem 3.7]{Lieb-Loss}. The competing nature of the two terms in the energy \eqref{E} renders problem \eqref{eq:equation} delicate. This problem has been extensively treated in recent years, see \cite{choksi} for a detailed survey of results until 2017.

	Alexandrov’s classical theorem asserts that spheres are the only compact embedded surfaces $\Sigma$ that are critical for perimeter under the enclosed volume constraint, namely, with constant mean curvature, $H_{\Sigma}(x) = \lambda$ for all $x\in \Sigma$. In contrast, Problem \eqref{eq:equation} exhibits richer features.
	
	\medskip
	Let $\Omega $ be a region with $|\Omega|=m$. The scaling $ \Omega = m^{1\over 3} E$, so that $|E|=1$, shows that \equ{E} becomes
	$$
	{\mathcal E} (\Omega ) = m^{2\over 3} \left( {\rm Per} (E) + m D(E) \right) .
	$$
	In the small volume regime $m\sim0$, the above expression indicates that the perimeter term dominates the energy; while the Coulomb interaction energy is dominant in the large mass regime $m>>1$. This suggests that for any mass $m > 0$ sufficiently small, there is a global minimizer, while for large $m$, there are no global minimizers. Indeed, Knupfer-Muratov \cite{Knupfer-Muratov} proved that for any small $m>0$,  the energy \eqref{E} has the ball $B$ with $|B|=m$ as a minimizer, see also Julin \cite{julin} and  Bonacini-Cristoferi~\cite{Bonacini-Cristofri}. 
	Balls are no longer global minimizers for $m>m_*$, where 
	$$
	m_*= 5
	\frac{(2^{1/3} -1)} {
		1-2^{-2/3}}
	\approx  3.51 .
	$$
	The value of $m_*$ is precisely the one for which the energy of one ball of mass $m_*$ equals the energy of two balls of half such mass, located at an infinite distance: 
	$$
	\E\left( \left({m_* \over |B_1|}\right)^{1/3} \, B_1 \right) = 2\, \E \left( \left({m_* \over 2|B_1|}\right)^{1/3} \, B_1 \right).
	$$
	Choksi and Peletier \cite{Choksi Peletier}  conjectured that global minimizers exist if and only if $m\in(0,m_*]$ and, moreover, they are precisely balls. 
	Frank and Nam \cite{frank2} proved that for all $0<m\le m_*$ there exists a global minimizer. They also established that these minimizers are balls under the condition that for $m>m_*$  there is no minimizer. Frank-Killip-Nam \cite{frank1} proved that for $m>8$ there is no minimizer, improving a result by Lu-Otto \cite{Lu Otto}. The fractional scenario for minimizers was analyzed in \cite{figalli}, while the study of minimizers in the setting of density perimeter was carried out in \cite{zuniga}.  See also 
	\cite{frank3,RenWei0} for results on related models. 
	
	It is natural to ask whether for $m>m_*$ there exist solutions to Problem \eqref{eq:equation} other than balls. Very few solutions are known for larger $m$.
	It turns out that balls are local minimizers well beyond $m_*$: in fact, for any $0<m\le 10$, balls are linearly stable while stability is lost for $m>10$. At $m=10$ a local bifurcation branch in $m$ appears which, for $m>10$ yields a local minimizer that is not a ball and, for $m<m_*$, renders a saddle point. The presence of this bifurcation was first detected by Bohr and Wheeler \cite{Bohr Wheeler} and rigorously established by Frank \cite{frank4}. As far as the authors are aware, the only known solutions to Problem  \equ{eq:equation}, other than balls for very large $m$, are the ones constructed by Ren-Wei: the revolution torii~\cite{RenWei1}, and the double torii~\cite{RenWei2}.
	
	Xu and Du \cite{Xu-Du} numerically obtained a global description of the bifurcation branches, leading for large masses to toroidal and Delaunay-like shapes, along with providing interesting pictures.
	
	This paper aims to unveil a new class of compact embedded solutions to \equ{eq:equation} with Delaunay-type shapes. The existence of such equilibria is not at all obvious, since for the closely related constant mean curvature problem $H_\Sigma = \la$, the only compact embedded solutions are spheres, as stated by the classical Alexandrov result.

	\subsection{Delaunay unduloids}\label{d}
	The Delaunay surfaces in the Euclidean 3-dimensional space, denoted as \( \mathbb{R}^3 = \{(y_1, y_2, y_3) \mid y_j \in \mathbb{R}\} \), are Constant Mean Curvature (CMC) surfaces of revolution which are translationally periodic. After a rigid motion and a dilation, we can position its axis of revolution to align with the \( y_3 \)-axis, and its constant mean curvature is set to \( H = 2 \) (hereafter assumed).

	We examine the profile curve in the half-plane \( \{(y_1, 0, y_3) \in \mathbb{R}^3 \mid y_3 > 0\} \), which, when rotated about the \( y_3\)-axis, traces out a Delaunay surface. This curve alternates periodically between maximal and minimal heights concerning the positive \( y_3 \) direction. These respective heights are referred to as the bulge radius and the neck radius of the Delaunay surface. Denote the neck radius by \( a \).
	Embedded Delaunay surfaces manifest as a 1-parameter family, known as {\em unduloids}, which can be parameterized by the neck radius \( a \in (0, 1/2] \).  For unduloids, \( a = 1/2 \) corresponds to the round cylinder. The singular surface as \( a\rightarrow 0 \) is a chain of tangent spheres with radii 1 centered along the \( y_3 \)-axis.
	Notice that we employ the definition of the mean curvature being the sum of the principal curvatures, rather than their average.
	
	\begin{figure}  { Delaunay surface  $\color{blue} \Sigma_a$, 
			$\color{blue} 0<a< \frac 12 $ }
		
		\smallskip
		\includegraphics[scale=0.25]{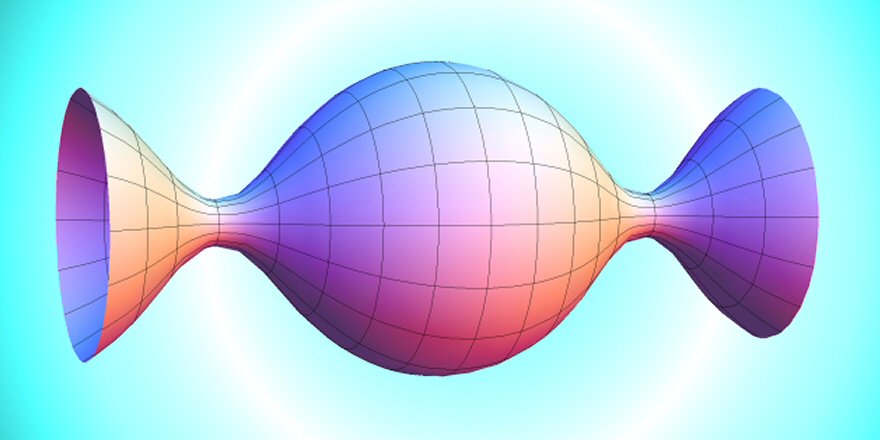}
	\end{figure}

	\medskip
	We assume $0<a<{1\over 2} $ is fixed in what follows.  
	The unduloid $\Sigma$ with neck size $a$, also written $\Sigma(a)$, can be parametrized in the form 
	\be \label{parametrization}
	y (\omega) =  (f(y_3)\cos\theta, f(y_3) \sin\theta , y_3),  \quad \omega = (\theta,y_3) \in [0,2\pi)\times \R     
	\ee
	where $f (t) = f_a (t) $ is a  positive, smooth function, periodic with period $T = T_a >0$, corresponding to the distance between two consecutive necks, its minima, which  
	can also be chosen even, i.e. 
	$ f(s)= f(-s) $.    
	The function $f(s)$ solves the Cauchy problem
	\be \label{eqf} \left\{
	\begin{aligned}
		- {f'' \over (1+ (f')^2 )^{3\over 2}} &+ {1\over f \sqrt{1+ (f')^2}} = 2\\
		f(0) = 1-a, & \quad f'(0)= 0. 
	\end{aligned}\right.
	\ee
	Let us denote by $\Omega$ 
	the volume enclosed by $\Sigma$, namely, 
	$$ 
	\Omega\, = \, \big \{\, y \in \R^3 \mid |(y_1,y_2)| < f(y_3)\, \big \}.   
	$$
	$\Sigma$ can be decomposed as the superposition of identical portions which we describe as follows. Let us define 
	\be\label{sigma0}
	\Sigma_0= \left\{\,  y \in \Sigma\,\biggr| -{T \over 2}  \leq y_3 < {T\over 2}  \right\} ,
	\ee
	and denote 
	$$
	\Sigma_k =  \Sigma_0 + (0,0, kT) = \left\{\,  y \in \Sigma\,\biggr| -{T \over 2} + kT \leq y_3 < {T\over 2} + k T \right\} .$$  
	We define similarly $\Omega_k$. 
	Thus  \be \label{sigmak}  \Sigma =   \bigcup_{k=-\infty}^\infty \Sigma_k, \quad  \Omega =   \bigcup_{k=-\infty}^\infty \Omega_k .\ee
	The volume of a single block of the Delaunay surface $\Sigma=\Sigma (a)$ will be denoted by
	\be \label{volume0}
	V= V_a = |\Omega_0|.
	\ee
	To carry out the construction of the critical set, we first consider finite truncations of the sets above, $\Sigma^n$ and $\Omega^n$, given by 
	\be \label{truncated}  
	\Sigma^n :=   \bigcup_{k=0}^{n-1} \Sigma_k, \quad  \Omega^n :=   \bigcup_{k=0}^{n-1}\Omega_k  \quad \text{ for } n\in\N, 
	\ee
	and then we allow for \emph{normal graphs perturbations} of such sets, $\Sigma^n_h$ and $\Omega^n_h$, depending on a small smooth function $h(y)$ defined on $\Sigma^n$, as follows
	\be \label{truncated-modified}
	\Sigma_h^n := \{ y + h(y) \nu (y) \, | \, y \in \Sigma^n \},
	\ee
	where $\nu(y)$ is the  
	unit normal vector at $y\in \Sigma$, which is explicitly given by  
	\be \label{normal0}
	\nu(y) = \frac {1}{\sqrt{1+ f' (y_3)^2}}  \left(  \begin{aligned}   & 
		f(y_3)^{-1} { (y_1, y_2) }\\ &  -f'(y_3)  \end{aligned}\right) 
	\ee
	and the corresponding enclosed region $\Omega_h^n $ is described by
	$$
	\Omega_h^n := \{ (r y_1, ry_2 , y_3)  \mid  y= (y_1,y_2 , y_3) \in \Sigma_h^n , \, r \in [0,1]\,  \}.$$
	The construction procedure allow us to work with perturbations $h: \Sigma \to \R$ satisfying the following symmetries:
	\be\label{simh} 
	\begin{aligned}
		h(y_1, y_2, y_3) &=  h(-y_1 , y_2 , y_3 ) , \\  h(y_1, y_2 , y_3) &= h(y_1, y_2 , -y_3 ), \\ h(y_1, y_2 , y_3) &= h(y_1, y_2 , y_3 +T).
	\end{aligned}
	\ee
	Let us now consider the change of coordinates $X$ given by
	\begin{equation}\label{defX}
		\begin{aligned}
			X:\R^3\setminus\{0\}&\to\R^3\setminus\{0\},\quad (\tilde y_1, \tilde y_2, \tilde y_3) =  X(y_1,y_2,y_3),\\
			X(y_1,y_2,y_3) 
			&=
			\Big ( y_1 \, , \, (R+y_2) \cos 
			\left (  \frac  { y_3} { R}    \right ) ,  (R+y_2) \sin  \left (\frac  { y_3} { R}    \right )\Big ),
		\end{aligned}
	\end{equation}
	where $R>0$ is determined by the constraint
	\be \label{defR}
	2\, \pi \, R= n \, T_a.
	\ee
	The transformation $X$ is the composition of a rigid translation of $(0,R,0)$, followed by  a rotation in the plane $(y_2,y_3)$ at the angle ${y_3 \over R}$.
	Finally, we now define the {\em  Delaunay unduloid collar} $\tilde \Sigma^n := X(\Sigma^n)$
	and its perturbed version 
	\be\label {7}
	\tilde \Sigma_h^n  :=  X \left( \Sigma_h^n \right) ,\ee  
	and we also introduce the associated solid bounded regions
	\be\label{tt}
	\tilde \Omega^n = X \left(  \Omega^n \right) , \quad \tilde \Omega_h^n = X\left( \Omega_h^n \right).
	\ee
	Under the symmetry assumptions \eqref{simh} on $h$, we have that, for $$
	\tilde y \in \tilde \Sigma_h^n, \quad \tilde y = X (y), \quad y \in \Sigma_h^n,
	$$
	there holds
	\be \label{variable}
	\begin{aligned}
		X ( y_1,y_2 , y_{3} +T)&= {\mathcal R}_{2\pi \over n} \tilde y \\
		X ( -y_1,y_2 , y_{3} )&= {R}_{2,3} \tilde y , \quad -{T\over 2} < y_3 <{T\over 2} \\
		X ( y_1,y_2 , - y_{3} )&= { R}_{1,2} \tilde y ,\quad -{T\over 2} < y_3 <{T\over 2}
	\end{aligned}
	\ee
	where  ${\mathcal R}_{\theta}$ is the rotation of angle $\theta$ in the $(y_2,y_3)$-plane, $R_{2,3}$ is the reflection with respect to the plane $y_1=0$, and
	$R_{1,2}$ is the reflection with respect to the plane $y_3=0$:
	{\small {$$
			\begin{aligned}
				{\mathcal R}_{\theta} &=\left( \begin{matrix} 1 & 0 & 0 \\
					0 & \cos \theta & -\sin \theta \\
					0& \sin \theta & \cos \theta\end{matrix} \right), \quad R_{2,3}= \left( \begin{matrix} -1 & 0 & 0 \\
					0 & 1 & 0 \\
					0& 0 & 1\end{matrix} \right), \quad 
				R_{1,2}= \left( \begin{matrix} 1 & 0 & 0 \\
					0 & 1 & 0 \\
					0& 0 & -1\end{matrix} \right).
			\end{aligned}
			$$}}

	\subsection{Main Result} 
	
	In this paper, we prove that in the large mass regime $m$, we can find a collar-like set $\Omega$ that is an equilibrium (critical) for the liquid drop problem, which is a small normal perturbation of a coiled Delaunay surface, of small neck size $a$.
	
	Our main result states that, for any sufficiently small neck size $a>0$ and any large mass $m>>1$, there exists a solution $ \Omega$ to the problem  \equ{eq:equation} with $ | \Omega|=m$ and such that the scaled domain 
	$ c{|\log n|^{\frac 13} } \Omega $ is a $O( |\log n|^{-1})$-perturbation of the coiled unduloid $\ttt \Omega^n$ in \equ{tt}. The number $n$ is taken close to $ [m(\log m - \log(\log m)) C_a]$ 
	for an explicit positive number $C_a$, where $[\cdot]$ denotes the integer part. 
	
	\begin{theorem}\label{teo1} 
		For any sufficiently small neck size $a>0$ there exist (explicit) positive constants $c_a, C_a$ so that, for all 
		sufficiently large  $m>1$ there is a domain $\Omega $ with 
		$| \Omega|=m$, solution of Problem $\eqref{eq:equation}$ and $n\in\N$ with 
		\be \label{mm}  
		|[m(\log m - \log(\log m)) C_a] - n| \le   1,
		\ee
		such that the scaling $ c_a|\log n|^{\frac 13}\Omega$ is a $O( |\log n|^{-1})$-normal perturbation of the  surface $\ttt \Omega^n$ in $\equ{tt}$.
	\end{theorem} 
	

	
	Let us consider a number $\alpha>0$ and replace $\Omega $ with $\alpha\ttt \Omega$  
	in equation \equ{eq:equation} and get  the equivalent problem 
	$$
	\alpha^{- 1 } H_{\pp \ttt \Omega} ( x) +   
	\int_{ \ttt \Omega } \frac { \alpha^3 dy}{| \alpha x- \alpha y|} =  \la  \foral x\in  \pp\ttt \Omega,      
	$$
	so that letting $\gamma = \alpha^3$ and redefining the Lagrange multiplier $\la$, 
	we get that Problem  \equ{eq:equation}  with $|\Omega|=m$ is equivalent to finding $\ttt \Omega$  such that 
	\be \label{3new}  
	\left \{ \begin{aligned}
		|\ttt \Omega | = & \gamma^{-1}  m ,   \\  H_{\pp \ttt \Omega} ( x)  \, + &\, \gamma \int_{  \ttt \Omega } \frac {  dy}{|  x-  y|}\  =\ \la  \foral x\in \pp\ttt \Omega . 
	\end{aligned}\right.
	\ee 
	To prove Theorem \ref{teo1} we consider a large number $n\in\N$ and then prove that there exists a number $\gamma \approx \frac c{\log n} $ with $c>0$ depending of the neck size $a$, and a region $\ttt \Omega$ satisfying 
	\be \label{3}
	H_{\pp \ttt \Omega} ( x)  \, + \, \gamma \int_{  \ttt \Omega } \frac {  dy}{|  x-  y|}\  =\ \la  \foral x\in \pp\ttt \Omega,
	\ee
	for a domain $\ttt \Omega $ which is $O(|\log n|^{-1})$-close to the coiled truncated Delaunay $\ttt \Omega^n$. This is the content of Theorem \ref{teo2} below. We will be able to establish this for $a>0$ small or $\frac 14 \leq a< \frac 12$. Using a continuity argument in small $a>0$ 
	we will then be able to establish the relation $|\ttt \Omega | = \gamma^{-1}  m $ for any large $m$ and suitable $n$ that satisfies \equ{mm}, so that \equ{3} is satisfied. 
	\begin{figure}  { Domain in Theorem \ref{teo1}: $\color{blue} c_a|\log n|^{\frac 13} \Omega \approx \ttt \Omega^n $ }
		
		\smallskip
		\includegraphics[scale=0.3]{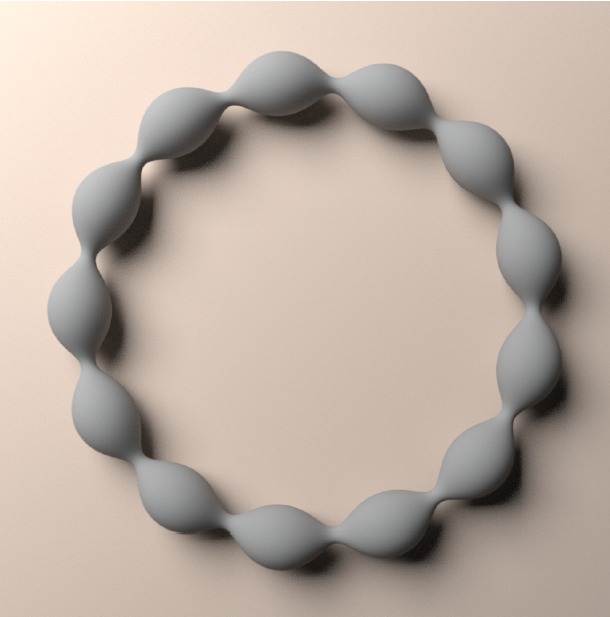}
	\end{figure}

	\section{Scheme of the proof.}
	
	Let us recall that we want to solve \equ{3}, that is, 
	\be \label{33}
	H_{\ttt \Sigma } ( x)  \, + \, \gamma \int_{  \ttt \Omega } \frac {  dy}{|  x-  y|}\  =\ \la  \foral x\in \ttt \Sigma =\partial \ttt \Omega. 
	\ee
	First, we fix a Delaunay surface $\Sigma = \Sigma(a)$ for a fixed value of the neck size $a$, for which we prove the existence of  solutions  to Problem \eqref{33} with $\ttt \Sigma$ of the form \equ{7}, namely
	$$
	\ttt \Sigma := \tilde{\Sigma}_h^n = \pp \ttt \Omega_h^n,
	$$ 
	provided $\gamma$ is chosen to fulfill a certain solvability condition. This is the main step of our construction and will occupy most of the paper.
	
	
	\medskip
	The second step consists of showing that by changing the value of the parameter $a$, we can still obtain a surface solution to \eqref{33} that is of the form \eqref{7}, having its volume prescribed by any large value $m$. Thus, we fix $0<a< \frac 12$ and define
	\be \label{defIa}
	I_a:= \int_0^{T/2} \frac f{(1+ f'(s)^2)^\frac 52 } \Big [  ff''(2-f'(s)^2)  + (1+ 3f'(s)^2)   (1+ f'(s)^2)    \Big ]\, ds.
	\ee
	where $T=T_a$ is the period of the Delaunay surface $\Sigma$ and $f=f_a$ is the profile describing $\Sigma$  as in \eqref{parametrization}, characterized by equation  \eqref{eqf}. We denote by $V_a= |\Omega_0|$ the volume of a single block repeated $T_a$-periodically, to conform the whole of $\Sigma$.  
	\medskip
	
	We will establish the following result.
	
	\begin{theorem}\label{teo2}
		Let $a \in (0,{1\over 2})$ be such that
		\be \label{Ia}
		I_a>0.
		\ee
		For all sufficiently large $n\in\N$, there exist  $\gamma_{n,a}, \lambda_{n,a}\in\R$, a function $h=h_{n,a}:\Sigma (a) \to \R$, and a compact surface as in \equ{7} of the form 
		$$  
		\tilde \Sigma_h^n  = \pp \tilde \Omega_h^n,  
		$$    
		which solves Problem \eqref{33}, with
		\be \label{volume}
		|\tilde \Omega_h^n | = n \, V_a \, \left(1+ o(1) \right) \ass n\to \infty. 
		\ee
		Moreover, there exists a uniform constant $C>0$  such that
		\be \label{est}
		\| h \|_{C^{2, \alpha} (\Sigma)} \leq {C \over \ln n}, \quad \gamma_{n,a} = {2 I_a \, T_a \over V_a } \, {1\over \ln n} \, \left(1+ o(1) \right) 
		\ee
		In \eqref{volume} and \eqref{est}, the limit $o(1) \to 0$ as $n \to \infty$ is uniformly in $a$. 
	\end{theorem}
	
	\begin{remark}
		The positivity of the integral $I_a$ is a crucial element in our construction. When $a=0$ we have $\frac T2 =1$ and $f(s) = \sqrt{1-s^2}$, so the integral equals $0$. When $a=\frac 12 $, the case of the cylinder, we have  $f\equiv  \frac 12$ and the integral is trivially positive. We also have $T\to\pi$ as $a\to\frac 12^{-}$. In fact, assumption \equ{Ia} holds for any small neck size $a$ and also for $a\geq\frac 14$. This follows from the following lemma, whose proof we postpone to appendix \S \ref{appe}. We conjecture that $I_a>0$ for all $a\in(0,\frac 12)$.
	\end{remark}
	
	\begin{lemma}\label{LemmaIa}
		There exists $a_* \in (0, {1\over 4})$ such that
		$$
		I_a >0  \foral a \in (0 , a_*] \cup [\tfrac 14 , \tfrac 12).
		$$
		In addition, there holds
		\be\label{iia}
		I_a = 2 a + O(a^2), \quad {\mbox {as}} \quad a \to 0^+.
		\ee
	\end{lemma}
	
	\medskip
	As in many problems of this type, including the related construction of asymptotically singular CMC surfaces, such as those in \cite{Mazzeo Pacard},  the proof of Theorem \ref{teo2} involves \emph{a Lyapunov–Schmidt type reduction argument}. We now briefly detail this procedure, applied to our problem. For all $\gamma $ suitably small, we perform an expansion of the terms involved in the equation as functions of the perturbation $h$. In 
	Section \S \ref{sec3} we expand the mean curvature $H_{\tilde \Sigma_h^n}$ of the perturbed toroidal Delaunay $\tilde \Sigma_h^n$, see \eqref{7}. This expansion involves the Jacobi operator of the Delaunay surface under periodic and symmetry conditions. In Section \S \ref{sec4} we expand the integral operator in $h$ and combine inhomogeneous and smaller order terms in $h$ arising from both parts. This leads to solving a nonlinear PDE involving the Jacobi operator of the perturbation $h$ and an integral operator of $h$, for which we first solve a projected version. The solution turns out to be a nice function of the parameter $\gamma$. A solution to the full problem is then found by suitably adjusting $\gamma$, that needs to be of size $O(|\log n|^{-1})$. This approach leads us to a perturbation $h$ of a comparable size. Finally, a fixed point argument is performed to solve the nonlinear PDE, in which the Jacobi operator is inverted under a zero-mean condition for $h$. 
	
	In sections \S 5 and \S 6 we introduce the necessary linear elements and the full fixed point argument.  We devote the appendix to some technical results needed in the body of the paper. 
	
	
	\medskip
	The result in Theorem \ref{teo2} shows that for $n\in\N$ large enough there exists  a solution to Problem \eqref{3} for regions $\Omega\subset\R^3$ satisfying
	\begin{align*}
		& |\Omega | = \gamma_{n,a}^{-1} m, \quad {\mbox {for}} \quad m= \gamma_{n,a}  \,  |\tilde \Omega_h^n | = {n \over \ln n} \, g_n (a) \\
		&{\mbox {where}} \quad g_n (a) = 2 \, I_a \, T_a \,  (1+ o(1)),
	\end{align*}
	and $o(1) \to 0$ as $n \to \infty$, uniformly for $a$.
	The second step in our argument consists in showing that the result remains true for any $m$ sufficiently large.
	
	\medskip
	To this purpose, let $m\in\R$ be given, large and positive, and take $n \in \N$ so that
	$$
	{n \over \ln n} \, g_n (a)  < m < {n +1 \over \ln (n+1) } \, g_{n+1} (a), 
	$$
	for some $a\in (0 , a_*)$ fixed as in Lemma \ref{LemmaIa}.
	Since
	$$
	{n+1 \over \ln (n+1)} ={n \over \ln n}
	\left( 1+ {1\over n \ln n} {1\over 1+ {\ln (1+ n^{-1}) \over \ln n}} - {\ln (1+ n^{-1}) \over (\ln n)^2} \right),
	$$
	we have that
	$$
	{n \over \ln n} \, g_n (a)  < m < {n  \over \ln n } \, g_{n+1} (a) (1+ o(1)), \quad 
	$$
	with $o(1) \to 0$  as $n \to \infty$. From \eqref{iia} in Lemma \ref{LemmaIa} and possibly choosing a smaller $a_*$, we can find $\delta >0$ small (for large $n$) such that
	\begin{align*}
		& a-\delta,\, \, a+\delta \in (0,  a_*),\text{ and}\\
		& {n \over \ln n} g_n (a-\delta ) < {n \over \ln n} \, g_n (a)  < m < {n  \over \ln n } \, g_{n+1} (a) (1+ o(1))  < {n \over \ln n} g_n (a+\delta ).  
	\end{align*}
	From the Intermediate Value Theorem there exists $b \in (0 ,  a_*)$ such that $m= {n \over \ln n} g_n (b),$ and Theorem \ref{teo2} guarantees the existence of a solution to \eqref{3} with $|\Omega|=m$. This concludes the proof of Theorem \ref{teo1}.
	
	\medskip
	Observe that the result in Theorem \ref{teo2} is valid for $a\in[{1\over 4} , {1\over 2})$. Moreover, the validity of Theorem \ref{teo2} only relies upon the condition $I_a >0$. This condition is satisfied in this range of values for $a$. 
	Further information on $I_a$, such as on its derivatives in $a$, are required to establish Theroem \ref{teo1} for $a$ in this range.  While we do not address this issue here, we believe that
	$I_a >0$ and ${d I_a \over da} >0$ for all $a \in (0,{1\over 2}).$

	\section{Expansion of the curvature for  $\tilde \Sigma^n_h$} \label{sec3}

	The purpose of this section is to derive an expression for the mean curvature $H_{\tilde \Sigma_h^n}$ of the perturbed toroidal Delaunay $\tilde \Sigma_h^n$, see \eqref{7} in the Introduction. Before stating our result, we introduce some notation.
	
	The Delaunay surfaces $\Sigma^n$ mentioned in the introduction are surfaces of revolution, thus a natural parametrization for them is given in cylindrical coordinates. Specifically, if the axis of rotation is vertical, and $y_3$ represents a coordinate along this axis while $\theta$ represents the angular variable around it, then a parametrization of  the unperturbed surface $\Sigma^n$ is
	\begin{align*}
		y (\omega ) = \left[ \begin{matrix} f (y_3)  \cos \theta \\
			f(y_3) \sin \theta \\
			y_3 \quad  \end{matrix}\right] , \quad \omega = (\omega_1 , \omega_2) =(\theta , y_3),
	\end{align*}
	see \eqref{parametrization}-\eqref{eqf}. The set of parameters $\omega$ is given by
	\be \label{defLambda}
	\omega \in \Lambda 
	:=\{ (\theta , y_3) \, : \, \theta \in [0,2\pi), \, y_3 \in [-\tfrac T2 , -\tfrac T2 + n T ] \}
	\ee
	and the function $f$ solves the Cauchy problem \eqref{eqf} and it is periodic of period $T$.
	Tangent vectors to the surface $\Sigma^n$ are linear combinations of 
	\be \label{tangetv}
	Y_1 (\omega )= D_{\omega_1} y =  \left[ \begin{matrix} -f (y_3)  \sin \theta \\
		f(y_3) \cos \theta \\
		0 \quad  \end{matrix}\right], \quad Y_2 (\omega )= D_{\omega_2} y =  \left[ \begin{matrix} f' (y_3)  \cos \theta \\
		f'(y_3) \sin \theta \\
		1 \quad  \end{matrix}\right]
	\ee
	where $D_{\omega_1} = D_\theta$, $D_{\omega_2} = D_{y_3}$.
	The unit normal vector $\nu (y)$ to $\Sigma^n$ at $y \in \Sigma^n$ as in \eqref{normal0} is expressed in the coordinates $\omega$ as
	\be \label{normal}
	\nu (\omega ) = {1\over \sqrt{1+ f' (y_3)^2} } \left[ \begin{matrix}  \cos \theta \\
		\sin \theta \\
		-f' (y_3 ) \quad  \end{matrix}\right], \quad \omega = (\theta , y_3) .
	\ee
	The metric on $\Sigma^n$ is defined by
	\be \label{defgg}
	g= Y^T \, Y, \quad {\mbox {where}} \quad Y = [Y_1 \, \, Y_2] 
	\ee
	and the second fundamental form corresponds to
	\be \label{defAA}
	A= - (Y^T \, Y)^{-1} \, Y^T \, D_\omega \nu.
	\ee
	In this Section we prove the following result.	
	\begin{prop}\label{lemma1}
		Let $\alpha \in (0,1)$ and  $h: \Sigma \to \R$ be a $C^{2,\alpha}$-function satisfying~\eqref{simh}. For a point $y\in\Sigma^n$, we write
		\begin{equation}\label{tildeHH}
			\tilde H_{\tilde \Sigma_h^n} ( y ) = H_{\tilde \Sigma_h^n} \left( \tilde y_h \right) , \quad \tilde y_h  = X (y_h) \in \tilde \Sigma_h^n, \quad 
		\end{equation}
		where
		\be \label{yh}
		y_h = y + h(y) \nu (y) \in \Sigma_h^n,\quad y \in \Sigma^n, 
		\ee
		and $\nu $ given by \eqref{normal0}. Then, there holds the expansion of the curvature
		\be \label{expH}
		\begin{aligned}
			\tilde H_{\tilde \Sigma_h^n} (y ) 
			=&\;  2+ {y_2 \over R \, f} \Biggl(    { (2-(f')^2 ) \, f \, f'' \over (1+ (f')^2 )^{5\over 2}} +{1 + 3 (f')^2 \over (1+ (f')^2 )^{3\over 2}} \Biggl)\\
			& + n^{-2} a (y)  - J_{\Sigma^n} [h]   + n^{-1} \, \ell [h, D h, D^2 h](y) \\
			&+ q[h, D h , D^2 h] (y)\quad {\mbox {as }} n \to \infty.
		\end{aligned}
		\ee 
		where $R$, in \eqref{expH}, is given by \eqref{defR}, $a : \Sigma^n \to \R$ is a smooth function uniformly bounded together with its derivatives as $n \to \infty$, and it satisfies the symmetries \eqref{simh}. Here, $J_{\Sigma^n}$ denotes the Jacobi operator of $\Sigma^n$, defined by
		\be \label{Jacobi}
		J_{\Sigma^n } [h]  =  \frac 1{ \sqrt{\det g}} \pp_j (g^{ij} \sqrt{\det g}\,\pp_j h) + |A|^2 h 
		\ee
		where $g$ is the metric in \eqref{defgg}, $g^{ij}:=(g^{-1})_{ij}$, while $A$ is the second funcdamental form defined in \eqref{defAA}. In \eqref{expH}, $\ell $ is a smooth function, which is uniformly bounded, together with its derivatives, as $n \to \infty$, and it satisfies the symmetries \eqref{simh}. Moreover, $\ell (0,0,0) = 0$, $\nabla \ell (0,0,0) \not= 0$. In \eqref{expH}, $q$ is a smooth function, uniformly bounded with its derivatives as $n \to \infty$, and it satisfies the symmetries \eqref{simh}. It depends of $h$, $D_\omega h$ and $D_\omega^2 h$ with
		$$
		q(0,0,0)=0, \quad D q (0,0,0) = 0, \quad D^2 q(0,0,0) \not= 0. 
		$$
		If $h$ is also even with respect to $y_2$, then $q$ is even in $y_2$.
	\end{prop}
	
	\begin{proof} 
		We introduce the parametrization for $\tilde \Sigma_h^n$ which is induced by one for $\Sigma^n_h$ via the map $X$, see \eqref{defX}.
		
		\medskip
		
		For a small function $h$ on $\Sigma$ we consider the normal graph $\Sigma_h^n$ to the unperturbed Delaunay surface $\Sigma^n$, given by  
		$$
		\Sigma_h^n =   \{  y+ h(y)\nu(y)\ | \    y\in \Sigma^n\} .
		$$
		This surface is naturally parametrized as 
		\begin{align*}
			y_h (\omega ) = \left[ \begin{matrix} (f(y_3) + {h (\omega ) \over \sqrt{1+ (f')^2 (y_3) }} ) \cos \theta \\
				(f (y_3) + {h (\omega )\over \sqrt{1+ (f')^2 (y_3) }} ) \sin \theta \\
				y_3 - {f' (y_3)  h (\omega ) \over \sqrt{1+ (f')^2 (y_3) }}  \quad  \end{matrix}\right] , \quad \omega = (\theta , y_3).
		\end{align*}
		Therefore, a parametrization of the surface 
		$ \tilde \Sigma_h^n =   X( \Sigma_h^n )$ is given by $\tilde y_h = X\circ y_h$ or  
		\begin{align*}
			\tilde y_h (\omega) &= \left[ \begin{matrix} f_h  \cos \theta \\
				( R+ f_h   \sin \theta ) \, \cos \left( {y_3 - \bar  h \over R}\right)  \\
				( R+ f_h   \sin \theta ) \, \sin \left( {y_3 - \bar  h \over R}\right)   \end{matrix}\right] ,  \quad \omega = (\theta , y_3) \quad  {\mbox {where}}\\
			f_h (\omega)  &=f  + {h  \over \sqrt{1+ (f')^2  }},\quad
			\bar h (\omega ) = {f'  h  \over \sqrt{1+ (f'  )^2} }.
		\end{align*}
		As a first step in our argument, we look for a relation between the mean curvatures of the perturbed toroidal-Delaunay surface $\tilde \Sigma_h^n$ and the perturbed Delaunay surface $ \Sigma_h^n$.

		\begin{lemma}\label{prop31}
			Let $\alpha \in (0,1)$ and  $h: \Sigma\to \R$ be a $C^{2, \alpha}$ function satisfying \eqref{simh}. Then for $y \in \Sigma^n$
			\be \label{expH*}
			\begin{aligned}
				\tilde H_{\tilde \Sigma_h^n} (y ) &= H_{\Sigma_h^n } (y_h) 
				+ {\sin \theta \over R} \Biggl(    { (2-(f')^2 ) \, f \, f'' \over (1+ (f')^2 )^{5\over 2}} +{1 + 3 (f')^2 \over (1+ (f')^2 )^{3\over 2}} \Biggl)
				\\&+ n^{-2} \, a(y)
				+ n^{-1} \, \ell [h, D h, D^2 h] \quad {\mbox {as }} n \to \infty,
			\end{aligned}
			\ee 
			where $\tilde  H_{\tilde \Sigma_h^n} $ as in \eqref{tildeHH} and
			$$
			y_h = y + h(y) \nu (y)
			$$
			as in \eqref{yh}.
			Here $a(y)$ and $\ell$ stand for  smooth functions of $y$, which are uniformly bounded with their derivatives, as $n \to \infty$, and they satisfy the symmetries \eqref{simh}. Besides,
			$\ell $ depends on $h$ and its derivatives, with  $\ell (0,0,0) = 0$ and  $\nabla \ell (0,0,0) \not= 0$.
			
		\end{lemma}
		
		The second step in our argument to prove Proposition \ref{lemma1} consists in relating the mean curvatures of the perturbed Delaunay surface $ \Sigma_h^n$ and the un-perturbed Delaunay surface $ \Sigma^n$.

		\begin{lemma}\label{prop32}
			Let $\alpha \in (0,1)$ and  $h: \Sigma^n \to \R$ be a  $C^{2, \alpha}$ function satisfying \eqref{simh}. Then for $y \in \Sigma^n$
			$$
			H_{\Sigma^n_h} (y_h) = H_{\Sigma^n} (y) - J_{\Sigma^n}  [h]  + q[h, D_\omega h , D_\omega^2 h]
			$$
			where 
			$$
			y_h = y + h(y) \nu (y)
			$$
			as in \eqref{yh}, and 
			$J_{\Sigma^n}$ is the Jacobi operator defined in \eqref{Jacobi}.
			The function $q$ is a smooth function of $y$, uniformly bounded with its derivatives as $n \to \infty$, and it satisfies the symmetries \eqref{simh}. If $h$ is also even with respect to $y_2$, then $q$ is even in $y_2$. It depends of $h$, $D_\omega h$ and $D_\omega^2 h$ with
			$$
			q(0,0,0)=0, \quad D q (0,0,0) = 0, \quad D^2 q(0,0,0) \not= 0. 
			$$
		\end{lemma}

		The proof of Proposition \ref{lemma1} readily follows from Lemmas \ref{prop31}
		and \ref{prop32}, and the fact that $H_{\Sigma^n} (y)=2.$

	\end{proof}

	\medskip
	\begin{remark}\label{rem1}
		For $h: \Sigma^n \to \R$  satisfying \eqref{simh}, the function
		$$
		y \to H_{\Sigma_h^n} (y_h), \quad y_h = y + h(y) \nu (y)
		$$
		satisfies the same symmetries \eqref{simh}. Besides, if $h$ is even with respect to $y_2$, then $H_{\Sigma_h^n} (y_h)$ is even in $y_2$. Also the function 
		$$
		y \to \tilde H_{\tilde \Sigma_h^n} (y)
		$$
		defined in \eqref{tildeHH}, enjoys \eqref{simh}: it is even in $y_1$, $y_3$ and periodic of period $T$ in $y_3$. This is consequence of the properties \eqref{variable} of the map $X$. 
	\end{remark}

	\medskip
	
	The rest of the section is devoted to prove Lemmas \ref{prop31}
	and \ref{prop32}.
	\begin{proof}[Proof of Lemma \ref{prop31}] 
		Let 
		$\tilde Y_h$ and $Y_h$ be the $3 \times 2$ matrices given by
		\be \label{defY}
		\tilde Y_h (\omega ) = [ \tilde Y_{h1} \quad  \tilde Y_{h2} ], \quad   Y_h (\omega ) = [  Y_{h1} \quad  Y_{h2} ],
		\ee
		where
		\be \label{defYi}
		\tilde Y_{hj} = D_{\omega_j} \tilde  y_h (\omega), \quad  Y_{hj} = D_{\omega_j}   y_h (\omega),
		\ee
		for $j=1,2$.
		The Riemannian metrics on $\tilde \Sigma_h^n$ and $\Sigma_h^n$ are respectively given by the $2 \times 2$ matrices 
		\be \label{defg}
		\tilde g_h  = \tilde Y_h^T \, \tilde Y_h, \quad  g_h =   Y_h^T  \, Y_h, 
		\ee
		where 
		the symbol $^T$ stands for taking the transposed matrix.
		Let $\nu_h$ be the unit normal to $\Sigma_h^n$
		$$
		\nu_h (\omega )= {Y_{h1} \times Y_{h2}  \over \| Y_{h1} \times Y_{h2} \|} 
		$$
		and  $\tilde \nu_h (\omega) $ be the unit normal vector to $\tilde \Sigma_h^n$,
		$$
		\begin{aligned}
			\tilde \nu_h (\omega ) &= {\tilde Y_{h1}  \times \tilde Y_{h2}  \over \| \tilde Y_{h1}  \times \tilde Y_{h2} \|}, \quad {\mbox {where}} \quad \tilde Y_{hj} = D_{\omega_j}  \tilde y_h = D_j \left( X \circ y_h \right) , \quad  j=1,2.
		\end{aligned}
		$$
		We also introduce the $3 \times 2$ matrices for $\tilde \Sigma^n_h$ and $\Sigma_h^n$ given respectively by 
		\be \label{defB}
		\tilde B_h = [ D_{\omega_1}  \tilde \nu_h \quad D_{\omega_2} \tilde \nu_h ] \quad  B_h = [ D_{\omega_1} \nu_h \quad D_{\omega_2} \nu_h ].
		\ee
		The mean curvatures of $\tilde \Sigma^n_h$ and $\Sigma^n_h$ can be computed using  the   following formulas (see for instance \cite{pressley})
		\be \label {curvas}
		\begin{aligned}   
			H_{\tilde \Sigma_h^n} &= \, {d \over dz} \log \left( \sqrt{ \det \tilde g_h (z)} \right)_{|_{z=0}} \\
			H_{ \Sigma_h^n} &=\, {d \over dz} \log \left( \sqrt{ \det  g_h (z)} \right)_{|_{z=0}},
		\end{aligned}\ee
		where for any $z \in \R$ 
		\begin{align*}
			\tilde g_h (z)  &= [ \tilde Y_h + z \tilde B_h  ]^T [ \tilde Y_h + z \tilde B_h  ] , \\ g_h(z) &= [Y_h + z B_h  ]^T [Y_h + z  B_h ].
		\end{align*}
		Let $\tilde A_h$ and $A_h$ be the second fundamental forms of $\tilde \Sigma_h^n$ and $\Sigma_n^h$ respectively
		$$
		\tilde A_h = -  \left(  \tilde Y_h^T \,  \tilde Y_h \right)^{-1}  \,  \tilde Y_h^T 
		\,  D_\omega \tilde \nu_h, \quad 	 A_h = -  \left(   Y_h^T \, Y_h \right)^{-1}  \,   Y_h^T 
		\,  D_\omega  \nu_h .
		$$
		We claim that
		\be \label{tildegh}
		\begin{aligned}
			\tilde g_h (z) &= 
			\tilde Y_h^T \, \tilde Y_h \, \left( I - z \tilde A_h \right)^2 = \tilde g_h \, \left( I - z \tilde A_h \right)^2 \\
			g_h (z) &= 
			Y_h^T \,  Y_h \, \left( I - z  A_h \right)^2 =  g_h \, \left( I - z  A_h \right)^2.
		\end{aligned}
		\ee
		The proof of the first formula in \eqref{tildegh} follows from the following geometric considerations. 
		For any $z \in \R$ a straightforward computation gives
		\begin{align*}
			\tilde g_h(z)&=  \tilde Y_h^T \, \tilde Y_h + z \left( \tilde Y_h^T \,  D_\omega \tilde \nu_h  + D_\omega \tilde \nu_h^T \,  \tilde Y_h \right) + z^2 D_\omega \tilde \nu_h^T \, D_\omega \tilde \nu_h .
		\end{align*}
		By definition of $\tilde \nu_h$ we have that $ \tilde Y_{hi}^T  \, \tilde \nu_h = \tilde \nu_h \cdot  \tilde Y_{hi} =0$ for $i=1,2$ (see \eqref{defYi}). Differentiating this expression by $\omega_j$ and using the fact that $D_{\omega_i }  \tilde Y_{hj} = D_{\omega_j }  \tilde Y_{hi} $, we obtain that
		$D_{\omega_i} \tilde \nu_h \cdot \tilde Y_{hj} = D_{\omega_j} \tilde \nu_h \cdot \tilde Y_{hi}  $. Hence for $i, j=1,2$,
		$$
		\left( \tilde Y_h^T  \, D_\omega \tilde \nu_h \right)_{ij} =  \tilde Y_{hi}^T  \, D_{\omega_j} \tilde \nu_h = 
		D_{\omega_j} \tilde \nu_h^T  \,  \tilde Y_{hi} = \left( D_\omega^T\tilde \nu_h  \, \,  \tilde Y_h \right)_{ij},
		$$
		that is
		$$
		\tilde Y_h^T \, D_\omega \tilde \nu_h = D_\omega^T \tilde \nu_h  \, \tilde Y_h.
		$$
		We use that $D_\omega \tilde \nu_h^T = D_\omega^T \tilde \nu_h$. Besides, the vectors $D_{\omega_i}  \tilde \nu_h $, $i=1,2$, belong to the tangent space to $\tilde \Sigma_h^n$ which is  spanned by $ \tilde Y_{h1}$ and $ \tilde Y_{h2}$. Hence $D_\omega \tilde \nu_h$ coincides with its projection on that space, that is 
		$$
		D_\omega \tilde \nu_h=  \tilde Y_h \left(  \tilde Y_h^T \,  \tilde Y_h \right)^{-1} \,  \tilde Y_h^T \, D_\omega \tilde \nu_h. 
		$$
		Using the above facts, we obtain the first equality in \eqref{tildegh}. The second equality can be derived in the same manner.
		
		Formula \eqref{tildegh} readily gives  that 
		\begin{align*}
			\det \tilde g_h (z) &= 
			\det \tilde g_h \left( 1 - z \, {\rm trace } (\tilde A_h ) + z^2 \det \tilde A_h \right)^2
		\end{align*}
		and hence from \eqref{curvas},
		\begin{align*}
			H_{\tilde \Sigma_h^n} &=-\,  {\rm trace} (\tilde A_h) \quad {\mbox {and}} \quad H_{\Sigma_h^n} =-\,  {\rm trace} ( A_h). 
		\end{align*}
		Thus the relation between  the mean curvatures of $\tilde \Sigma^n_h$ and $\Sigma^n_h$ 
		follows from comparing the trace of the matrix $\tilde A_h = -  \left(  \tilde Y_h^T \,  \tilde Y_h \right)^{-1}  \,  \tilde Y_h^T 
		\,  D_\omega \tilde \nu_h$ with the trace of the matrix $A_h = -  \left(  Y_h^T \,   Y_h \right)^{-1}  \,   Y_h^T 
		\,  D_\omega  \nu_h$.
		
		We recall that the parametrization $\tilde y_h$ of $\tilde \Sigma^n_h$ is given in terms of the parametrization  $ y_h$ of $ \Sigma^n_h$ through the map $X$ defined in \eqref{defX}, that is 
		$\tilde y_h (\omega) = ( X \circ y_h ) \, (\omega)$. Hence
		\begin{align*}
			\tilde Y_h = DX \, Y_h , \quad {\mbox {and}} \quad \tilde Y_h^T  \, \tilde Y_h &=Y_h^T  \, (D^TX \, DX )\, Y_h.
		\end{align*}
		A straightforward computation yields:
		\be \label{defQandM}
		\begin{aligned}
			D X &= Q+ {y_2 \over R} \, M , \quad {\mbox {where}} \quad   \\
			Q&= \left[\begin{matrix} 1&0&0\\
				0& \cos ({y_3 \over R}) & -  \sin ({y_3 \over R}) \\
				0& \sin ({y_3 \over R}) &  \cos ({y_3 \over R}) 
			\end{matrix} \right] , \quad M= \left[\begin{matrix} 0&0&0\\
				0& 0 & - \sin ({y_3 \over R}) \\
				0& 0 &  \cos ({y_3 \over R}) 
			\end{matrix} \right]
		\end{aligned}
		\ee
		and
		\be \label{matC}
		DX^T \, DX = I + \left( 2 {y_2 \over R} + {y_2^2 \over R^2} \right) \, C, \quad C= \left[\begin{matrix} 0&0&0\\
			0& 0 & 0 \\
			0& 0 &  1
		\end{matrix} \right].
		\ee
		An interesting characteristic of the matrix $M$ that we will use later  is that
		\be \label{matrixM}
		M \left[\begin{matrix} a\\
			b\\
			c
		\end{matrix} \right]= c \left[\begin{matrix} 0\\
			- \sin ({y_3 \over R}) \\
			\cos ({y_3 \over R}) 
		\end{matrix} \right], \quad M^T \left[\begin{matrix} a\\
			b\\
			c
		\end{matrix} \right]= \left[\begin{matrix} 0\\
			0 \\
			- \sin ({y_3 \over R}) b +\cos ({y_3 \over R}) c 
		\end{matrix} \right]
		\ee
		Thus 
		\begin{align*}
			\tilde Y_h^T \, \tilde Y_h &= Y_h^T \, \left( I + \left( 2 {y_2 \over R}  + {y_2^2 \over R^2} \right)  C \right)  \, Y_h \\
			&= Y_h^T \, Y_h + \left( 2 {y_2 \over R}  + {y_2^2 \over R^2} \right) \, Y_h^T  \, C \,Y_h
		\end{align*}
		and
		\be \label{tildeghtogh}
		\begin{aligned}
			\left(  \tilde Y_h^T \, \tilde Y_h\right)^{-1}  &= \, \left( I + \left( 2 {y_2 \over R}  + {y_2^2 \over R^2} \right)  \,  g_h^{-1}  \,  Y_h^T \, C \, Y_h \right)^{-1} \, g_h^{-1} ,  \\
			\quad & {\mbox {where}} \quad g_h =  \, Y_h^T \, Y_h 
		\end{aligned}
		\ee
		and $C$ is given by \eqref{matC}.
		This is telling that
		\begin{align*}
			\tilde A_h &=- \left( I + \left( 2 {y_2 \over R}  + {y_2^2 \over R^2} \right)  \, g_h^{-1} \,  Y_h^T \, C \, Y_h  \right)^{-1} \, g_h^{-1} \, Y_h^T \, (D \, X )^T D_\omega \tilde \nu_h .
		\end{align*}
		Now we look for an expression of $D_\omega \tilde \nu_h$ in terms of $D_\omega \nu_h$. From the relation 
		$\tilde Y_h =DX \, Y_h $, we use \eqref{defQandM} and \eqref{matrixM} to compute
		\be \label{Zh}
		\begin{aligned}
			\tilde Y_{h1}  \times \tilde Y_{h2}  &= Q Y_{h1}  \times Q Y_{h2}  + {y_2 \over R} Z_h\\
			& \quad \quad {\mbox {where}} \\
			Z_h &= M \, Y_{h1} \times Q \, Y_{h2}
			+ Q \, Y_{h1}  \times M \, Y_{h2} .
		\end{aligned}
		\ee
		and
		\begin{align*}
			\|    & \tilde Y_{h1} \times \tilde Y_{h2} \|^{-1} =  \|   Y_{h1} \times Y_{h2} \|^{-1} \times\\
			& \left( 1+ 2 {y_2 \over R} { Q (Y_{h1} \times Y_{h2} ) \cdot Z_h \over \| Y_{h1} \times Y_{h2} \|^2  }  +{y_2^2 \over R^2} {\| Z_h \|^2 \over \| Y_{h1} \times Y_{h2}  \|^2} \right)^{-{1\over 2}}.
		\end{align*}
		Hence we explicitely obtain that
		\be \label{tildenuhtonuh}
		\tilde \nu_h = Q \, \nu_h + \nu_h^{(1)}
		\ee
		where
		\begin{align*}
			\nu_h^{(1)} &= Q \, \nu_h \left[\left( 1+ 2 {y_2 \over R} { Q (\nu_h ) \cdot Z_h \over \| Y_{h1} \times Y_{h2} \|  }  +{y_2^2 \over R^2} {\| Z_h \|^2 \over \| Y_{h1} \times Y_{h2}  \|^2} \right)^{-{1\over 2}}  -1 \right] \\
			&+{y_2 \over R} {Z_h \over \|Y_{h1}  \times Y_{h2} \|}  \left( 1+ 2 {y_2 \over R} { Q (\nu_h ) \cdot Z_h \over \| Y_{h1} \times Y_{h2} \|  }  +{y_2^2 \over R^2} {\| Z_h \|^2 \over \| Y_{h1} \times Y_{h2}  \|^2} \right)^{-{1\over 2}}.
		\end{align*}
		Later we  will analyze $\nu_h^{(1)}$ in detail.
		For the moment we observe that
		\begin{align*}
			D_\omega \tilde \nu_h &= Q D_\omega \nu_h + ( D_\omega Q ) \nu_h + D_\omega \nu_h^{(1)}
		\end{align*}
		and since  $Q^T Q=I$, we obtain
		\be\label{tildeto}
		\begin{aligned}
			D^TX \, D_\omega \tilde \nu_h  &= (Q^T +{y_2 \over R} M^T ) \,  \left( Q D_\omega   \nu_h + ( D_\omega Q ) \nu_h + D_\omega \nu_h^{(1)} \right)  \\
			&= D_\omega   \nu_h  +V_h \\
			& \quad \quad {\mbox {where}}\\
			V_h&= Q^T  D_\omega \nu_h^{(1)} + Q^T \, (D_\omega Q) \nu_h + {y_2 \over R} M^T D_\omega \tilde \nu_h.
		\end{aligned}
		\ee
		We insert the computations  \eqref{tildeghtogh}, \eqref{tildeto} in the definition of $\tilde A_h$, see \eqref{tildegh}, 
		and we get
		\be \label{tildeAhtoAh}
		\tilde A_h = -  \, g_h^{-1} \, Y_h^T \, D_\omega \nu_h + A_{n,h}^{(1)}
		\ee
		where
		\begin{align*}
			A_{n,h}^{(1)} = &- g_h^{-1} \, Y_h^T\,  V_h 
			\\&- \left[ \left( I + \left( 2 {y_2 \over R}  + {y_2^2 \over R^2} \right)  \, g_h^{-1} \,  Y_h^T \, C \, Y_h  \right)^{-1}  - I \right]  \, g_h^{-1} \, Y_h^T \, (DX)^T \, D_\omega \nu_h ,
		\end{align*}
		with $V_h$ given in \eqref{tildeto}.
		
		\medskip
		Since $H_{\Sigma_h^n} = - \, {\rm trace} (A_h)$, with $A_h= -  \, g_h^{-1} \, Y_h^T \, D_\omega \nu_h$, we get 
		that
		\be \label{tildetracetotrace}
		\tilde H_{\tilde \Sigma^n_h } (y) =  H_{ \Sigma^n_h } (y) + H^{(1)}_{n,h} (y) , \quad {\mbox {where}} \quad H^{(1)}_{n,h} = - \, {\rm trace}  \, A_{n,h}^{(1)}.
		\ee
		
		\medskip
		The rest of the proof of Lemma \ref{prop31} is devoted to analyzing $H^{(1)}_{n,h} $.
		
		\medskip
		Let $A_n^{(1)}$ be the matrix 
		$A_{n,h}^{(1)}$ in \eqref{tildeAhtoAh} taking $h=0$.
		We claim that
		\be \label{asma}
		\begin{aligned}
			{\rm trace} \, A_n^{(1)} &=  - {\sin \theta \over R} \Biggl(    { (2-(f')^2 ) \, f \, f'' \over (1+ (f')^2 )^{5\over 2}} +{1 + 3 (f')^2 \over (1+ (f')^2 )^{3\over 2}} \Biggl)+ {a(y)\over n^2}
		\end{aligned}
		\ee
		where  $a(y)$ is a smooth function of $y$, uniformly bounded with its derivatives, as $n \to \infty$, and it satisfies the symmetries \eqref{simh}.
		Besides
		\be \label{H1}
		{\rm trace} \, A_{n,h}^{(1)} = {\rm trace} \, A_n^{(1)} + O\left( {1\over n} \right)  \, \ell (h , D_\omega h, D^2_\omega h ), \quad {\mbox {as }} n \to \infty
		\ee
		where $\ell $ is a smooth function in $\omega \in \Lambda$ (see \eqref{defLambda}), which is uniformly bounded, with its derivatives, as $n \to \infty$, and it satisfies the symmetries \eqref{simh}. Besides $\ell (0,0,0) = 0$, $\nabla \ell (0,0,0) \not= 0$.
		
		\medskip
		The rest of this proof is devoted to establish the validity of \eqref{asma} and \eqref{H1}.
		
		We start with \eqref{asma}.
		We  write
		\begin{align*}
			A_{n}^{(1)} &= A_{n1} + A_{n2} \\
			A_{n1}&=- g^{-1} \, Y^T\,  V 
			\\
			A_{n2}=&- \left[ \left( I + \left( 2 {y_2 \over R}  + {y_2^2 \over R^2} \right)  \, g^{-1} \,  Y^T \, C \, Y  \right)^{-1}  - I \right]  \, g^{-1} \, Y^T \, (DX)^T \, D_\omega  \nu .
		\end{align*}
		To simplify notation, we wrote $V$ for the function $V_h$ as in \eqref{tildeto} for $h=0$.
		
		\medskip
		We first compute the trace of  $ A_{n2}$. 
		Since
		\be \label{g}
		g= \left[\begin{matrix} f^2&0\\
			0 &  1 + (f')^2
		\end{matrix} \right], \quad g^{-1} = {1\over f^2 (1+ (f')^2) } \left[\begin{matrix} 1 + (f')^2 &0\\
			0 &  f^2
		\end{matrix} \right]
		\ee
		we use the explicit definition of the matrix  $C$ \eqref{matC} to obtain
		$$
		g^{-1} \, Y^T \, C \, Y = {1\over 1+ (f')^2} \left[\begin{matrix} 0 &0\\
			0 &  1
		\end{matrix} \right].
		$$
		Hence
		\begin{align*}
			\left( I + \left( 2 {y_2 \over R}  + {y_2^2 \over R^2} \right)  \, g^{-1} \,  Y^T \, C \, Y  \right)^{-1} &= \left[\begin{matrix} 1 &0\\
				0 & \left[ 1+ \left( 2 {y_2 \over R}  + {y_2^2 \over R^2} \right) {1\over 1+ (f')^2}  \right]^{-1}
			\end{matrix} \right]   ,
		\end{align*}
		and
		\begin{align*}
			\left( I + \left( 2 {y_2 \over R}  + {y_2^2 \over R^2} \right)  \, g^{-1} \,  Y^T \, C \, Y  \right)^{-1} &-I = \left[\begin{matrix} 0 &0\\
				0 & \alpha (y)
			\end{matrix} \right]   \\
			\alpha (y) &= -{ \left( 2 {y_2 \over R}  + {y_2^2 \over R^2} \right) {1\over 1+ (f')^2} \over  1+ \left( 2 {y_2 \over R}  + {y_2^2 \over R^2} \right) {1\over 1+ (f')^2}  }.
		\end{align*}
		Observe now
		\begin{align*}
			{\rm trace} (A_{n2} )&=-   {\rm trace} \left( \left[\begin{matrix} 0 &0\\
				0 & \alpha (y)
			\end{matrix} \right]  \, g^{-1} \, Y^T \, (DX)^T \, D_\omega  \nu \right).
		\end{align*}
		Letting 
		$$
		g^{-1} \, Y^T \, (DX)^T \, D_\omega  \nu = \left[\begin{matrix} a &b\\
			c & d
		\end{matrix} \right]
		$$
		we get that
		\begin{align*}
			{\rm trace} (A_{n2} )&=-   \alpha (y) \, d.
		\end{align*}
		Thus we proceed with the computation of the element $(2,2)$ of the matrix $ g^{-1} \, Y^T \, (DX)^T \, D_\omega  \nu.$ Using again \eqref{g} we obtain that
		$$
		d= 
		{ Y_2 \cdot \left( (DX)^T D_{y_3} \nu \right) \over 1+ (f')^2},$$
		where $Y_2$ is given by \eqref{tangetv}. A direct computation gives
		\be \label{dnormal}
		\begin{aligned}
			D_1 \nu &= {1\over \sqrt{1+ (f')^2} } \left[ \begin{matrix} -\sin \theta \\ \cos \theta  \\0  \end{matrix}\right] , \quad \\
			D_2 \nu &= \left({1\over \sqrt{1+ (f')^2} }\right)' \left[\begin{matrix} \cos \theta  \\ \sin  \theta  \\-f'  \end{matrix}\right] - {f'' \over \sqrt{1+ (f')^2}} \left[\begin{matrix} 0 \\ 0  \\1\end{matrix}\right]
		\end{aligned}
		\ee
		as consequence of \eqref{normal}, and 
		\begin{align*}
			Y_2 \cdot \left( (DX)^T D_{y_3} \nu \right)&= -{f'' \over \sqrt{1+ (f')^2}}    \\
			&+ (\cos {y_3 \over R} -1 ) \left( {f' \over (\sqrt{1+ (f')^2})'} \sin^2 \theta - ({f' \over \sqrt{1+ (f')^2}})' (1+ {y_2 \over R}) \right)\\
			&- \sin {y_3 \over R} \, \sin \theta \, \left( f'  ({f' \over \sqrt{1+ (f')^2}})'+  {1 \over (\sqrt{1+ (f')^2})'} (1+ {y_2 \over R}) \right).
		\end{align*}
		Combining these computations with the previous ones we conclude that
		\be \label{tracea2}
		{\rm trace} (A_{n2} )= - 2 {y_2 \over R} \,  {f'' \over (1+ (f')^2 )^{5\over 2}} + {a_2(y)\over n^2} \quad {\mbox {as }} n \to \infty,
		\ee
		where $a_2$ is a smooth function, uniformly bounded as $n \to \infty$ which satisfies the symmetry assumptions \eqref{simh}.

		\medskip
		We now treat the trace of $A_{n1}$. From \eqref{g} we observe that
		$$
		{\rm trace} \left (g^{-1} \left[\begin{matrix} a &b\\
			c &  d
		\end{matrix} \right] \right) =  {1\over f^2 (1+ (f')^2) } \left( (1+ (f')^2) \, a + f^2 \, d \right).
		$$
		This suggests that 
		\be \label{traceAn1}
		\begin{aligned}
			{\rm trace} \, A_{n1} &= -{1\over f^2 (1+ (f')^2) } \left( (1+ (f')^2) \, a + f^2 \, d \right)\\
			a&= Y_1 \cdot V_1 , \quad d= Y_2 \cdot V_2
		\end{aligned}
		\ee
		where $V_1$ and $V_2$ are the first and second columns of the $3\times 2$ matrix 
		$$
		V=  Q^T  D_\omega \nu^{(1)} + Q^T \, (D_\omega Q) \nu + {y_2 \over R} M^T (  D_\omega \nu + ( D_\omega Q ) \nu + D_\omega \nu^{(1)}) .
		$$
		In order to find $V_1$ and $V_2$, we observe that 
		$$
		{y_2 \over R} M^T (   ( D_\omega Q ) \nu + D_\omega \nu^{(1)})  = O({1\over n^2}) \quad {\mbox {as }} n \to \infty.
		$$
		We thus turn to analyze 
		$Q^T  D_\omega \nu^{(1)} + {y_2 \over R} M^T   D_\omega \nu + Q^T \, (D_\omega Q) \nu $. Calling $V_j^\#$, $V_j^*$ and $V_j^{**}$ the columns respectively of $Q^T  D_\omega \nu^{(1)}$,
		${y_2 \over R} M^T   D_\omega \nu $ and $ Q^T \, (D_\omega Q) \nu_h$, we have that
		$$
		V_j = V_j^\# + V_j^* +V_j^{**}+O({1\over n^2}) \quad j=1,2, \quad {\mbox {as }} n \to \infty.
		$$
		We start from ${y_2 \over R} M^T   D_\omega \nu$.
		From \eqref{dnormal} 
		we get the columns $V_j^*$ of ${y_2 \over R} M^T   D_\omega \nu$ are
		\be \label{Cj*}
		V_1^* = O({1\over n^2}) , \quad V_2^* = - {y_2 \over R} \left( {f' \over \sqrt{1+ (f')^2}} \right)' \, \left[\begin{matrix} 0 \\ 0  \\1\end{matrix}\right] + O({1\over n^2}) \quad {\mbox {as }} n \to \infty.
		\ee
		Next we consider $ Q^T \, (D_\omega Q) \nu$. From \eqref{defQandM} we easily check that
		\be \label{Cj**}
		V_1^{**} = O({1\over n^2}) , \quad V_2^{**} =  {1\over R} {1\over \sqrt{1+ (f')^2}}  \, \left[\begin{matrix} 0 \\ f'  \\ \sin \theta \end{matrix}\right] + O({1\over n^2}) \quad {\mbox {as }} n \to \infty.
		\ee
		We turn now to $Q^T  D_\omega \nu^{(1)}$ and its columns $V_j^\#$.
		From \eqref{tildenuhtonuh} we get that
		$$
		\nu^{(1)} = {y_2 \over R} {1\over \| Y_1 \times Y_2 \|} \left[ Z - (\nu \cdot Z) \nu \right] + O({1\over n^2}) \quad {\mbox {as }} n \to \infty
		$$
		where $Z$ is the vector field $Z_h$ introduced in \eqref{Zh} taking $h=0$. Using \eqref{matrixM} and \eqref{tangetv}, a close inspection to $Z$ gives 
		$$
		Z= Y_1 \times \left[ \begin{matrix} 0 \\ 0 \\1  \end{matrix}\right] + O({1\over n}) = f \left[ \begin{matrix} \cos \theta  \\ \sin \theta  \\0  \end{matrix}\right]  + O({1\over n}), \quad {\mbox {as }} n \to \infty.
		$$
		Since $\| Y_1 \times Y_2 \| = f \sqrt{1+ (f')^2}$, $y_2 = f \sin \theta$ and $ \nu \cdot Z = {f \over \sqrt{1+ (f')^"}} + O({1\over n})$, we get
		\begin{align*}
			\nu^{(1)} &= { \sin \theta   \over R} v + O({1\over n^2}), \quad v= {f \over \sqrt{1+ (f')^2} }\left(    \left[ \begin{matrix} \cos \theta  \\ \sin \theta  \\0  \end{matrix}\right]  -{\nu \over \sqrt{1+ (f')^2} }    \right) .
		\end{align*}
		We get that the columns of $Q^T  D_\omega \nu^{(1)}$ are
		\be \label{Cj}
		\begin{aligned}
			V_1^\# := Q^T \, D_1 \nu^{(1)} &= { \cos \theta \over R} v + {\sin \theta    \over R} D_1 v + O({1\over n^2})\\
			V_2^\# := Q^T \, D_2 \nu^{(1)} &=  {\sin \theta \over R} D_2 v +  O({1\over n^2})
		\end{aligned}
		\ee
		as $n \to \infty$.
		We recall that $D_1$ stands for $D_\theta$, and $D_2$ for $D_{x_3} = \, '$.
		From \eqref{Cj*}-\eqref{Cj**}-\eqref{Cj} and using the fact that $Y_1 \cdot v=0$ we obtain that
		\begin{align*}
			Y_1 \cdot V_1 &= {\sin \theta   \over R} \, Y_1 \cdot D_1 v + O({1\over n^2})\\
			Y_2 \cdot V_2 &= - {y_2 \over R} \left( {f' \over \sqrt{1+ (f')^2}} \right)' + {\sin \theta \over R} \, \sqrt{1+ (f')^2}  \\
			& + {\sin \theta  \over R} Y_2 \cdot D_2 v + O({1\over n^2})
		\end{align*}
		as $n \to \infty$.
		It is straightforward to check that
		\begin{align*}
			D_1 v &= {f \, (f')^2 \over (1+ (f')^2 )^{3\over 2}  }      \left[ \begin{matrix} - \sin  \theta  \\ \cos \theta  \\0  \end{matrix}\right]   \\
			D_2 v &= \left( {f \over \sqrt{1+ (f')^2} }  \right)' \left[ \begin{matrix} \cos  \theta  \\ \sin \theta  \\0  \end{matrix}\right]
			- \left( {f \over 1+ (f')^2}   \right)' \nu 
			-   {f \over 1+ (f')^2}   D_2 \nu 
			.
		\end{align*}
		With these computations at hand and using that $y_2 = f \, \sin \theta$, we conclude that
		\begin{align*}
			Y_1 \cdot V_1 &= {\sin \theta  \over R} \,{ f^2 \, (f')^2 \over (1+ (f')^2 )^{3\over 2}  }   \, + O({1\over n^2})\\
			Y_2 \cdot V_2 &= - {y_2 \over R} \left( {f' \over \sqrt{1+ (f')^2}} \right)'  + {\sin \theta \over R} \, \sqrt{1+ (f')^2} \\
			&+ {\sin \theta  \over R} \, f' \, \left( {f \over \sqrt{1+ (f')^2} }  \right)' + {\sin \theta \over R} {f \, f'' \over (1+ (f')^2 )^{3\over 2}}  + O({1\over n^2})\\
			&= {\sin \theta \over R} \Biggl[ - { (f')^2 \, f \, f'' \over (1+ (f')^2 )^{3\over 2}} +{1 + 2 (f')^2 \over (1+ (f')^2 )^{1\over 2}} \Biggl]  + O({1\over n^2})
		\end{align*}
		as $n \to \infty$.
		From \eqref{traceAn1} we conclude that
		\be\label{uurca}
		{\rm trace} \, A_{n1} =
		-{\sin \theta \over R} \, \Biggl(    -{  (f')^2 \, f \, f'' \over (1+ (f')^2 )^{5\over 2}} +{1 + 3 (f')^2 \over (1+ (f')^2 )^{3\over 2}} \Biggl) + {a_1 (y)\over n^2}
		\ee
		where  $a_1(y)$ is a smooth function of $y$, uniformly bounded with its derivatives, as $n \to \infty$.
		Combining this result with \eqref{tracea2}
		we get
		\begin{align*}
			{\rm trace} \, A_n^{(1)} &=  - {\sin \theta \over R} \Biggl(    { (2-(f')^2 ) \, f \, f'' \over (1+ (f')^2 )^{5\over 2}} +{1 + 3 (f')^2 \over (1+ (f')^2 )^{3\over 2}} \Biggl)+ {a (y)\over n^2},
		\end{align*}
		where  $a(y)$ is a smooth function of $y$, uniformly bounded with its derivatives, as $n \to \infty$.
		To complete the proof of \eqref{asma} we now show that the function $a_1(y)$ in \eqref{uurca} satisfies the symmetries \eqref{simh}. We see	from \eqref{uurca} that this is consequence of the fact that
		$$
		y \to {\rm trace} \, A_{n1} (y)
		$$
		satisfies \eqref{simh}. Going back to formula \eqref{traceAn1}, we shall prove that the functions in \eqref{traceAn1}
		$$
		a = a(\theta, y_3), \quad d= d(\theta, y_3)
		$$
		are even in $\theta$, even in $y_3$ and periodic in $y_3$ of period $T$. The periodicity in $y_3$ follows from the fact that all functions involved in the definition of $a$ and $d$ are $T$-periodic in $y_3$.
		The remaining symmetries can be verified checking that the evenness of the components of the columns $V_1$ and $V_2$ in \eqref{traceAn1} have the form
		$$
		V_1= \left[ \begin{matrix} {\mbox { even in }} y_3 , {\mbox { even in }} \theta \\
			{\mbox { even in }} y_3, {\mbox { odd in }} \theta \\  {\mbox {any function}} \end{matrix} \right] , \quad V_2= \left[ \begin{matrix} {\mbox { odd in }} y_3 , {\mbox { odd in }} \theta \\
			{\mbox { odd in }} y_3, {\mbox { even in }} \theta \\ {\mbox { even in }} y_3, {\mbox { even in }} \theta  \end{matrix} \right].
		$$
		We check this for $V_1$, which is given by
		$$
		V_1= Q^T D_\theta \nu^{(1)} + {y_2 \over R} M^T D_\theta \tilde \nu.
		$$
		Since the first two component of $M^T D_\theta \tilde \nu$ are zero (because of \eqref{matrixM}), we just need to check $Q^T D_\theta \nu^{(1)} $. We have
		$$
		\nu^{(1)} =  \left[ \begin{matrix} {\mbox { even in }} y_3 , {\mbox { odd in }} \theta \\
			{\mbox { odd in }} y_3, {\mbox { even in }} \theta \\  {\mbox { even in }} y_3 , {\mbox { even in }} \theta \end{matrix} \right], \quad D_\theta \nu^{(1)} = \left[ \begin{matrix} {\mbox { even in }} y_3 , {\mbox { even in }} \theta \\
			{\mbox { odd in }} y_3, {\mbox { odd in }} \theta \\  {\mbox { even in }} y_3 , {\mbox { odd in }} \theta \end{matrix} \right]
		$$
		and
		$$
		Q^T  D_\theta \nu^{(1)} = \left[ \begin{matrix} {\mbox { even in }} y_3 , {\mbox { even in }} \theta \\
			{\mbox { even in }} y_3, {\mbox { odd in }} \theta \\  {\mbox { odd in }} y_3 , {\mbox { odd in }} \theta \end{matrix} \right],
		$$
		which is what expected.
		
		For $V_2$, we have
		$$
		V_2 = Q^T (D_{y_3} Q) [\nu] +Q^T D_{y_3} \nu^{(1)} + {y_2 \over R} M^T D_{y_3} \tilde \nu.
		$$
		The desired form of $V_2$ follows from the fact that $Q\nu$, $\tilde \nu$ and $\nu^{(1)}$ have the same structure
		$$
		\left[ \begin{matrix} {\mbox { even in }} y_3 , {\mbox { odd in }} \theta \\
			{\mbox { odd in }} y_3, {\mbox { even in }} \theta \\  {\mbox { even in }} y_3 , {\mbox { even in }} \theta \end{matrix} \right].
		$$
		This completes the proof of \eqref{asma}.

		\medskip We now turn to \eqref{H1}.
		We observe that
		\begin{align*}
			A_{n,h}^{(1)} = & A_n^{(1)} - B_1 - B_2 \\
			B_1&=  g_h^{-1} \, Y_h^T\,  V_h - g^{-1} \, Y^T\,  V 
			\\
			B_2&= \left[ \left( I + \left( 2 {y_2 \over R}  + {y_2^2 \over R^2} \right)  \, g_h^{-1} \,  Y_h^T \, C \, Y_h  \right)^{-1}  - I \right]  \, g_h^{-1} \, Y_h^T \, (DX)^T \, D_\omega \nu_h \\
			&- \left[ \left( I + \left( 2 {y_2 \over R}  + {y_2^2 \over R^2} \right)  \, g^{-1} \,  Y^T \, C \, Y  \right)^{-1}  - I \right]  \, g^{-1} \, Y^T \, (DX)^T \, D_\omega \nu .
		\end{align*}
		We further split
		\begin{align*}
			B_1&= g_h^{-1} \, Y_h^T\,  (V_h - V) \,  + \, g_h^{-1} \, (Y_h^T - Y^T) \,  V \, + \, (g_h^{-1} - g^{-1} ) \, Y^T\,  V 
		\end{align*}
		and
		\begin{align*}
			B_2 &= \left[ \left( I + \left( 2 {y_2 \over R}  + {y_2^2 \over R^2} \right)  \, g_h^{-1} \,  Y_h^T \, C \, Y_h  \right)^{-1}  - I \right]  \\
			&\times   \, \left( g_h^{-1} \, Y_h^T \, (DX)^T \, D_\omega \nu_h -  g^{-1} \, Y^T \, (DX)^T \, D_\omega \nu \right) \\
			&+ \left[ \left( I + \left( 2 {y_2 \over R}  + {y_2^2 \over R^2} \right)  \, g_h^{-1} \,  Y_h^T \, C \, Y_h  \right)^{-1} - \left( I + \left( 2 {y_2 \over R}  + {y_2^2 \over R^2} \right)  \, g^{-1} \,  Y^T \, C \, Y  \right)^{-1}   \right]\\
			& \times \, g^{-1} \, Y^T \, (DX)^T \, D_\omega \nu .
		\end{align*}
		We claim that
		\be\label{1}
		\begin{aligned}
			{\rm trace} (B_1 + B_2)&= {1\over n} \,  \ell (h, D_\omega h, D_\omega^2 h)  , \quad {\mbox {as}} \quad n \to \infty  
		\end{aligned}
		\ee
		where $\ell $ is a smooth function in $\omega \in \Lambda$ (see \eqref{defLambda}), which is uniformly bounded, together with its derivatives, as $n \to \infty$. Besides $\ell (0,0,0) = 0$, $\nabla \ell (0,0,0) \not= 0$.
		
		This fact follows is consequence of 
		\begin{align*}
			Y_h &= Y + h D_\omega \nu + \nu^T D_\omega h,\\
			g_h &= g + h \, \left( Y^T D_\omega \nu + D_\omega \nu^T Y \right) + Y^T \nu^T D_\omega h + D_\omega h^T \nu Y\\ 
			&+ h^2 D_\omega \nu^T D_\omega \nu + h D_\omega \nu^T \nu^T D_\omega h + h D_\omega h^T \nu D_\omega \nu + D_\omega h^T \nu \nu^T D_\omega h, \\
			&\quad {\mbox {and}} \\
			\nu_h &= \nu + \vec \ell_1 (h, D_\omega h)
		\end{align*}
		where $\vec \ell_1$ is an explicit vector, which is a smooth function of $\omega$,  such that
		$\vec \ell_1 (0, 0) =0 $, $D \vec \ell_1 (0, 0) \not=0.$
		Using that $2\, \pi \, R = n \, T$, from \eqref{tildenuhtonuh} and \eqref{tildeto} we obtain
		\begin{align*}
			\nu_h^{(1)} &= {1\over n} \, \nu^{(1)}_* + {1\over n} \vec \ell_1 (h, D_\omega h)\\
			V_h &= {1\over n} \, V_* + {1\over n} \vec \ell_2 (h, D_\omega h, D_\omega^2 h)\\
		\end{align*}
		with $\nu^{(1)}_*$ and $V_*$ smooth functions of $\omega$, which are uniformly bounded (with their derivatives) as $n \to \infty$; $\vec \ell_1$ has the same properties as before, and $\vec \ell_2$ is an explicit vector, which is a smooth function of $\omega$,  such that
		$\vec \ell_2 (0, 0,0) =0 $, $D \vec \ell_2 (0, 0,0) \not=0.$
		
		Combining these expansions  with \eqref{H1}, we obtain \eqref{expH*}. We conclude with the remark that
		$y \to \ell (h, D_\omega h , D_\omega^2 h) (y)$ in \eqref{1} satisfies \eqref{simh}, as consequence of Remark \ref{rem1} and of \eqref{asma}-\eqref{tildetracetotrace}.
		
	\end{proof}

	\begin{proof}[Proof of Lemma \ref{prop32}] \ \ 
		Recall that
		\be \label{Ah1}
		H_{\Sigma_h^n} = -\,  \hbox{trace}\, (A_h) , \quad {\mbox {where}} \quad  A_h= -g_h^{-1} Y_h^T B_h . 
		\ee
		We refer to \eqref{defg} for $g_h$, to \eqref{defY} for $Y_h$ and \eqref{defB} for $B_h$.

		Recall the parametrization of  surface $\Sigma_h^n$ given by 
		$$  y_h(w) = y(w) + h(w) \nu(w), \quad  $$ 
		and 
		$$
		Y_h =  D_w (y  + h \nu ) = Y+ B h  +   \nu D_\omega h, \quad B= D_\omega \nu .  
		$$
		The metric $g_h$ on $\Sigma_h^n$ is thus given by 
		\be \begin{aligned} g_h =   &  ( Y+ B h  +   \nu D_\omega h  ) ^T  (Y+ Bh + \nu D_\omega h )\\ 
			&=  Y^TY  +   h ( B^TY  + Y^TB ) +  h^2 B^TB\\
			&+
			Y^T \nu D_\omega h  +   D_\omega h^T \nu^T Y  +  D_\omega h^T \nu^T \nu D_\omega h   \end{aligned}\label{cc}\ee
		Recall that $\nu^T \nu =1$. The relations $\nu^T Y_i =0 = \nu^T Y_j $ imply   
		$$\nu^T Y = Y^T\nu =0,$$ and after 
		differentiation,   
		$$  \nu^T  \pp_j Y_i   +   B_j^T Y_i= 0=\nu^T\pp_i Y_j +  B_i^T Y_j  .  $$
		Since  $\pp_j Y_i= \pp_i Y_j$, 
		we get 
		$
		(B^T Y)_{ij}  = B_j^T Y_i =  B_i^T Y_j  = (B^T Y)_{ji}.
		$
		Thus the matrix  $B^T Y$ is symmetric and therefore 
		$$  B^TY  = Y^TB . $$ 
		We observe that the vectors $B_i$ belong to the tangent space, which is spanned by the $Y_i's$. Hence for a column vector $\alpha_j $  we have 
		$
		B_i  =    Y \alpha_{i} . 
		$
		Hence $ \alpha_i =    (Y^T Y)^{-1}   Y^T B_i$. It follows that 
		\be 
		B  =     Y (Y^T Y)^{-1}   Y^T B. 
		\label{B} \ee
		Using these relations in \eqref{cc} we obtain 
		$$\begin{aligned}
			g_h  =&  Y^TY  +   2h Y^TB  + h^2B^TB+     D_\omega h^T D_\omega h \\ =&  
			g(0) \left [   I  +  2h (Y^TY)^{-1}  Y^TB    +   h^2   (Y^TY)^{-1}B^TB \right ]+     D_\omega h^T D_\omega h.
		\end{aligned}   $$
		We observe that 
		$$\begin{aligned}
			[ (Y^TY)^{-1}  Y^TB]^2  = &  (Y^TY)^{-1}  Y^TB (Y^TY)^{-1}  Y^TB\\ =&
			(Y^TY)^{-1}  B^TY (Y^TY)^{-1}  Y^TB = 
			(Y^TY)^{-1}  B^TB, 
		\end{aligned}$$
		and hence 
		\begin{equation}\label{defgh}
			g_h =   g(0) [ I- hA ]^2 \left( I   + \left( g(0) [ I- hA ]^2 \right)^{-1}   D_\omega h^T D_\omega h \right) .
		\end{equation}
		where $A$ is the matrix
		$$
		A =  - [\:Y^TY]^{-1} Y^T B.
		$$
		Next we compute an expansion in $h$ for the normal vector to $\Sigma_h^n$, $$\nu_h =   \frac {Y_{1h} \times Y_{2h}} {|Y_{1h} \times Y_{2h}|}$$   as follows. It holds  
		$$
		\begin{aligned}
			Y_{h1} \times Y_{h2}  =&(Y_1  + hB_1 + \nu  D_1 h) \times (Y_2 + hB_2 + \nu D_2 h )  \\ = &    Y_1 \times Y_2+  v (h)   + \ q (h , D_\omega h )
		\end{aligned}
		$$
		where $B_i$, $i=1,2$, are the columns of $B$, $D_i$ stands for $D_{\omega_i}$, $i=1,2$,
		$$ \begin{aligned}
			v (h) = &  D_2 h\, ( Y_1\times \nu)   + D_1 h \, (\nu\times Y_2)   + h\, [ D_2h\, (B_1\times \nu) + 
			D_1h  (\nu   \times B_2) ]   \\ =&      h\, (B_1 \times Y_2 + Y_1\times B_2 ) +  D_2h  \, (Y_1\times \nu) +  D_1 h \, (\nu\times Y_2)    
		\end{aligned}
		$$
		and $q (h, D_\omega h) $ a smooth function, with $q(0,0)= Dq (0,0)=0$ and uniformly bounded as $n \to \infty$.
		Then
		$$\begin{aligned}
			\nu_h  =& \nu  - \nu \frac {\nu\cdot v (h) } { |Y_1\times Y_2 |  } + 
			\frac { v (h) } { |Y_1\times Y_2 |  } + \hat \nu_h\\     = &  \nu  
			+ \frac 1  { |Y_1\times Y_2 |  } Y (Y^TY)^{-1} Y^T  v (h)  + \hat \nu_h .\end{aligned}
		$$
		Hence  we have 
		\be \label{Bh}
		\begin{aligned}
			B_h = D_\omega  \nu_h  =&  D_\omega \nu  +   D_\omega  \left (    \frac 1{  |Y_1\times Y_2 | } Y \, (Y^TY)^{-1} Y^T  v (h)     \right) + D_\omega \hat\nu_h\\  =&  B  + D_\omega \left (   
			\frac 1{  |Y_1\times Y_2 |} ( Y g^{-1}Y^T v (h) )\right)  + q(h, D_\omega h , D_\omega^2 h) 
		\end{aligned}
		\ee
		where $q (h, D_\omega h , D_\omega^2 h ) $ a smooth function, with $q(0,0, 0)= Dq (0,0,0)=0$ and uniformly bounded as $n \to \infty$.
		Inserting \eqref{defgh}-\eqref{Bh} into \eqref{Ah1}, we get
		\begin{align*}
			A_h &=  A  - 2hAg^{-1}Y^TB - g^{-1} B^TB   h  \\
			&- 
			g^{-1} Y^T      D 
			\left (    \frac 1   {  |Y_1\times Y_2 |}  Y(Y^TY)^{-1} Y^T  \, v  (h)     \right)
			+ q (h, D_\omega h, D_\omega^2 h)   
		\end{align*}
		for a quadratic term $q$ with the same properties as the one in \eqref{Bh}.
		Observe now that 
		$$ 
		-A^2 =  A g^{-1}Y^TB, \quad   g^{-1}B^T B  = A^2,   
		$$
		hence 
		\begin{equation}\label {Ah}
			\begin{aligned}
				A_h &=  A + hA^2  - g^{-1} Y^T  D 
				\left (    \frac 1   {  |Y_1\times Y_2 |}  Y(Y^TY)^{-1} Y^T  v (h)     \right)\\
				&+ q(h, D_\omega h , D_\omega^2 h) .
			\end{aligned}
		\end{equation}
		On the other hand, we use the vector identity 
		$$ a\cdot (b\times c) =  \det [ a\, b\, c]  $$  
		to get 
		$$
		\begin{aligned}
			\nu\times Y_2 \cdot Y_1  = & Y_1 \times \nu \cdot Y_2  = -\sqrt{\det g }  \\\nu\times Y_2 \cdot Y_2 =& \nu \times Y_1 \cdot Y_1  = 0 ,\\
			(B_1\times Y_2 + Y_1 \times B_2) \cdot Y_1 =&  (B_1\times Y_2 + Y_1 \times B_2)\cdot Y_2 = 0,
		\end{aligned}
		$$
		and we get the simple relation  
		$$
		Y^T v (h) = - \sqrt{\det g }D^T h
		$$
		which simplifies formula \eqref{Ah} to 
		$$\begin{aligned}
			A_h =& A + hA^2  + g^{-1} Y^T D_\omega( Y g^{-1} D_\omega^T h) +  q(h, D_\omega h , D_\omega^2 h) . \\  
		\end{aligned}
		$$
		From here we obtain
		\begin{equation}\label{asma1}
			\begin{aligned}
				{\rm trace} \, A_h &=  {\rm trace} \, A + h |A|^2 + {\rm trace} \,    \left( g^{-1} Y^T D( Y g^{-1} D^T h) \right) \\
				&+ q(h, D_\omega h , D_\omega^2 h ) .
			\end{aligned}
		\end{equation}
		The function $q$ is a smooth function of $y$, uniformly bounded with its derivatives as $n \to \infty$. It depends of $h$, $D_\omega h$ and $D_\omega^2 h$ with
		$$
		q(0,0,0)=0, \quad D q (0,0,0) = 0, \quad D^2 q(0,0,0) \not= 0. 
		$$
		We now focus on the term $g^{-1} Y^T D_\omega ( Y g^{-1} D_\omega^T h) $. We drop $\omega$ from the notation $D_\omega$. We write 
		\begin{align*}
			g^{-1} Y^T D( Y g^{-1} D^T h)  &=g^{-1} \left[ D_1 D^T h \, | \, D_2 D^T h \right] \\
			&+ \left[ D_1 g^{-1}    \ |\ D_2 g^{-1}   \right] D^Th \\
			&+ 
			\left[  g^{-1} Y^T D_1Y g^{-1}D^Th  \ |\  g^{-1} 
			Y^T D_2Y  g^{-1} D^Th\right]
		\end{align*}
		A direct computation gives
		\begin{align*}
			{\rm trace } \left(  g^{-1} \left[ D_1 D^T h \, | \, D_2 D^T h \right] \right) 
			&= {\rm trace } \left( \left[\begin{matrix} g^{11} & g^{12} \\ g^{12} & g^{22} \end{matrix}\right]
			\left[\begin{matrix} D_{11} h & D_{12} h  \\ D_{12} h  & D_{22} h \end{matrix}\right] \right) \\
			&= g^{ij} D_{ij} h,
		\end{align*}
		\begin{align*}
			{\rm trace }& \left( \left[ D_1 g^{-1}   D^Th  \ |\ D_2 g^{-1}  D^Th  \right] \right) 
			\\
			&= {\rm trace } \left( \left[\begin{matrix} D_1 g^{11} & D_1 g^{12} \\ D_1 g^{12} & D_1 g^{22} \end{matrix}\right]
			\left[\begin{matrix} D_{1} h  \\ D_2 h \end{matrix} \right]  \, | \,  \left[\begin{matrix} D_2g^{11} & D_2 g^{12} \\ D_2 g^{12} & D_2 g^{22} \end{matrix}\right]
			\left[\begin{matrix} D_{1} h  \\ D_2 h \end{matrix} \right]  \right) \\
			&= D_i g^{ij} D_j h
		\end{align*}
		and
		\be 
		\begin{aligned}
			{\rm trace }\, & [g^{-1} Y^T D( Y g^{-1} D^T h) ]   \\
			&= {\rm trace }\, \left[  g^{-1} Y^T D_1Y g^{-1}D^Th  \ |\  g^{-1} 
			Y^T D_2Y  g^{-1} D^Th\right]\\ &=
			[g^{-1} Y^T D_1Y g^{-1}D^Th]_1 +  [g^{-1} 
			Y^T D_2Y  g^{-1} D^Th]_2\\& = 
			[ (g^{-1}Y^TD_1Y)_{11}  + (g^{-1}Y^TD_2Y)_{21}    ] \, [ g^{11} D_1 h + g^{12} D_2 h   ]\\  &  +
			[ (g^{-1}Y^TD_1Y)_{12}  + (g^{-1}Y^TD_2Y)_{22}    ]\,  [ g^{21} D_1 h + g^{22} D_2 h   ]. 
		\end{aligned} \label{iii}\ee 
		Since 
		\begin{align*}
			D_1Y_2 =& D_2Y_1,\quad 
			D_1Y_1\cdot Y_2 = D_1 g_{12} - \frac 12 D_2  g_{11}, \quad\\
			&
			D_2Y_2\cdot Y_1 = D_2 g_{12} - \frac 12 D_1  g_{22}, 
		\end{align*}
		we obtain 
		$$
		Y^T D_1Y = \left [ \begin{matrix}
			\frac 12 D_1 g_{11}    &   \frac 12 D_2 g_{11} \\ D_1 g_{12} - \frac 12 D_2 g_{11}   &   \frac 12 D_1 g_{22}  
		\end{matrix}   \right ]
		$$
		\medskip
		$$
		Y^T D_2Y = \left [ \begin{matrix}
			\frac 12 D_2 g_{11}    &   D_2 g_{12} - \frac 12 D_1 g_{22} \\ \frac 12 D_1 g_{22}   &   \frac 12 D_2 g_{22}  
		\end{matrix}   \right ].
		$$
		We have 
		\begin{align*}
			(g^{-1}Y^TD_1Y)_{11}  =& \frac 1{2\det g} [ g_{22}  D_1 g_{11}  + g_{12} D_2 g_{11} - 2 g_{12} D_1 g_{12}     ] \\ 
			(g^{-1}Y^TD_1Y)_{21} =&  \frac 1{2\det g} [ g_{11}  D_1 g_{22}   - g_{12} D_2 g_{11}   ]\\ 
			(g^{-1} Y^T D_1Y)_{12} = & \frac 1{2\det g} [g_{22}D_2 g_{11} - g_{12} D_1 g_{22}    ]\\ 
			(g^{-1} Y^T D_2Y)_{22} = & \frac 1{2\det g} [g_{11}D_2 g_{22} + g_{12} D_1 g_{22} -  2 g_{12} D_2 g_{12}  ].
		\end{align*}
		From the above formulas we get
		\begin{align*}
			{\rm trace }\,&  [g^{-1} Y^T D( Y g^{-1} D^T h) ]   \\
			&=  \frac 1{2\det g} \left [ 
			g_{22} D_1 g_{11}   + g_{11} D_1 g_{22}  - 2 g_{12} D_1 g_{12}    \right ] \left [    g^{11} D_1 h  +  g^{12} D_2 h   \right ]   \\
			& +   \frac 1{2\det g} \left [ 
			g_{22} D_2 g_{11}   + g_{11} D_2 g_{22}  - 2 g_{12} D_2 g_{12}    \right ]\, \left [    g^{21} D_1 h  +  g^{22} D_2 h   \right ] \\ 
			= &   \frac 12 \frac {D_1 \det g}{\det g}   \left [    g^{11} D_1 h  +  g^{12} D_2 h   \right ] 
			\ + \    \frac 12 \frac {D_2 \det g }{\det g} \,  \,  \left [    g^{21} D_1 h  +  g^{22} D_2 h   \right ].
		\end{align*}
		Using these expressions in formula \eqref{iii}, we obtain from \eqref{Ah} and \eqref{asma1},
		$$
		H_{\Sigma_h^n }  = H_{\Sigma^n }  - J_\Sigma [h]  + q(h, D_\omega h , D_\omega^2 h) 
		$$
		where 
		$$
		J_{\Sigma^n} [h]  =  \frac 1{2 \sqrt{\det g}} \pp_j (  g^{ij} \sqrt{\det g}  \pp_j h) + |A|^2 h   
		$$
		and $q(h, D_\omega h , D_\omega^2 h)$ has the same properties as in \eqref{asma1}. 
		We conclude with the remark that
		$y \to q (h, D_\omega h , D_\omega^2 h) (y)$ in satisfies \eqref{simh}, as consequence of Remark \ref{rem1}, the explicit expression of the Jacobi operator $J$ and the fact that $H_{\Sigma^n} (y) =2.$  
		This completes the proof of Lemma \ref{prop32}.

	\end{proof}

	\section{Expansion of the interaction term}
	\label{sec4}

	Let us consider a small smooth function $h(y)$ defined on $\Sigma$   and the  {\em normal graph} 
	$\Sigma^n_h$ defined as in \eqref{truncated-modified}. We also assume that the 
	perturbation $h: \Sigma \to \R$ satisfies the symmetries \eqref{simh}.

	For a point $\tilde y_h \in \tilde \Sigma^n_h = X (\Sigma^n_h)$ we want to compute the interaction term 
	$$
	\int_{\tilde \Omega^n_h} {d \tilde x \over |\tilde y_h - \tilde x| } .
	$$
	For $y \in \Sigma^n$, we write 
	\be \label{N2}
	\tilde N_{\Sigma^n_h} (y ) =
	\int_{\tilde \Omega^n_h}\dfrac{d \tilde x}{|X(y_h) -\tilde x|},\quad \tilde y_h  = X (y_h) \in \tilde \Sigma^n_h, \quad y_h \in \Sigma^n_h,
	\ee
	where
	$$
	y_h = y + h (y) \nu (y)
	$$
	and $\nu(y)$ is the  
	unit normal vector at $y\in \Sigma^n$ as in \eqref{normal}.
	
	Thanks to symmetries, it is enough to know $  \tilde N_{\Sigma^n_h} (y ) $  at points $y$ in $\Sigma_0$, the first piece  of the truncated Delaunay $\Sigma^n$ as defined in \eqref{sigma0}, that is
	$$
	\Sigma_0= \{\,  y \in \Sigma \mid -{T \over 2}  \leq y_3 < {T\over 2}  \} .
	$$
	Indeed, by construction  any point $\tilde y_h \in \tilde \Sigma^n_h$ is the rotation of a multiple of ${2\pi \over n}$ of a point $\tilde y_{0h} = X(y_{0h})$, with $y_{0h} = y_0 + h(y_0) \nu (y_0)$, $y_0 \in \Sigma_0$, see \eqref{variable}. 
	Besides, the domain of integration $\tilde \Omega^n_h$ in \eqref{N1} is invariant under rotations. Thus in this Section we will take points of the form
	\begin{equation} \label{uno}
		\begin{aligned}
			\tilde y_h &= X(y_h) , \quad y_h  = y + h(y) \nu (y) , \quad y = \left( \bar y ,  y_{3} \right) \in \Sigma_0, \quad \bar y = (y_{1}, y_{2}),  \quad\\  \bar y &= f(y_{3} ) \left( \cos \phi  , \sin \phi \right) = f(y_{3}) \, e^{i \phi}  \, , \quad {\mbox {for some}} \quad \phi \in [0,2\pi]\\
			&-{T\over 2} \leq y_{3} \leq {T\over 2}.
		\end{aligned}
	\end{equation}

	The main result is the following\\
	
	\begin{prop}\label{expN0} Let $\alpha \in (0,1)$ and  $h: \Sigma^n \to \R$ be a small $C^{2, \alpha}$ function satisfying \eqref{simh}. Let $y \in \Sigma_0$ as in \eqref{uno}. 
		Then
		\begin{equation}\label{expansion}
			\begin{aligned}
				\tilde N_{ \Sigma^n_h} (y) &={2|\Omega_0| \over T} \,\log n \, + F_0^n (y_{3} )  - y_{2} \, {2\pi |\Omega_0 | \over T^2} \, {\log n \over n} \, \\
				& + \frac 1n F_1^n (y)+ \, \left( {2  \over T} \, \ln n \, + b(y) \, \right) \, \int_{\Sigma_0  }   h  \\
				&+\, \ell_1 [h] (y) + \, {1\over n} \ell_2 [h] (y)
			\end{aligned}
		\end{equation}
		where $|\Omega_0|$ is the volume of the region enclosed by $\Sigma_0$ and $T$ is the period of the Delaunay surface $\Sigma$. The functions $F_0^n$, $F_1^n$ and $b$ are smooth functions of their variables, which are uniformly bounded, together with their derivatives, as $n \to \infty$. The function $F_0^n (y_3)$ is even and periodic of period $T$, the functions $F_1^n (y)$ and $b$ satisfy the symmetries \eqref{simh}.
		Moreover, $\ell_1$, $\ell_2$ are smooth functions,  uniformly bounded, together with their derivatives, as $n \to \infty$, and  $\ell_i (0) = 0$, $\nabla_h \ell_i (0) \not= 0$, $i=1,2$. They both satisfy the symmetries \eqref{simh}. Besides, if $h$ is also even with respect to $y_2$ then $\ell_1[h]$ is even in $y_2$.
	\end{prop}

	The proof of this Proposition follows from two results. First we assume that $h=0$ and 
	for a point $\tilde y \in \tilde \Sigma^n$ we  compute
	$$
	\int_{\tilde \Omega^n} {d \tilde x \over |\tilde y - \tilde x| } .
	$$
	We write
	\be \label{N1}
	N_{\Sigma^n} (y) =
	\int_{\tilde \Omega^n}\dfrac{d \tilde x}{|X(y) -\tilde x|},\quad \tilde y  = X (y) \in \tilde \Sigma^n, \quad y \in \Sigma^n.
	\ee

	\medskip
	
	\begin{lemma}\label{expN} Let $y \in \Sigma_0$ as in \eqref{uno}. Then 
		\begin{equation}\label{expansion}
			\begin{aligned}
				N_{ \Sigma^n} (y) &={2|\Omega_0| \over T} \,\log n \, + F_0^n (y_{3} )  - y_{2} \, {2\pi |\Omega_0 | \over T^2} \, {\log n \over n} \, \\
				& + \frac 1n F_1^n(y)
			\end{aligned}
		\end{equation}
		where $|\Omega_0|$ is the volume of the region enclosed by $\Sigma_0$. The functions $F_0^n$ and $F_1^n$ are smooth functions of their variables, which are uniformly bounded, together with their derivatives, as $n \to \infty$. The functions $F_0^n (y_3)$  and $F_1^n (y)$  satisfy the symmetries \eqref{simh}. 
	\end{lemma}

	\medskip
	The second result relates $N_{\Sigma^n}$ in \eqref{N1} and  $\tilde N_{\Sigma^n_h}$ in \eqref{N2}.
	
	\begin{lemma}\label{lemma1X}
		Let $\alpha \in (0,1)$ and  $h: \Sigma^n \to \R$ be a small $C^{2, \alpha}$ function satisfying \eqref{simh}. Let $y \in \Sigma_0$ as in \eqref{uno}. 
		Then
		\be \label{exptildeNn}
		\begin{aligned}
			\tilde N_{ \Sigma_h^n} (y ) &=N_{\Sigma^n} (y) 
			+ \, \left( {2  \over T} \, \ln n \, + b(y) \, \right) \,\,\int_{\Sigma_0  }   h  \\
			&+ \, \ell_1 [h] (y) + \, {1\over n} \ell_2 [h] (y)
			\quad {\mbox {as }} n \to \infty.
		\end{aligned}
		\ee 
		In \eqref{exptildeNn}, $N_{\Sigma^n} (y)$ is defined in \eqref{N1} and $b$, $\ell_1$, $\ell_2$ are smooth functions,  uniformly bounded, together with their derivatives, as $n \to \infty$,  and they satisfy the symmetries \eqref{simh}. 
		Besides, $\ell_i (0) = 0$, $\nabla_h \ell_i (0) \not= 0$, $i=1,2$, and if $h$ is also even with respect to $y_2$ then $\ell_1[h]$ is even in $y_2$.
	\end{lemma}

	\begin{remark} \label{rem2}
		The functions $\tilde N_{\Sigma_h^n} (y)$ as defined in \eqref{N2} and 
		$N_{\Sigma^n} (y)$ in \eqref{N1} satisfy the symmetries described in \eqref{simh}: they are even with respect to $y_1$ and $y_3$, and they are $T$-periodic in $y_3$.
	\end{remark}
	
	\medskip
	Proposition \ref{expN0} is a direct consequence of Lemmas \ref{expN} and \ref{lemma1X}.
	
	\medskip
	The rest of the Section is devoted to prove these results. We start with Lemma \ref{expN}.
	
	\medskip	
	
	\begin{proof}[Proof of Lemma \ref{expN}.]  
		Assume $y$ satisfies \eqref{uno}. 
		The expression of the non-local operator $N_{\tilde \Sigma^n}$ becomes simpler after the natural  change of variables
		\begin{equation}\label{change}
			\tilde x= X(x)  =      
			\Big ( x_1 \, , \, (R+x_2) \cos 
			\left (  \frac  { x_3} { R}    \right ) ,  (R+x_2) \sin  \left (\frac  { x_3} { R}    \right )\Big )
		\end{equation}
		which has volume element $d\tilde x  = (1+{x_2 \over R} ) \, dx$.  
		
		For the integrand, we check that
		\begin{align*}
			|\tilde x -  X(y)|^2 
			&= |\bar x - \bar y|^2 + a_R^2 \left( 1+{x_2 + y_{2} \over R} + {x_2 y_{2} \over R^2 }\right)
		\end{align*}
		where we use the notation
		$\tilde x = X (x)$ as in \eqref{change} and
		\begin{equation}\label{not1}
			x= (\bar x , x_3), \quad a_R^2= 2\, R^2 \, \left(1-\cos  \left( {x_3 - y_{3} \over 2\, R}\right) \right).
		\end{equation}
		Hence, 
		$$
		N_{ \Sigma^n}  (y ) = \iiint\limits_{\Omega^n}\frac{(1+{x_2 \over R})}{\sqrt{|\bar x - \bar y|^2 + a_R^2  \left( 1+{x_2 + y_{2} \over R} + {x_2 y_{2} \over R^2 }\right)}} \,   \, dx .$$
		Since
		$$
		\Omega^n = \bigcup_{k=0}^{n-1} \Omega_k $$
		as in \eqref{truncated}, 
		we split the region of integration over the sets $\Omega_k$ and get 
		\begin{align*}
			N_{ \Sigma^n}  (y ) &= \iiint\limits_{\Omega^n}\frac{(1+{x_2 \over R})}{\sqrt{|\bar x - \bar y|^2 + a_R^2  \left( 1+{x_2 + y_{2} \over R} + {x_2 y_{2} \over R^2 }\right)}} \,   \, dx\\&= 
			\sum_{k=0}^{n-1} \iiint\limits_{\Omega_k } \frac{(1+{x_2 \over R})}{\sqrt{|\bar x - \bar y|^2 + a_R^2  \left( 1+{x_2 + y_{2} \over R} + {x_2 y_{2} \over R^2 }\right)}} \,   \, dx  \\&= 
			\sum_{k=0}^{n-1}  I_k (y)  \quad {\mbox {where}}  \quad a_R \quad {\mbox {is in \eqref{not1}, and }}\\
			I_k (y) &:=\iiint\limits_{\Omega_0} \frac{(1+{x_2 \over R})}{\sqrt{|\bar x - \bar y|^2 + a_{Rk}^2  \left( 1+{x_2 + y_{2} \over R} + {x_2 y_{2} \over R^2 }\right)}} \,   \, dx
		\end{align*}
		with
		$$
		a_{Rk} =   2\, R \, \sin \left({k T \over 2R } + {x_{3} - y_{3} \over 2\, R} \right).$$ 
		\medskip
		
		Recalling that $2 R \pi = n \, T$, we get that  $a_{Rk} = {n T \over \pi} \sin ( {k \over n} \pi + {x_3 - y_{3} \over nT } \pi ) $ and
		\begin{align*}
			I_k (y) &:={\pi \over n T} \iiint\limits_{\Omega_0} \frac{(1+{x_2 \over R})}{\sqrt{  \sin^2 ( {k \over n} \pi + {x_3 - y_{3} \over nT } \pi )  \left( 1+{x_2 + y_{2} \over R} + {x_2 y_{2} \over R^2 }\right) + {|\bar x - \bar y|^2 \over n^2 T^2} \pi^2}} \,   \, dx \\
			&= {\pi \over n T} \iiint\limits_{\Omega_0} \frac{\gamma }{\sqrt{  \sin^2 ( {k \over n} \pi + {\alpha  \over n } )   + \frac{\beta^2}{n^2}}} \,   \, dx
		\end{align*}
		where, to simplify notations, we used
		\begin{equation}\label{parameter}
			\begin{aligned}
				\alpha &=  {x_3 - y_{3} \over T } \, \pi,  \quad \beta^2  = {|\bar x - \bar y|^2 \over  T^2} \, \pi^2 \, (1+ {x_2 + y_{2} \over R} + {x_2 y_{2} \over R^2 })^{-1} \\\delta &= (1+{x_2 \over R}) (1+ {x_2 + y_{2} \over R} + {x_2 y_{2} \over R^2 })^{-{1\over 2}}.
			\end{aligned}
		\end{equation}
		Notice that $\alpha$, $\beta$ and $\delta$ do not depend on $k$.  
		Besides, we can expand 
		\begin{equation}\label{gamma}
			\delta =  1   -   \frac {y_{2}}{2R}   + \frac {x_2}{2R} + 
			O( \frac 1{n^2} F_n(x,y) ) .
		\end{equation}
		Thus
		\begin{align*}
			N_{\Sigma^n} & (y ) =
			{1\over T} \, \iiint\limits_{\Omega_0} \delta  \, {\pi \over n} \, \left( \sum_{k=0}^{n-1} \frac{1}{\sqrt{  \sin^2 ( {k \over n} \pi + {\alpha  \over n } )   + \frac{\beta^2}{n^2}}}  \right) \,   \, dx
		\end{align*}
		which we write as follows
		\begin{equation}\label{d1}
			\begin{aligned}
				N_{ \Sigma^n}  (y ) &=
				{1\over T} \, \iiint\limits_{\Omega_0} \delta  \, \left(\int_0^\pi  \frac{1}{\sqrt{  \sin^2 ( s + {\alpha  \over n } )   + \frac{\beta^2}{n^2}}} \, ds  \right) \, dx\\
				&+ {1\over T} \, \iiint\limits_{\Omega_0} \delta  \,  J (x; y)\,  \, dx  
			\end{aligned}
		\end{equation}
		where
		$$
		J (x; y)= \sum_{k=0}^{n-1} \int_{k\pi \over n}^{(k+1) \pi \over n} \left(\frac{1}{\sqrt{  \sin^2 ( {k \over n} \pi + {\alpha  \over n } )   + \frac{\beta^2}{n^2}}} -            \frac{1}{\sqrt{  \sin^2 ( s + {\alpha  \over n } )   + \frac{\beta^2}{n^2}}} \right) \, ds 
		$$
		We claim that
		\begin{equation}\label{ee2}
			\begin{aligned}
				\int_0^\pi  \frac{1}{\sqrt{  \sin^2 ( s + {\alpha  \over n } )   + \frac{\beta^2}{n^2}}} & \, ds
				=2 \ln n  - \ln |\bar x- \bar y|^2 + 2\pi \,  {x_2 + y_{2} \over n\, T} \\
				&+ \ln 4T^2 +2\int_0^{\pi \over 2} {t-\sin t \over t \sin t} \, dt  +  {\ln n \over n^2} \, F_n (\alpha , \beta)  \\
			\end{aligned}
		\end{equation}
		and
		\begin{equation}\label{ee3}
			\begin{aligned}
				J (x; y_0)&= F(\alpha , \beta)  + {1\over n}  \, F_n (\alpha , \beta) 
			\end{aligned}
		\end{equation}
		where $F$, $F_n$ are uniformly  bounded, as $n \to \infty$, together with their derivatives in $\alpha $ \and $\beta$. The specific definition of
		$F_n$ changes from one formula to the other. We refer to \eqref{parameter} for $\alpha$ and $\beta$.

		\medskip
		We assume for the moment the validity of \eqref{ee2}, and \eqref{ee3}, whose proofs are postponed below.

		\medskip
		Using the expansion of $\delta$ as in  \eqref{gamma}  from \eqref{ee2} we get
		\begin{equation}\label{ee4}
			\begin{aligned}
				{1\over T} \, &\iiint\limits_{\Omega_0} \delta  \, \left(\int_0^\pi  \frac{1}{\sqrt{  \sin^2 ( s + {\alpha  \over n } )   + \frac{\beta^2}{n^2}}} \, ds  \right) \, dx = {2 |\Omega_0| \over T} \, \ln n \\
				&+ {|\Omega_0 |\over T} \,  A -{1\over T} \iiint\limits_{\Omega_0} \ln (|\bar x - \bar y|^2 )\, dx \\
				&- {2\, \pi \,  |\Omega_0|  \over T^2} \, y_{2} \, {\ln n \over n} \\
				&+ {2\pi \over n T^2} \, |\Omega_0 | \, \left( 1-{A\over 2} + {1\over 2 |\Omega_0 |} \iiint\limits_{\Omega_0} \ln (|\bar x - \bar y|^2 )\, dx \right) \, y_{2} \\
				&+{1\over n^2} F_n (y )  \\
				&{\mbox {where}} \quad A= \left( \ln 4T^2 +2\int_0^{\pi \over 2} {t-\sin t \over t \sin t} \, dt  \right) .
			\end{aligned}
		\end{equation}
		In order to compute ${1\over T} \, \iiint\limits_{\Omega_0} \delta  \,  J (x; y)\,  \, dx $, we need a more precise expansion of \eqref{ee3}.
		From \eqref{parameter}, we get
		\begin{align*}
			\beta  &= {|\bar x - \bar y| \over  T} \, \pi \, (1+ {x_2 + y_{2} \over R} + {x_2 y_{2} \over R^2 })^{-{1\over 2} } \\
			&=\bar \beta \left( 1- {x_2 + y_{2} \over 2R} + O ({1\over n^2} ) F_n (x,y)\right) \quad {\mbox {where}} \quad \bar \beta = {|\bar x - \bar y| \over  T} \, \pi .
		\end{align*}
		We thus write the function $F(\alpha , \beta)$ in \eqref{ee3} as follows
		$$
		F(\alpha , \beta ) = F(\alpha , \bar \beta) -{\pi \over n T}  \partial_\beta F (\alpha , \bar \beta ) [\bar \beta (x_2 + y_{2} )] + O({1\over n^2}) F_n (x,y).
		$$
		Thus we conclude
		\begin{equation}\label{ee5}
			\begin{aligned}
				{1\over T} \, & \iiint\limits_{\Omega_0} \delta  \,  J (x; y)\,  \, dx = {1\over T} \iiint\limits_{\Omega_0} F(\alpha , \bar \beta ) \, dx \\
				&- {\pi \over n T^2} \, y_{2} \, \left( \iiint\limits_{\Omega_0} F(\alpha , \bar \beta) \, dx + \iiint\limits_{\Omega_0} \partial_\beta F (\alpha , \bar \beta ) \bar \beta \, dx \right)
				\\
				&+{1\over n} F_n (\bar y ) .
			\end{aligned}
		\end{equation}
		Observe now that the function ${1\over T} \iiint\limits_{\Omega_0} \ln (|\bar x - \bar y|^2 )\, dx$ in \eqref{ee4} depends on $|\bar y|$. Similarly, ${1\over T} \iiint\limits_{\Omega_0} F(\alpha , \bar \beta ) \, dx$ in \eqref{ee5} only depends on $y_{3}$.
		The validity of the expansion \eqref{expansion} readily follows from \eqref{d1}, \eqref{ee4} and \eqref{ee5}.

		\medskip
		Next we prove the evenness of $N_{\Sigma^n}$ as a function of $y_{3}$. Observe that
		\begin{align*}
			I_k & (\bar y, - y_{3} )  =\iiint\limits_{\Omega_0}
			\frac{(1+{x_2 \over R})}{\sqrt{|\bar x - \bar y|^2 + \left( 2\, R \, \sin \left({k T \over 2R } + {x_{3} + y_{3} \over 2\, R} \right) \right)^2  \left( 1+{x_2 + y_{2} \over R} + {x_2 y_{2} \over R^2 }\right)}}\,  dx \\
			\\
			&\quad ( x_3= - z_3 , \quad \bar x = \bar z )\\
			&=\iiint\limits_{\Omega_0}\frac{1}{\sqrt{|\bar z - \bar y|^2 + \left( 2\, R \, \sin \left({k T \over 2R } + {-z_{3} + y_{3} \over 2\, R} \right) \right)^2  \left( 1+{y_2 + y_{2} \over R} + {y_2 y_{2} \over R^2 }\right)}} \,  dz \\&= \iiint\limits_{\Omega_0}\frac{1}{\sqrt{|\bar x - \bar y|^2 + \left( 2\, R \, \sin \left(-{k \over n } \pi + {x_{3} - y_{3} \over 2\, R} \right) \right)^2 \left( 1+{x_2 + y_{2} \over R} + {x_2 y_{2} \over R^2 }\right)}} \,  dx\\
			&= \iiint\limits_{\Omega_0}\frac{1}{\sqrt{|\bar x - \bar y|^2 + \left( 2\, R \, \sin \left({n-k \over n } \pi + {x_{3} - y_{3} \over 2\, R} \right) \right)^2\left( 1+{x_2 + y_{2} \over R} + {x_2 y_{2} \over R^2 }\right) }} \,  dy \\
			&= I_{n-k} (\bar y , y_{3}).
		\end{align*}
		Hence $ N_{ \Sigma^n} (\bar y , y_{3})$ is an even function in $y_{3}$, for $(\bar y, y_{3}) \in \Sigma_0$, namely
		\begin{align*}
			N_{ \Sigma^n}  (\bar y, - y_{3}  ) &= \sum_{k=0}^{n-1} I_k (-y_{03}) = \sum_{k=0}^{n-1} I_{n-k} (y_{3}) = \sum_{k=0}^{n-1} I_k (y_{3}) =   N_{ \Sigma^n}  (\bar y, y_{3}  )
		\end{align*}
		as $I_0 (\bar y , y_{3}) = I_n (\bar y, y_{3}).$ This proves part of the statements in Remark \ref{rem2}.
		
		\medskip
		
		The rest of the proof is devoted to obtaining \eqref{ee2} and \eqref{ee3}. 
		
		\medskip
		
		{\it Proof of \eqref{ee2}}. \ \ 
		We observe that
		\begin{align*}
			\int_0^\pi  \frac{1}{\sqrt{  \sin^2 ( s + {\alpha  \over n } )   + \frac{\beta^2}{n^2}}} \, ds &= ( \int_0^{\pi \over 2} + \int_{\pi \over 2}^\pi  )  \frac{1}{\sqrt{  \sin^2 ( s + {\alpha  \over n } )   + \frac{\beta^2}{n^2}}} \, ds \\
			=  \int_0^{\pi \over 2}  \frac{1}{\sqrt{  \sin^2 ( s + {\alpha  \over n } )   + \frac{\beta^2}{n^2}}} \, ds & + \int_0^{\pi \over 2}  \frac{1}{\sqrt{  \sin^2 ( s - {\alpha  \over n } )   + \frac{\beta^2}{n^2}}} \, ds
		\end{align*}
		since
		\begin{align*}
			\int_{\pi \over 2}^\pi & \frac{1}{\sqrt{  \sin^2 ( s + {\alpha  \over n } )   + \frac{\beta^2}{n^2}}} \, ds = \quad (s=\pi - t ) \\
			&= -\int_{\pi \over 2}^0 \frac{1}{\sqrt{  \sin^2 ( \pi -t  + {\alpha  \over n } )   + \frac{\beta^2}{n^2}}} \, dt=  \int_0^{\pi \over 2}  \frac{1}{\sqrt{  \sin^2 ( s - {\alpha  \over n } )   + \frac{\beta^2}{n^2}}} \, ds 
		\end{align*}
		So, let us consider 
		\begin{align*}
			\int_0^{\pi \over 2}  & \frac{1}{\sqrt{  \sin^2 ( s + {\alpha  \over n } )   + \frac{\beta^2}{n^2}}} \, ds
			=\int_0^{\pi \over 2} ( \frac{1}{\sqrt{  \sin^2 ( s + {\alpha  \over n } )   + \frac{\beta^2}{n^2}}} \, -   \frac{1}{\sqrt{   ( s + {\alpha  \over n } )^2   + \frac{\beta^2}{n^2}}} ) \, ds \\
			&+ \ln \left( {n\pi +\alpha \over \beta}  + \sqrt{ ({n\pi +\alpha \over \beta} )^2+1}\right) - \ln  \left( { \alpha \over \beta}  + \sqrt{ ({\alpha \over \beta} )^2+1}\right)\\
			&= \int_0^{\pi \over 2}  ( \frac{1}{\sqrt{  \sin^2 ( s + {\alpha  \over n } )   + \frac{\beta^2}{n^2}}} \, -   \frac{1}{\sqrt{   ( s + {\alpha  \over n } )^2   + \frac{\beta^2}{n^2}}} ) \, ds \\
			&+\ln n + \ln {2\pi \over \alpha + \sqrt{\alpha^2 + \beta^2} } + \ln (1+ {\alpha \over n \pi} ) + \ln \left({1+ \sqrt{1+{\beta^2 \over (n\pi + \alpha )^2} }\over 2} \right)
		\end{align*}
		Besides
		\begin{align*}
			\int_0^{\pi \over 2} &  ( \frac{1}{\sqrt{  \sin^2 ( s + {\alpha  \over n } )   + \frac{\beta^2}{n^2}}} \, -   \frac{1}{\sqrt{   ( s + {\alpha  \over n } )^2   + \frac{\beta^2}{n^2}}} ) \, ds\\
			&= 
			\int_0^{\pi \over 2} \frac{( s + {\alpha  \over n } )^2    - \sin^2 ( s + {\alpha  \over n } )   }{\sqrt{  \sin^2 ( s + {\alpha  \over n } )   + \frac{\beta^2}{n^2}}  \sqrt{   ( s + {\alpha  \over n } )^2   + \frac{\beta^2}{n^2}}  \left( \sqrt{  \sin^2 ( s + {\alpha  \over n } )   + \frac{\beta^2}{n^2}} + \sqrt{   ( s + {\alpha  \over n } )^2   + \frac{\beta^2}{n^2}} \right) } \, ds \\&= (t= s+ {\alpha \over n})\\
			&=  \int_{\alpha \over n}^{{\pi \over 2} + {\alpha \over n}  } \frac{t^2    - \sin^2 t   }{\sqrt{  \sin^2 t  + \frac{\beta^2}{n^2}}  \sqrt{  t^2   + \frac{\beta^2}{n^2}}  \left( \sqrt{  \sin^2 t  + \frac{\beta^2}{n^2}} + \sqrt{   t^2   + \frac{\beta^2}{n^2}} \right) } \, dt \\
			&= (\int_{C \over n}^{{\pi \over 2} + {\alpha \over n}  }  + \int_{\alpha \over n}^{{C \over n}  } ) f(t, {\beta^2 \over n^2} )\, dt
		\end{align*}
		for some fixed $C>0$, where
		$$
		f(t, \theta) = \frac{t^2    - \sin^2 t   }{\sqrt{  \sin^2 t  + \theta}  \sqrt{  t^2   + \theta}  \left( \sqrt{  \sin^2 t  + \theta } + \sqrt{   t^2   + \theta } \right) } . $$
		Since $\partial_\theta f(t,0) \sim {1\over t} $, as $t \to 0$, we have
		\begin{align*}
			\int_{C \over n}^{{\pi \over 2} + {\alpha \over n}  }   f(t, {\beta^2 \over n^2} )\, dt &=  \int_{C \over n}^{{\pi \over 2} + {\alpha \over n}  }   f(t, 0)\, dt + \ln n \, O({\beta^2 \over n^2})\\
			&= \int_0^{\pi \over 2} {t-\sin t \over t \sin t} \, dt +  \ln n \, O({\beta^2 \over n^2})
		\end{align*}
		Moreover, we check directly that
		$ \int_{C \over n}^{{\pi \over 2} + {\alpha \over n}  }  f(t, {\beta^2 \over n^2} )\, dt = O({1\over n^2})$ and we conclude that
		\begin{align*}
			\int_0^{\pi \over 2}   \frac{1}{\sqrt{  \sin^2 ( s + {\alpha  \over n } )   + \frac{\beta^2}{n^2}}} \, ds&= \ln n +  \ln {2\pi \over \alpha + \sqrt{\alpha^2 + \beta^2} }+ \int_0^{\pi \over 2} {t-\sin t \over t \sin t} \, dt \\
			& + \ln (1+ {\alpha \over n \pi} ) +  \ln n \, O({\beta^2 \over n^2})
		\end{align*}
		hence
		\begin{align*}
			\int_0^{\pi }   \frac{1}{\sqrt{  \sin^2 ( s + {\alpha  \over n } )   + \frac{\beta^2}{n^2}}} \, ds&=   2\ln n +  \ln {2\pi \over \alpha + \sqrt{\alpha^2 + \beta^2} }+ \int_0^{\pi \over 2} {t-\sin t \over t \sin t} \, dt \\
			& + \ln (1+ {\alpha \over n \pi} ) +  \ln n \, O({\beta^2 \over n^2})\\
			&+  \ln {2\pi \over -\alpha + \sqrt{\alpha^2 + \beta^2} }+ \int_0^{\pi \over 2} {t-\sin t \over t \sin t} \, dt \\
			& + \ln (1+ {-\alpha \over n \pi} ) +  \ln n \, O({\beta^2 \over n^2})
		\end{align*}
		which, in combination with \eqref{parameter}, gives the expansion \eqref{ee2}. 
		
		\medskip
		{\it Proof of \eqref{ee3}}. \ \ 
		To estimate $J$, we start rename $k= n-k-1$ and change the variable of integration $t=\pi-s$ to get
		\begin{align*}
			& J= \sum_{k=0}^{n-1} \int_{(n-k-1)\pi \over n}^{(n-k) \pi \over n} \left(\frac{1}{\sqrt{  \sin^2 ( {n-k-1 \over n} \pi + {\alpha  \over n } )   + \frac{\beta^2}{n^2}}} - \frac{1}{\sqrt{  \sin^2 ( s + {\alpha  \over n } )   + \frac{\beta^2}{n^2}}} \right) \, ds \\
			&=\sum_{k=0}^{n-1} \int_{(n-k-1)\pi \over n}^{(n-k) \pi \over n} \left(\frac{1}{\sqrt{  \sin^2 ( {k+1 \over n} \pi - {\alpha  \over n } )   + \frac{\beta^2}{n^2}}} - \frac{1}{\sqrt{  \sin^2 ( s + {\alpha  \over n } )   + \frac{\beta^2}{n^2}}} \right) \, ds\\
			&= \sum_{k=0}^{n-1} \int_{k\pi \over n}^{(k+1) \pi \over n} \left(\frac{1}{\sqrt{  \sin^2 ( {k+1 \over n} \pi - {\alpha  \over n } )   + \frac{\beta^2}{n^2}}} - \frac{1}{\sqrt{  \sin^2 ( t - {\alpha  \over n } )   + \frac{\beta^2}{n^2}}} \right) \, dt
		\end{align*}
		Combining this expression for $J$ with its definition we get
		\begin{align*}
			2 &J=  \sum_{k=0}^{n-1}  (a_k + b_k ) \,  , \quad {\mbox {where}} \\
			a_k&= \int_{k\pi \over n}^{(k+1) \pi \over n} \left( \frac{1}{\sqrt{  \sin^2 ( {k \over n} \pi + {\alpha  \over n } )   + \frac{\beta^2}{n^2}}} - \frac{1}{\sqrt{  \sin^2 ( s + {\alpha  \over n } )   + \frac{\beta^2}{n^2}}} \right) \, ds \\
			b_k &= \int_{k\pi \over n}^{(k+1) \pi \over n} \left(\frac{1}{\sqrt{  \sin^2 ( {k+1 \over n} \pi - {\alpha  \over n } )   + \frac{\beta^2}{n^2}}} - \frac{1}{\sqrt{  \sin^2 ( s - {\alpha  \over n } )   + \frac{\beta^2}{n^2}}} \right)\, ds.
		\end{align*}
		Using the change of variables $s={t\over n} +{k\pi \over n}$, we get
		\begin{align*}
			a_k 
			&= 
			\int_0^{\pi }  \left( \frac{1}{\sqrt{  n^2 \sin^2 ( {k \pi + \alpha \over n}   )   + \beta^2} }- \frac{1}{\sqrt{ n^2  \sin^2 ( {t+ k \pi + \alpha  \over n } )   + \beta^2}}\right) \, dt 
		\end{align*}
		Using the change of variables ${t\over n} ={k+1 \over n} \pi - s$
		\begin{align*}
			b_k 
			&= 
			\int_0^{\pi }  \left( \frac{1}{\sqrt{  n^2 \sin^2 ( {(k +1) \pi - \alpha \over n}   )   + \beta^2} }- \frac{1}{\sqrt{ n^2  \sin^2 ( -{t\over n} + {(k+1)  \pi -\alpha  \over n } )   + \beta^2}}\right) \, dt 
		\end{align*}
		Thus we have
		\begin{align*}
			2J &= \sum_{k=0}^{n-1}  \int_0^{\pi }  \left( \frac{1}{\sqrt{  n^2 \sin^2 ( {k \pi + \alpha \over n}   )   + \beta^2} }- \frac{1}{\sqrt{ n^2  \sin^2 ( {s+ k \pi + \alpha  \over n } )   + \beta^2}}\right) \, ds \\
			&+ \sum_{k=0}^{n-1}  \int_0^{\pi }  \left( \frac{1}{\sqrt{  n^2 \sin^2 ( {(k +1) \pi - \alpha \over n}   )   + \beta^2} }- \frac{1}{\sqrt{ n^2  \sin^2 (  {-s+ (k+1)  \pi -\alpha  \over n } )   + \beta^2}}\right) \, ds 
		\end{align*}
		Let us decompose
		\begin{align*}
			2J &= \sum_{k=0}^{[\sqrt{n}]-1} + \sum_{k=[\sqrt{n}]}^{n-1 - [\sqrt{n}]} + \sum_{k=n - [\sqrt{n}]}^{n-1} =   J_1 + J_2 + J_3
		\end{align*}
		
		We claim that
		\begin{equation}\label{J1J3}
			J_1 = J_3.
		\end{equation}
		Indeed, we check they have the same number of elements which can be redistributed as follows
		\begin{equation}\label{ee1}
			J_1 = \sum_{k=0}^{[\sqrt{n}]-1} (a_k + b_k) =
			\sum_{k=n-[\sqrt{n}]}^{n-1}  (a_{n-k-1} + b_{n-k-1} )
		\end{equation}
		where
		\begin{align*}
			a_{n-k-1} &+ b_{n-k-1} = \int_0^{\pi }  \left( \frac{1}{\sqrt{  n^2 \sin^2 ( {(n-k-1 )\pi + \alpha \over n}   )   + \beta^2} }- \frac{1}{\sqrt{ n^2  \sin^2 ( {s+ (n-k-1) \pi + \alpha  \over n } )   + \beta^2}}\right) \, ds \\
			&+  \int_0^{\pi }  \left( \frac{1}{\sqrt{  n^2 \sin^2 ( {(n-k ) \pi - \alpha \over n}   )   + \beta^2} }- \frac{1}{\sqrt{ n^2  \sin^2 (  {-s+ (n-k)  \pi -\alpha  \over n } )   + \beta^2}}\right) \, ds \\
			&= \int_0^{\pi }  \left( \frac{1}{\sqrt{  n^2 \sin^2 ( {(k+1 )\pi - \alpha \over n}   )   + \beta^2} }- \frac{1}{\sqrt{ n^2  \sin^2 ( {s \pi - k \pi + \alpha  \over n } )   + \beta^2}}\right) \, ds \\
			&+  \int_0^{\pi }  \left( \frac{1}{\sqrt{  n^2 \sin^2 ( {k \pi + \alpha \over n}   )   + \beta^2} }- \frac{1}{\sqrt{ n^2  \sin^2 (  {-s-k  \pi -\alpha  \over n } )   + \beta^2}}\right) \, ds \\
			&= \int_0^{\pi }  \left( \frac{1}{\sqrt{  n^2 \sin^2 ( {(k+1 )\pi - \alpha \over n}   )   + \beta^2} }- \frac{1}{\sqrt{ n^2  \sin^2 ( {s \pi - k \pi + \alpha  \over n } )   + \beta^2}}\right) \, ds \\
			&+  \int_0^{\pi }  \left( \frac{1}{\sqrt{  n^2 \sin^2 ( {k \pi + \alpha \over n}   )   + \beta^2} }- \frac{1}{\sqrt{ n^2  \sin^2 (  {s+k  \pi +\alpha  \over n } )   + \beta^2}}\right) \, ds .
		\end{align*}
		Observe now that the  second term of the last equality  can be written as 
		\begin{align*}
			\int_0^\pi & \frac{1}{\sqrt{ n^2  \sin^2 ( {s-\pi - k \pi + \alpha  \over n } )   + \beta^2}} ds  = \quad (s-\pi = -t)\\
			&= \int_0^\pi  \frac{1}{\sqrt{ n^2  \sin^2 ( {t +k \pi - \alpha  \over n } )   + \beta^2}} dt.
		\end{align*}
		Thus we conclude  that
		$$
		a_{n-k-1} + b_{n-k-1} = a_k + b_k .
		$$
		Inserting this identity in \eqref{ee1}, we conclude that $J_1=J_3$.
		
		\medskip
		Let us analyze the term $J_1$. Thus we are assuming $0\leq k \leq [\sqrt{n}]$.

		For $0\leq k \leq [\sqrt{n}]$, we Taylor expand the first integrand in the definition of $a_k$ 
		\begin{align*}
			& \left(  n^2 \sin^2 ( {k \pi + \alpha \over n}   )   + \beta^2 \right)^{-{1\over 2}}= 
			\left( (k\pi + \alpha)^2 (1-{1\over 3} ({k\pi +\alpha  \over n})^2 + O({k\over n})^4 ) +\beta^2 \right)^{-{1\over 2}}\\
			&= \left(1+{1\over 6} ({k\pi +\alpha  \over n})^2 + O({k\over n})^4 \right) \left( (k\pi + \alpha)^2  +{\beta^2 
				\over  (1-{1\over 3} ({k\pi +\alpha  \over n})^2 + O({k\over n})^4 ) } \right)^{-{1\over 2}} \\
			&= \left(1+{1\over 6} ({k\pi +\alpha  \over n})^2 + O({k\over n})^4 \right) \left( (k\pi + \alpha)^2 + \beta^2 \right)^{-{1\over 2} } \\
			& \quad \quad \quad \times 
			\left(1- {\beta^2 \over (k\pi + \alpha)^2 + \beta } ({1\over 6} ({k\pi +\alpha  \over n})^2 + O({k\over n})^4) )  \right) \\
			&= \left(1+{1\over 6} ({k\pi +\alpha  \over n})^2  + O({1\over n^2}) \right) \left( (k\pi + \alpha)^2 + \beta^2\right)^{-{1\over 2} }\\
			&= \left(1+{1\over 6} {k^2\pi^2  \over n^2} \right) \left( (k\pi + \alpha)^2 + \beta^2\right)^{-{1\over 2} } + \left( O({k\over n^2})  + O({1\over n^2}) \right)  \left( (k\pi + \alpha)^2 + \beta^2\right)^{-{1\over 2} }
			.
		\end{align*}
		Similarly we get, for $t \in (0,\pi) $
		\begin{align*}
			\left(  n^2 \sin^2 ( {t+k \pi + \alpha \over n}   )   + \beta^2 \right)^{-{1\over 2}}&=   
			\left(1+{1\over 6} {k^2\pi^2  \over n^2} \right) \left( (t+ k\pi + \alpha)^2 + \beta^2 \right)^{-{1\over 2} } \\
			&+ \left( O({k\over n^2})  + O({1\over n^2}) \right)  \left( (t+ k\pi + \alpha)^2 + \beta^2\right)^{-{1\over 2} }
			.
		\end{align*}
		Therefore, setting
		\begin{equation}\label{defell} \ell (s) = \sqrt{ (s+k\pi)^2 + \beta^2}, 
		\end{equation}
		we get
		\begin{align*}
			a_k &= {1+{1\over 6} {k^2\pi^2  \over n^2} \over \ell (\alpha  ) } \int_0^\pi {s^2 + 2s (k\pi + \alpha ) \over \ell (s+\alpha )  [ \ell (\alpha ) + \ell (s+\alpha )   ]} ds \\
			&+ \left( O({k\over n^2})  + O({1\over n^2}) \right)  \left( \int_0^\pi {ds \over  \ell (\alpha ) }\,   + \int_0^\pi {ds \over \ell (\alpha +s) }\,   \right) 
		\end{align*}
		and
		\begin{align*}
			b_k &= {1+{1\over 6} {k^2\pi^2  \over n^2} \over \ell (  \pi - \alpha )  }  \int_0^\pi {s^2 - 2s ( \pi - \alpha ) \over \ell ( \pi - \alpha - s) [ \ell ( \pi - \alpha ) + \ell ( \pi - \alpha -s) ] } ds \\
			&+ \left( O({k\over n^2})  + O({1\over n^2}) \right)  \left( \int_0^\pi {ds \over \ell (\pi-\alpha ) \, } + \int_0^\pi {ds \over \ell (\pi-\alpha - s )} \,  \right) .
		\end{align*}
		The above expressions can be further expanded
		\begin{align*}
			a_k &= {1\over 2 \pi^2 k^2} 
			\left( 1+{1\over 6} {k^2\pi^2  \over n^2} \right) \int_0^\pi 2s\, ds \\
			&+   (1+{1\over 6} {k^2\pi^2  \over n^2})  \int_0^\pi    {2 s \, \over \sqrt{ (1 + {\alpha \over k \pi}  )^2 + {\beta^2 \over (k\pi)^2} }\,  \ell (s+\alpha ) [ \ell (\alpha ) + \ell (s+\alpha)]  } -{s \over  \pi^2 k^2}\, ds \\
			&+ {1+{1\over 6} {k^2\pi^2  \over n^2} \over \sqrt{ (k\pi + \alpha )^2 + \beta^2 } }\int_0^\pi {s^2 +2s\alpha  \over \ell (s+\alpha ) [ \ell (\alpha ) + \ell (s+\alpha)]  } ds \\
			&+ \left( O({k\over n^2})  + O({1\over n^2}) \right)  \left( \int_0^\pi \frac 1 {\ell (\alpha ) } \, ds  + \int_0^\pi \frac {ds} {\ell (\alpha +s )}  \right),\\
			b_k&= - {1\over 2\pi^2 k^2} \left( 1+{1\over 6} {k^2\pi^2  \over n^2} \right) \int_0^\pi 2s\, ds \\
			&+   ( 1+{1\over 6} {k^2\pi^2  \over n^2})    \int_0^\pi    {-2 s \,  \over \sqrt{ (1 + {\pi - \alpha \over k \pi}  )^2 + {\beta^2 \over (k\pi)^2}  } \, \ell (\pi - \alpha - s) [ \ell (\pi - \alpha ) + \ell (\pi - \alpha -s) ]  } +{s \over \pi^2 k^2}  \, ds \\
			&+ {1+{1\over 6} {k^2\pi^2  \over n^2} \over \sqrt{ ((k+1) \pi   -\alpha )^2 + \beta^2 } }\int_0^\pi {s^2 - 2s (\pi - \alpha )  \over  \ell (\pi - \alpha - s) [ \ell (\pi - \alpha ) + \ell (\pi - \alpha -s) ] } ds \\
			&+ \left( O({k\over n^2})  + O({1\over n^2}) \right)  \left( \int_0^\pi \frac 1{\ell  (\pi-\alpha ) } + \int_0^\pi  \frac 1 {\ell (\pi-\alpha - s ) } \, ds \right) \\
		\end{align*}
		so that 
		\begin{align*}
			a_k + b_k &=   i + ii + iii + iv \\
			&+ \left( O({k\over n^2})  + O({1\over n^2}) \right)  \left( \int_0^\pi {ds \over \ell (\alpha )} \,   + \int_0^\pi {ds \over \ell (\alpha +s ) } \,   \right),\\
			&+ \left( O({k\over n^2})  + O({1\over n^2}) \right)  \left( \int_0^\pi {ds \over \ell(\pi-\alpha ) }\,  + \int_0^\pi {ds \over \ell (\pi-\alpha - s ) }\,  \right) , \quad {\mbox {with}} \\
			i&=(1+{1\over 6} {k^2\pi^2  \over n^2})  \int_0^\pi    {2 s \, \over \sqrt{ (1 + {\alpha \over k \pi}  )^2 + {\beta^2 \over (k\pi )^2}}\,  \ell (s+\alpha ) [ \ell (\alpha ) + \ell (s+\alpha)]  } -{ s \over \pi^2 k^2}  \, ds \\
			ii&=( 1+{1\over 6} {k^2\pi^2  \over n^2})    \int_0^\pi    {-2 s \,  \over \sqrt{ (1 + {\pi - \alpha \over k \pi}  )^2 + {\beta^2 \over (k\pi)^2} } \, \ell (\pi - \alpha - s) [ \ell (\pi - \alpha ) + \ell (\pi - \alpha -s) ]  } +{s \over \pi^2 k^2}  \, ds  \\
			iii&= {1+{1\over 6} {k^2\pi^2  \over n^2} \over \sqrt{ (k\pi + \alpha )^2 + \beta^2 } }\int_0^\pi {s^2 +2s\alpha  \over \ell (s+\alpha ) [ \ell (\alpha ) + \ell (s+\alpha)]    } ds \\
			iv&= {1+{1\over 6} {k^2\pi^2  \over n^2} \over \sqrt{ (  \pi - \alpha )^2 + \beta^2 } }\int_0^\pi {s^2 - 2s (\pi - \alpha )  \over  \ell (\pi - \alpha - s) [ \ell (\pi - \alpha ) + \ell (\pi - \alpha -s) ] } ds.
		\end{align*}
		Expanding $\ell$ given by \eqref{defell} as $k $ becomes large, we compute
		\begin{align*}
			i &= -{3\over 2 \pi^3 k^3} \int_0^\pi (2\alpha + s ) \, ds +{1\over k^4} F_k (\alpha , \beta) \\
			ii&= {1\over  \pi^3 k^3} \int_0^\pi (s^2 + 2 \alpha s ) \, ds +{1\over k^4} F_k (\alpha , \beta) \\
			iii&= {3\over 2 \pi^3 k^3} \int_0^\pi (2\pi -2\alpha - s ) \, ds +{1\over k^4} F_k (\alpha , \beta) \\
			iv&= {1\over  \pi^3 k^3} \int_0^\pi (s^2 - 2s (\pi - \alpha)  ) \, ds +{1\over k^4} F_k (\alpha , \beta)
		\end{align*}
		where $F_k$ denotes a generic function, whose precise definition changes from line to line, with the property that it is bounded,  together with its derivatives $\partial_\alpha F_k$, $\partial_\beta F_k $.
		We check that 
		\begin{align*}
			-{3\over 2 \pi^3 } \int_0^\pi (2\alpha + s ) \, ds &
			+ {1\over  \pi^3 } \int_0^\pi (s^2 + 2 \alpha s ) \, ds \\
			&+{3\over 2 \pi^3 } \int_0^\pi (2\pi -2\alpha - s ) \, ds + {1\over  \pi^3 } = ({1\over 6} - {\alpha \over 2\pi})
		\end{align*}
		We conclude that
		\begin{equation}\label{J1est}
			\begin{aligned}
				J_1&= F(\alpha , \beta)  + {1\over n}  \, F_n (\alpha , \beta) 
			\end{aligned}
		\end{equation}
		where $F$, $F_n$ are uniformly  bounded, as $n \to \infty$, together with their derivatives in $\alpha $ \and $\beta$.

		\medskip
		We now consider
		$$
		J_2 = \sum_{k=[\sqrt{n}]}^{n-1-[\sqrt{n}]} a_k + b_k.
		$$
		We write
		\begin{align*}
			J_2 &= {1\over n} \sum_{k=[\sqrt{n}]}^{n-1-[\sqrt{n}]} \int_0^\pi \left( f({\alpha \over n}) - f ({\alpha + s \over n}) + f({\pi-\alpha \over n}) - f ({\pi +s -\alpha \over n}) \right)\, ds\\
			& \quad {\mbox {with}} \\
			f(\theta ) & = {1 \over \sqrt{ \sin^2 ({k\pi \over n} + \theta)+{\beta^2 \over n^2}}}.
		\end{align*}
		We have
		\begin{align*}
			\sqrt{  \sin^2 ({k\pi \over n} + {\theta \over n} ) +{\beta^2 \over n^2} }&= \sin ({k \pi\over n})
			\left( (\cos {\theta \over n} + \cot {k\pi \over n} \sin {\theta \over n} )^2  +{\beta^2 \over n^2 \, \sin^2 ({k \pi\over n})}  \right)^{1\over 2} \\
			&= \sin ({k \pi\over n})
			\left( 1+ 2 {\theta \over n} \cot{k\pi \over n} + O({1\over n^2})  \right)^{1\over 2}\\
			&= \sin ({k \pi\over n})
			\left( 1+  {\theta \over n} \cot{k\pi \over n} + O({1\over n^2})  \right).
		\end{align*}
		Hence
		\begin{align*}
			J_2&= {1\over n} \sum_{k=[\sqrt{n}]}^{n-1-[\sqrt{n}]} {\cot {k\pi \over n} \over \sin {k\pi \over n}} \int_0^\pi \left( {2s \over n} + O({1\over n^2}) \right) \, ds \\
			&= {1\over n^2} \sum_{k=[\sqrt{n}]}^{n-1-[\sqrt{n}]} {\cot {k\pi \over n} \over \sin {k\pi \over n}}  \left( \pi^2 + O({1\over n}) \right) \\
			&= {1\over n^2} \sum_{k=[\sqrt{n}]}^{n-1-[\sqrt{n}]} {\cos {k\pi \over n} \over \sin^2 {k\pi \over n}}  \left( \pi^2 + O({1\over n}) \right).
		\end{align*}
		Relabelling $k $ with $n-k$, we get
		\begin{align*}
			\sum_{k=[\sqrt{n}]}^{n-1-[\sqrt{n}]} &{\cos {k\pi \over n} \over \sin^2 {k\pi \over n}} =  \sum_{k=[\sqrt{n}]+1}^{n-[\sqrt{n}]} {\cos {(n-k) \pi \over n} \over \sin^2 {(n-k)\pi \over n}} \\
			&=  - \sum_{k=[\sqrt{n}]+1}^{n-[\sqrt{n}]} {\cos {k \pi \over n} \over \sin^2 {k \pi \over n}}\\
			&= - \sum_{k=[\sqrt{n}]}^{n-1-[\sqrt{n}]} {\cos {k \pi \over n} \over \sin^2 {k \pi \over n}} + {\cos {[\sqrt{n}] \over n} \pi\over \sin^2 {[\sqrt{n}] \over n} \pi } -
			{\cos {n-  [\sqrt{n}] \over n} \pi\over \sin^2 {n-[\sqrt{n}] \over n} \pi} \\
			&=  - \sum_{k=[\sqrt{n}]}^{n-1-[\sqrt{n}]} {\cos {k \pi \over n} \over \sin^2 {k \pi \over n}} + {2\cos {[\sqrt{n}] \over n} \pi\over \sin^2 {[\sqrt{n}] \over n} \pi } .
		\end{align*}
		Thus
		\begin{align*}
			\sum_{k=[\sqrt{n}]}^{n-1-[\sqrt{n}]} {\cos {k \pi \over n} \over \sin^2 {k \pi \over n}} &= {\cos {[\sqrt{n}] \over n} \pi\over \sin^2 {[\sqrt{n}] \over n} \pi }   , \quad {\mbox {and}}
		\end{align*}
		\begin{equation}\label{J2est}
			J_2= {\pi^2 \over n^2} \, {\cos {[\sqrt{n}] \over n} \pi\over \sin^2 {[\sqrt{n}] \over n} \pi }  \, (1+ O({1\over n})) = {\pi^2 \over n} \,  \, (1+ O({1\over n})) .
		\end{equation}
		Putting together \eqref{J1J3}, \eqref{J1est} and \eqref{J2est}, we obtain \eqref{ee3}.
		
	\end{proof}

	\medskip
	We next prove Lemma \ref{lemma1X}.
	
	\medskip

	\begin{proof}[Proof of Lemma \ref{lemma1X}.]
		We write
		\begin{align*}
			\tilde N_{\Sigma^n_h} (y)  &= N_1[h] (y) + N_2[h] (y)\\
			N_1[h] (y) &= \int_{\tilde \Omega^n_h  }\dfrac{d \tilde x}{|X(y) -\tilde x|}\\
			N_2[h] (y) &= \int_{\tilde \Omega^n_h  } \left(\dfrac{1}{|X(y_h) -\tilde x|} - \dfrac{1}{|X(y) -\tilde x|}\right) d \tilde x
		\end{align*}
		Here we use the same notations as in \eqref{N2}.
		
		We start with $N_2$. Thus we readily get
		$$
		N_2 [0] (y) = 0.
		$$
		Arguing as in the proof of Lemma \ref{expN}, we can show that there exists a constant $C>0$ such that
		$$
		| N_2 [h] (y)| \leq C \, \left(\sum_{j=1}^n {1\over j^2} \right) \, \| h \|_{L^\infty}
		$$
		for all $n$ large. 
		Assume now that $h$ is even in $y_2$. We aim at showing that $N_2$ is the sum of a term which is even in $y_2$ plus another term which is $n^{-1}$ smaller, as in \eqref{exptildeNn}. We start with the change of variable 
		$\tilde x = X (x)$ as in \eqref{change}, we have
		\begin{align*}
			N_2[h] (y) &= \int_{ \Omega^n_h  } \left(\dfrac{1}{|X(y_h) -X( x)|} - \dfrac{1}{|X(y) - X(x)|}\right) (1+ {x_2 \over R}) d  x 
		\end{align*}
		and we just consider $\int_{ \Omega^n_h  } \left(\dfrac{1}{|X(y_h) -X( x)|} - \dfrac{1}{|X(y) - X(x)|}\right) dx$, as the other part is $n^{-1}$-smaller for $n$ large.
		
		Writing 
		\begin{align*}
			X(y_h) &= X(y) + T[h] (y), \quad T[h] (y) = h (y) \int_0^1 DX (y+ s h \nu) \, [\nu]\, ds
		\end{align*}
		we have that
		\begin{align*}
			X( y ) &= (y_1 , -y_2+R , y_3) + O(R^{-1})\\
			X( y) &- X(x) = (y_1 - x_1 ,  y_2 - x_2 , y_3 - x_3 ) + O(R^{-1})\\
			T[h] ( y)&= h(y) \nu ( y) + O(R^{-1})
		\end{align*}
		Hence
		\begin{align*}
			N_2[h] (y )&= \bar N_2[h] (y)  + {1\over n} \ell_2 [h] (y)\\
			\bar N_2[h](y) &= \int_{ \Omega^n_h  } |y -  x|^{-1} \left[ \left( 1+ 2 {h \nu (y) \cdot (y -  x) \over |y -  x|^2} + { h^2 \over |y - x|^2} \right)^{-{1\over 2}} -1 \right]
		\end{align*}
		with $\ell_2 [0] =0$ and $ |\ell_2 [h] (y) | \leq C \| h \|_{L^\infty}$ for some constant $C$. Using the symmetries of the domain $\Omega_h^n$, we can show that
		$$
		\bar N_2 [h] (\hat y ) = \bar N_2 [h] (y), \quad \hat y = (y_1, - y_2 , y_3)
		$$
		since $h$ is assumed to be even in $y_2$.

		\medskip
		Let us now consider $N_1 [h] (y)$.
		We decompose the domain $\tilde \Omega^n_h$ into
		$$
		\tilde \Omega^n_h = \tilde \Omega^n \cup \left( \tilde \Omega^n_h \setminus \tilde \Omega^n \right) 
		$$
		and we split 
		\begin{align*}
			N_1[h] (y) &= N_{\Sigma^n} (y)+ \bar N_1 [h] (y),\\
			\bar N_1 [h] (y)&= 
			\int_{\tilde \Omega^n_h \setminus \tilde \Omega^n }\dfrac{d \tilde x}{|X(y) -\tilde x|}.
		\end{align*}
		We refer to \eqref{N1} for the definition of $N_{\Sigma^n} (y)$.

		\medskip
		As in the prof of Lemma \ref{expN}, we change variable $\tilde x = X(x)$, with $x \in \Omega_h^n \setminus \Omega^n$ to get
		$$
		\bar N_1 [h] (y)=
		\int_{ \Omega^n_h \setminus  \Omega^n }\dfrac{ (1+ {x_2 \over R} ) \, d  x}{|X(y) -X(x)|},
		$$
		since $d \tilde x =(1+ {x_2 \over R} ) \, d  x $.
		We perform a second change of variable in $\Omega^n_h \setminus  \Omega^n $
		$$
		x = z + s h (z ) \, \nu (z) , \quad z \in \Sigma^n, \quad s \in (0,1).
		$$
		Thus
		\begin{align*}
			\bar N_1 [h] (y)&=
			\int_{\Sigma^n  } h(z )  \, d \sigma (z) \, \int_0^1 \dfrac{ (1+ {(z + s h (z ) \, \nu (z))_2 \over R} ) \, }{|X(y) -X(z + s h (z ) \, \nu (z))|} \, ds \\
			&= \int_{\Sigma^n  }  \,  \dfrac{ h(z )  \, d \sigma (z)\, }{|X(y) -X(z) |} \, + \hat N_1 [h] (y) 
		\end{align*}
		with
		$$
		\hat N_1 [0] (y) = (D_h \hat N_1)[0](y) = 0, 
		\quad 
		| \hat N_1 [h] (y) | \leq C \| h \|_{L^\infty}^2,
		$$
		for some positive constant $C$ uniformly bounded as $n \to \infty$. On the other hand, writing $\Sigma^n = \cup_{k=0}^{n-1} \Sigma_k$ as in \eqref{truncated} and using the change of variable
		$$
		z = z_0 + k T \, e_3\in \Sigma_k, \quad z_0 \in \Sigma_0 , \quad e_3 = (0,0,1)
		$$
		we have
		\begin{align*}
			\int_{\Sigma^n  }  \, & \dfrac{ h(z )  \, d \sigma (z) \, }{|X(y) -X(z) |} = \sum_{k=0}^{n-1} \int_{\Sigma_k  } \, \dfrac{ h(z )  \, d \sigma (z)\, }{|X(y) -X(z) |} \\
			&= \sum_{k=0}^{n-1} \int_{\Sigma_0  }  \,  \dfrac{ h(z_0 + kT \, e_3 )  \, d \sigma (z_0)\, }{|X(y) -X(z_0 + kT \, e_3 ) |} \, \\
			&=\sum_{k=0}^{n-1} I_k (y), \quad {\mbox {where}}\\
			I_k (y) &= \int_{\Sigma_0  }  \, \ \dfrac{ h(z_0  )  \, d \sigma (z_0)\, }{\sqrt{|\bar y - \bar z_0|^2 + a_{Rk}^2 (1+{y_2 + z_{02} \over R} + {y_2 z_{02} \over R^2} )}} 
		\end{align*}
		with
		$$
		a_{Rk} =   2\, R \, \sin \left({k T \over 2R } + {y_{3} - z_{03} \over 2\, R} \right).$$ 
		To get this formula we use 
		that $h$ satisfies the symmetries \eqref{simh} and we argue 
		in a similar way as in the proof of Lemma \ref{expN}, formula \eqref{not1}. 
		
		\medskip
		We follow the arguments to prove \eqref{ee2}-\eqref{ee3}- \eqref{ee4}-\eqref{ee5} in the proof of Lemma \ref{expN} to get
		
		\begin{align*}
			\int_{\Sigma^n  }  \, & \dfrac{ h(z )  \, d \sigma (z) \, }{|X(y) -X(z) |} =  \left( {2  \over T} \, \ln n \, + b(y) \right) \, \int_{\Sigma_0  }  \, \  h(z_0  )  \, d \sigma (z_0)\,  + N_3 [h] (y)  
		\end{align*}
		with $b$ as in the statement of Lemma \ref{lemma1X} and
		$$
		N_3 [0] (y) = 0, \quad 
		| N_3 [h] (y)| \leq C  \, \| h \|_{L^\infty}
		$$
		for some constant $C$ uniformly bounded in $n$. Arguing as before, we can show that $N_3[h] (y)$ is even in $y_2$, if $h$ is.
		
		This gives the expansion in \eqref{exptildeNn}. 
		We conclude with the remark that
		$y \to \ell_i [h] (y)$ in \eqref{exptildeNn} satisfies \eqref{simh}, as consequence of Remark \ref{rem2}, Lemma \ref{expN} and the above argument.

	\end{proof}

	\section{Invertibility of the Jacobi operator}
	
	We combine the results of Propositions \ref{lemma1} and \ref{expN0} to re-write Problem \eqref{3} in one to find $h: \Sigma \to \R$ and constants $\gamma$, $\lambda$ such that
	\be \label{newP}
	\begin{aligned}
		J_{\Sigma^n} [h]   &= E_\gamma (y) 
		+ \, \gamma \, \left( {2  \over T} \, \ln n \, + b(y) \, \right) \,\,\int_{\Sigma_0  }   h  
		+ \, \gamma \, \ell_1 [h] (y) \\
		&+ n^{-1} \, \ell [h, D h, D^2 h] (y)
		+ q[h, D h , D^2 h] (y) + \lambda \\
		& {\mbox {where}} \\
		E_\gamma(y) &= \gamma \,  F_0^n (y_{3} ) + \frac {\gamma}n F_1^n (y)   \\
		&+
		{  y_2 \over  n } [ F(y_3) - \gamma \, {\ln n }  \,   \, {2\pi |\Omega_0 | \over  T^2}] , \\			
		F(y_3) =& \frac {2\pi } {T f} \Biggl(    { (2-(f')^2 ) \, f \, f'' \over (1+ (f')^2 )^{5\over 2}} +{1 + 3 (f')^2 \over (1+ (f')^2 )^{3\over 2}} \Biggl),
	\end{aligned}
	\ee
	Here
	$J_{\Sigma^n}$ denotes the Jacobi operator of $\Sigma^n$ defined by \eqref{Jacobi}, 
	$|\Omega_0|$ is the volume of the region enclosed by $\Sigma_0$, see \eqref{sigma0} and $T$ is the period of the Delaunay surface $\Sigma$. The functions $F_0^n$, $F_1^n$, $F$ and $b$ are smooth in their arguments and uniformly bounded together with their derivatives, as $n \to \infty$. They satisfy the symmetry assumptions \eqref{simh}.
	In addition, 
	if $h$ is even with respect to $y_2$, then the functions 
	$ \ell_1 [h] (y), \  q[h,Dh, D^2 h ] (y)
	$
	are even in $y_2.$	 The function $E_\gamma$ in \eqref{newP} satisfies the symmetries and periodicity conditions  \equ{simh}.
	
	For $h \in C^{2,\alpha}$ satisfying \eqref{simh}, $\ell_1 [h]$, $\ell [h, Dh , D^2 h]$ and $q [h, Dh , D^2 h]$ satisfy \eqref{simh}.  Besides, there exists a constant $C$ such that
	\be \label{chepalle}
	\begin{aligned}
		&\| \ell_1 [h] \|_{C^\alpha} + \| \ell [h, Dh , D^2 h] \|_{C^\alpha} \leq C \| h\|_{C^{2, \alpha} }, \quad \\
		&\| q [h, Dh , D^2 h] \|_{C^\alpha} \leq C \| h\|_{C^{2, \alpha} }^2.
	\end{aligned}
	\ee
	
	The proofs of Propositions \ref{lemma1} and \ref{expN0} yields uniform Lipschitz dependence on $h$ of $ \ell_1$, $ \ell$ and $q$ in \eqref{newP}. Precisely, there exists $C$ such that
	\be \label{chepalle1}
	\begin{aligned}
		\|  \ell_1 [h_1] - \ell_1 [h_2] \|_{C^\alpha} &\leq C  \| h_1 - h_2 \|_{C^{2, \alpha} } \\
		\|  \ell [h_1 , Dh_1 , D^2 h_1] - \ell [h_2, Dh_2 , D^2 h_2] \|_{C^\alpha} &\leq C  \| h_1 - h_2 \|_{C^{2, \alpha} }
	\end{aligned}
	\ee
	and
	\be\label{chepalle2}
	\begin{aligned}     & \|  q [h_1 , Dh_1 , D^2 h_1] - q [h_2, Dh_2 , D^2 h_2] \|_{C^\alpha} \\
		& \leq C \left( \| h_1\|_{C^{2, \alpha} } + \| h_2\|_{C^{2, \alpha} } \right)  \| h_1 - h_2 \|_{C^{2, \alpha} }
	\end{aligned}
	\ee

	\medskip
	While all these terms are in principle defined only on $\Sigma^n$, periodicity allows us to extend them to all $\Sigma$ naturally. 
	
	\medskip
	\subsubsection{Invertibility of the Jacobi operator}
	Let us consider a Delaunay surface $\Sigma $  with neck size $a\in (0, \frac 12 )$ and period $T$, as introduced in Section \S \ref{d} and its Jacobi operator 
	$$
	J_\Sigma [h]  =  \Delta_\Sigma h + |A_\Sigma|^2 h 
	$$ 
	defined for real functions $h \in C^{2,\alpha} (\Sigma)$ which are periodic in $y_3$.
	This subsection deals with solving a linear problem of the  form 
	\be\label{linear}
	J_\Sigma [h] = E(y)  \inn \Sigma
	\ee
	for right hand sides $E(y)$ that are $T$-periodic in $y_3$ and even in the variables $y_3$ and $y_1$. 
	So, we consider $E\in C^\alpha (\Sigma)$ that satisfies  
	\be\label{simh1} \begin{aligned}
		E(y_1, y_2, y_3) = &E(-y_1 , y_2 , y_3 ) , \\  E(y_1, y_2 , y_3) =& E(y_1, y_2 , -y_3 ), \\ E(y_1, y_2 , y_3) =& E(y_1, y_2 , y_3 +T).
	\end{aligned}
	\ee 
	To solve \equ{linear} in a suitable class of right-hand sides $E$ satisfying \equ{simh1}, we need to identify the class of Jacobi fields that are $T$-periodic in $y_3$.  The 
	invariance of the mean curvature of $\Sigma$ under the rigid motions represented by translations along the 3 coordinate axes, gives that the 3 components of the normal vector are Jacobi fields, namely 
	$$
	J_\Sigma[ \nu_j] =0, \quad    j=1,2,3. 
	$$
	It is known that these functions actually span the space of all $T$-periodic in $y_3$ Jacobi fields. In other words,
	\begin{lemma}\label{lema1}
		If 
		$ Z\in C^2 (\Sigma) $ is  $T$-periodic in the variable $y_3$   then
		$$
		J_\Sigma[Z]=0 \implies  \exists\  \alpha_1, \alpha_2, \alpha_3 \in \R \mid \  Z= \sum_{j=1}^3 \alpha_j \nu_j . 
		$$ 
	\end{lemma}
	See for instance \cite{Caldiroli Musso Iacopetti} Lemma 2.7, for a proof.

	Rather than  solving directly Problem \equ{linear} we consider the following projected version of it, as follows  
	\be \label{pp}
	\begin{aligned}
		J_\Sigma [h] &= E(y)  -c\nu_2   - d   \inn \Sigma,\\
		\int_{\Sigma_0} h  &= 0= \int_{\Sigma_0}  h\nu_2 .
	\end{aligned}
	\ee

	\begin{prop}\label{prop1}
		Let $\alpha \in (0,1)$ and let $E(y)\in C^\alpha (\Sigma) $ satisfy  the symmetries \equ{simh1}, namely $T$-periodic in $y_3$ and even in the $y_3$ and $y_1$ variables.   Then there exist scalars $c$ and $d$ given by \equ{cd} below and a unique solution $ h$ to problem \equ{pp} with the same symmetries. This $h$ defines a linear operator of $h=\mathcal T [E] $ with  
		$$  \| \mathcal T (E)\|_{C^{2,\alpha} (\Sigma)}  \le  C \|E\|_{C^\alpha (\Sigma)}. $$
	\end{prop}
	
	\noindent For the proof, we consider first the more general problem
	\be
	J_\Sigma [h] =  E(y)  \inn \Sigma .
	\label{yy}\ee
	The following result holds. 
	\begin{lemma}\label{lemma2}
		Let $E  \in C^\alpha(\Sigma)$ be $T$-periodic in $y_3$ and such that 
		\be \int_{\Sigma_0} E \nu_j = 0 ,\quad j=1,2,3,\label{orto0}\ee
		where $\nu_j$ are the 3 components of the normal vector to $\Sigma$.
		Then problem \equ{yy} has a unique solution $h(y)$ which is $T$-periodic in $y_3$ and satisfies 
		\be\label{orto1} \int_{\Sigma_0} h \nu_j = 0 ,\quad j=1,2,3. \ee $h$ defines a linear operator $h = \mathcal H (E)$ that satisfies an the estimate of the form 
		\be\label{eli}
		\| \mathcal H (E)\|_{C^{2,\alpha}(\Sigma) } \le   C \| E  \|_{C^{\alpha}(\Sigma) }.
		\ee 
	\end{lemma}

	\proof 
	Let us consider the functions 
	$$\xi_j= (1- \Delta ) \nu_j = (1+|A_\Sigma|^2) \nu_j , \quad j=1,2,3$$  
	and consider the operator 
	$$
	\mathcal A(h)   =    (1- \Delta_\Sigma )^{-1} (h  + |A_\Sigma |^2h), \quad    
	$$
	defined on the space $M$ of functions $h\in C^\alpha (\Sigma)$ that are $T$-periodic in $y_3$  such that  \be\label{ort0}\int_{\Sigma_0} \xi_j h = 0 , \quad j=1,2,3.\ee $\mathcal A$ is a compact operator in the Banach space $M$ endowed with its natural norm.   $\mathcal A$ applies the space $M$ into itself. Indeed, for $h\in M$ we have 
	$$\begin{aligned}
		\int_{\Sigma_0}  \mathcal A(h)\xi_j = & \int_{\Sigma_0}   (1- \Delta_\Sigma )^{-1} (h  + |A_\Sigma|^2h)  
		(1- \Delta_\Sigma ) \nu_j \\
		= &\int_{\Sigma_0} (h  + |A_\Sigma |^2h)\nu_j =  \int_{\Sigma_0} h\xi_j =0.     
	\end{aligned}  
	$$ 
	Now, if  $ h\in \ker (I - \mathcal {A})  $ then  $ h\in C^{2,\alpha}(\Sigma)$ 
	solves $J_\Sigma [h] =0$, hence from Lemma \ref{lema1}, $h$ is a linear combination of the functions $\nu_j$. Conditions \equ{ort0} then imply $h=0$. Now, we observe that Equation \equ{yy} can be written as 
	$$
	(I- \mathcal A)(h)  = G , \quad G=  (\Delta_\Sigma - 1)^{-1}(E). 
	$$
	We observe that  $G\in M$. Indeed, 
	$$  
	\begin{aligned}
		\int_{\Sigma_0}  G\xi_j =&  \int_{\Sigma_0}  (\Delta_\Sigma - 1)^{-1}(E) (\Delta_\Sigma - 1)\nu_j\\ =& \int_{\Sigma_0} E\nu_j = 0 , \quad j=1,2,3 .
	\end{aligned} $$
	Then, Fredholm's alternative yields a unique solution $h\in M $ to this problem and hence of \equ{yy}. That solution defines a linear operator of $E$ that satisfies a bound of the form \equ{eli} from elliptic estimates. 
	Finally,   we  modify $h$ in the form $h_* =   h +\sum_{i=1}^3 \alpha_i \nu_j $ 
	where $\alpha_i$ solves the linear system 
	$$
	\sum_{i=1}^3 \alpha_i  \int_{\Sigma_0} \nu_i \nu_j     = - \int_{\Sigma_0} h\nu_j , \quad j=1,2,3.
	$$
	Then $h_*$ still solves \equ{yy} with conditions \equ{orto1} and satisfies an estimate of the form \equ{eli}. The proof is concluded. \qed

	\bigskip 
	Next, we will solve  the equation  
	\be\label{eq1} 
	J_\Sigma [ \bar h ]   = 1  \inn \Sigma .
	\ee
	By odd symmetries of the functions $\nu_j$, the function $E=1$ satisfies the orthogonality conditions \equ{orto0}
	and the solution $\bar h$ of \equ{eq1} predicted by Lemma \ref{lemma2} exists and is unique.  We need to describe this function in a more explicit form. 
	
	\begin{lemma}\label{lemma3}
		The unique solution of Problem \equ{eq1}   is $T$-periodic and even in $y_3$ and satisfies $$ \int_{\Sigma_0}  \bar h > 0$$      
	\end{lemma}
	
	\proof 
	Using the notation in \cite{Caldiroli Musso Iacopetti}, Lemma 2.6, see also \cite{kapouleas,Mazzeo Pacard,Caldiroli Musso}, the Jacobi operator in $\Sigma$ can be expressed in the simple form   
	$$
	J_\Sigma [h]  =    x(t)^{-2} (h_{\theta\theta} + h_{tt}   + p(t)h  )
	$$
	Here the parametrization of the Delaunay surface  in isothermal coordinates,
	$$
	y(t,\theta) = (x(t) \cos \theta, x(t) \sin \theta, z(t) ) ,
	$$
	where $z(t)$ is strictly increasing and $x(t)$ is positive, even 
	and periodic with a period $2\tau >0$. We assume $x(-\tau) = x(\tau) = a, \quad x(0) = 1-a $, so that also 
	$z(\pm  \tau ) = \pm \frac T2$. See 
	\cite{Caldiroli Musso Iacopetti} for the precise definitions. Of course we have 
	$x(t) =  f(z(t))$. $p(t)$ is an explicit even function. 
	We look for a solution to Problem 
	\equ{eq1} which is a function only of $t$,  $ h= \bar h(t)$. The equation on a periodic cell $(-\frac \tau 2 , \frac \tau 2) $ then reads  
	\be 
	h''   + p(t) h  = x(t)^2 \inn  (- \tau  ,  \tau )  
	\label {er}\ee
	and we are looking for an even solution. 
	In coordinates $(\theta, t)$  we see that $\nu_3$ does not depend on $\theta$, more explicitly, 
	$\nu_3(t) =   - \frac {f'(z(t)) }{ \sqrt{1+ f'(z(t) )^2 }}$. We see that $\nu_3(t)$ is an odd function, positive on 
	$(0, \tau  )$.  This function annihilates the Jacobi operator, and hence the linear operator in 
	\equ{er}.  We look for a solution in the form $h(t) =  \nu_3(t)\omega(t) $. Substituting we obtain 
	$$  \omega (\nu_3''  + p(t) \nu_3) +  \nu_3\omega'' + 2\omega' \nu_3' = x(t)^2  $$ so that 
	$$ 
	(\nu_3^2(t) \omega(t) )'  =    x(t)^2 \nu_3(t),  
	$$
	from where we deduce that a solution to \equ{er} in $(0,\tau  )$ is given by
	$$  \bar h(t) =    \nu_3(t)  \int_0^t    \frac {dz}{\nu_3(z)}   \int_0^z x(s)^2 \nu_3(s) ds  .   $$ 
	This function is regular at $t=0$ and extends as an even solution to \equ{er} to the entire interval $(-\tau ,  \tau )$ and to an even $2\tau$-periodic non-negative solution for all $t$. Hence $\bar h$ solves Problem \equ{eq1} in the entire $\Sigma$. $h$ is positive except at $t=0$. In particular, we  have
	$\int_{\Sigma_0} \bar h > 0$, as desired.\qed

	\subsection*{Proof of Proposition \ref{prop1} }
	
	We begin by identifying the unique candidates for the constants $c$ and $d$. 
	Testing Equation \equ{pp} against $\nu_2$  in the periodic cell $\Sigma_0$  we get 
	$$\begin{aligned}  \int_{\Sigma_0 }   J_\Sigma [h] \nu_2 = \int_{\Sigma_0}  h J_\Sigma [ \nu_2] = 0& =
		\int_{\Sigma_0} E \nu_2  + 
		c\int_{\Sigma_0}  \nu_2^2  + d\int_{\Sigma_0} \nu_2. 
	\end{aligned} 
	$$
	Testing against $\bar h$ in Lemma \ref{lemma3}, we get 
	\be \begin{aligned}
		\int_{\Sigma_0 }   J_\Sigma [h] \bar h    =   \int_{\Sigma_0}  h J_\Sigma [ \bar h]  = \int_{\Sigma_0}  h=0&  = 
		\int_{\Sigma_0} E \bar h -
		c\int_{\Sigma_0} \bar h \nu_2        - d \int_{\Sigma_0}  \bar h   . 
	\end{aligned} 
	\label{nn} \ee
	Since $ \int_{\Sigma_0} \bar h \nu_2 =0 = \int_{\Sigma_0} \nu_2 $  and  $\int_{\Sigma_0}  \bar h >0 $ we can explicitly solve  for the coefficients $c$ and $d$ 
	\be \label{cd}
	d=      \frac{ \int_{\Sigma_0} E \bar h } {\int_{\Sigma_0}  \bar h } , \quad   c=       \frac{ \int_{\Sigma_0} E \nu_2 } {\int_{\Sigma_0}  \nu_2^2 }.
	\ee
	We observe that the choices \equ{cd} yield that the function $\bar E = E - c\nu_2 -d $ satisfies 
	$
	\int_{\Sigma_0} \bar E \nu_2= 0 
	$.  The even symmetries  \equ{simh1} in $E$ and the odd ones for $\nu_1$ and $\nu_3$ yield 
	$$
	\int_{\Sigma_0} \bar E \nu_j= 0, \quad j=1,2,3.  
	$$
	Thus, with these choices of $d$ and $c$, Lemma \ref{lemma3} yields the existence of a unique solution of \equ{pp}
	$T$-periodic in $y_3$ that satisfies conditions \equ{orto1}.  Testing the equation once more against $\bar h$, the computation \equ{nn} now yields    
	$\int_{\Sigma_0} h = 0.  
	$ Uniqueness implies evenness in the variable $y_3$. 
	Hence  with the choices \equ{cd} $h$ is $T$-periodic and even in $y_3$ and solves 
	$$\left\{
	\begin{aligned}
		J_\Sigma [h] &= E(y)  -c\nu_2   - d   \inn \Sigma,\\
		\int_{\Sigma_0} h  &= 0= \int_{\Sigma_0}  h\nu_2 .
	\end{aligned}\right.
	$$ 
	The elliptic estimate \equ{eli}  completes the proof. \qed

	\section{Resolution of the full problem}
	
	We recall that to solve Problem \eqref{3} we just need to find a function 
	$h: \Sigma \to \R$ and constants $\gamma$, $\lambda$ such that equation \eqref{newP} is satisfied, namely
	$$
	\begin{aligned}
		J_{\Sigma^n} [h]   &= E_\gamma (y) 
		+ \, \gamma \, \left( {2  \over T} \, \ln n \, + b(y) \, \right) \,\,\int_{\Sigma_0  }   h  
		+ \, \gamma \, \ell_1 [h] (y) \\
		&+ n^{-1} \, \ell [h, D h, D^2 h] (y)
		+ q[h, D h , D^2 h] (y) + \lambda ,
	\end{aligned}
	$$
	where 
	$$ \begin{aligned}
		E_\gamma(y) &= \gamma \,  F_0^n (y_{3} ) + \frac {\gamma}n F_1^n (y)   +
		{  y_2 \over  n } [ F(y_3) - \gamma \, {\ln n }  \,   \, {2\pi |\Omega_0 | \over  T^2}] .  \end{aligned}
	$$
	$F(y_3)$ and the rest of the terms are defined  in formula \equ{newP}.  
	
	Instead of solving the problem directly, we formulate the projected problem 
	\be\label{p1}
	\begin{aligned}
		J_{\Sigma^n} [h]   &= E_\gamma (y) 
		+ \, \gamma \, \left( {2  \over T} \, \ln n \, + b(y) \, \right) \,\,\int_{\Sigma_0  }   h  
		+ \, \gamma \, \ell_1 [h] (y) \\
		&+ n^{-1} \, \ell [h, D h, D^2 h] (y)
		+ q[h, D h , D^2 h] (y)  + d  + c\nu_2 
	\end{aligned}
	\ee
	and we look for a solution $h$ that satisfies $\int_{\Sigma_0} h = 0$ and  we additionally require $c=0$.

	We formulate this problem using the operator $\mathcal T $ in Proposition \ref{prop1} in the fixed point form 
	\be\label{p2}
	\begin{aligned}
		h &=   {\mathcal  E } (h, \gamma)  \\
		{\mathcal  E } (h, \gamma) &:=\mathcal T (  E_\gamma    
		+ \, \gamma \, \ell_1 [h] 
		+ n^{-1} \, \ell [h, D h, D^2 h] 
		+ q[h, D h , D^2 h] )
	\end{aligned}
	\ee 
	coupled with the relation $c=0$ which becomes 
	\be\label{p3} c(h,\gamma)=   \frac 1{ \int_\Sigma \nu_2^2  } \int_\Sigma (\, E_\gamma    
	+ \, \gamma \, \ell_1 [h]  \\
	+ n^{-1} \, \ell [h, D h, D^2 h] 
	+ q[h, D h , D^2 h] ) \nu_2  \ =\ 0.
	\ee
	The constant $d$ in \equ{p1} is given by 
	$$
	d=   \frac 1{ \int_\Sigma \bar h } \int_\Sigma (\,  E_\gamma +  \gamma \, \ell_1[h] 
	+ n^{-1} \, \ell (h, D h, D^2 h)
	+ q(h, D h , D^2 h) ) \bar h  
	$$
	where $\bar h$ is given by Lemma \ref{lemma3}.
	
	\medskip
	
	The relation \eqref{p3}   will  be  at main order the equation 
	$$
	\int_{\Sigma_0} E_\gamma \nu_2 = 0 .
	$$
	Let us compute the left-hand side of this equation. We will denote by $c_1,c_2, ...$ different positive constants arising in the computation.

	\medskip
	Using the parametrization of the Delaunay surface introduced in  \S \ref{sec3}, we have $\nu_2 = {\sin \theta \over \sqrt{1+ (f')^2}}$ and  $d\sigma = \sqrt{{\det} g} \, dy_3 \, d\theta $ (see \eqref{g}) and hence
	$$
	c_1 := \int_{\Sigma_0} \nu_2^2 = (\int_0^{2\pi} \sin^2 \theta \, d \theta ) \int_{-{T \over 2}}^{T\over 2} {f \over \sqrt{1+ (f')^2 } } \, dy_3 .$$
	Using Propositions \ref{lemma1} and \ref{expN0} we find
	\begin{align*}
		\int_{\Sigma_0} E \, \nu_2 &= {1\over R} \int_0^{2\pi} \int_{-{T \over 2}}^{T\over 2} f \sin^2 \theta  \, \Biggl(    { (2-(f')^2 ) \, f \, f'' \over (1+ (f')^2 )^{5\over 2}} +{1 + 3 (f')^2 \over (1+ (f')^2 )^{3\over 2}} \Biggl)  \, dy_3 \, d\theta \\
		&-\gamma \, {\ln n \over n}  \,  
		{2\pi |\Omega_0| \over  T^2}  \int_0^{2\pi} \int_{-{T \over 2}}^{T\over 2} f^2 \sin^2 \theta \, d y_3 \, d\theta + O ({\gamma \over n})\\
		&= {2\pi \over nT} \, (\int_0^{2\pi} \sin^2 \theta \, d \theta )\,  \left( 2 \, I_a - \gamma \, \ln n \, {|\Omega_0 |\over T} + O(\gamma ) \right) 
	\end{align*}
	where $I_a$ is the integral defined in \eqref{defIa}. Setting 
	$$ \begin{aligned} c_2 = &{2\pi \over T} \, (\int_0^{2\pi} \sin^2 \theta \, d \theta )\, \\ 
		c_3 =&  {|\Omega_0 |\over T} ,\\ 
		c_4 =& {2\pi \over T}  \, \left( \int_{-{T \over 2}}^{T\over 2} {f \over \sqrt{1+ (f')^2 } } \, dy_3  \right)^{-1}  \,
	\end{aligned}$$
	we find 
	\be \label{pp3}
	\begin{aligned}
		{ \int_{\Sigma_0} E_\gamma \, \nu_2  \over \int_{\Sigma_0} \nu_2^2 }&=  \frac {c_4} n  \left( 2 \, I_a -  c_3 \gamma \, \ln n \,  + O(\gamma ) \right)  
	\end{aligned}\ee
	and hence at main order 
	$$ 
	\gamma \approx   \gamma_n := \frac {c_5} {\log n }, \quad  c_5= \frac 2{c_3} I_a.
	$$
	The value of $\gamma_n$ will indeed be a positive number 
	if $I_a>0$, which is precisely what we are assuming. Now, 
	we assume  that we are working in a region of positive scalars $\gamma$  with 
	\be\label{gama} 
	|\gamma -\gamma_n| \le \frac M {\log^2 n}
	\ee
	for a fixed and adequately large number $M$.
	
	\medskip
	Our problem is thus reduced to finding a solution 
	$(h, \gamma)$ in system \eqref{p2}-\eqref{p3}.
	The principle is simple. We look for a solution $h$ whose size is of the order of that of $E_\gamma$ which turns out to be  $O(\gamma)$. The contribution of the terms $h$-dependent of the right-hand side will then be $o(\gamma)$. It is convenient to solve first a "piece" of equation \equ{p2}, 
	$$
	h_0 = \mathcal T ( \gamma F_0^n (y_3)  + \gamma \, \ell_1[h_0] 
	+ n^{-1} \, \ell (h_0, D h_0, D^2 h_0)  + q(h_0,Dh_0,D^2 h_0)) .
	$$
	This problem can be embedded in the space of functions $h$ that in addition to the symmetry and periodicity conditions \equ{simh1} are even  in the $y_2$-variable:  $h(y_1,-y_2,y_3)= h(y_1,y_2,y_3)$.
	Thanks to \eqref{chepalle}-\eqref{chepalle1}-\eqref{chepalle2}, we apply the contraction mapping principle which yields that this problem has a unique small solution $h_0$, with 
	\be \label{hache} \int_{\Sigma_0} h_0 =0, \quad {\mbox {and}} \quad \| h_0 \|_{C^{2, \alpha} (\Sigma_0)} \leq {C \over \log n}.
	\ee
	The reason being that the  terms $\gamma F_0^n (y_3)  + \gamma \, \ell_1[h] 
	+ q(h,Dh,D^2 h))$
	satisfy conditions \equ{simh1} and evenness in $y_2$ if $h$ does. 
	
	Now let us decompose $h=h_0+ h_1$. The problem for $h_1$ then becomes 
	\be \label{p4} \begin{aligned}
		h_1= &\mathcal T \Big ( E_\gamma^1  + \vartheta_1(h_1) + \vartheta_2(h_1)    \Big ) 
	\end{aligned}
	\ee
	where
	$$\begin{aligned}
		E^1_\gamma(y) =& \frac {\gamma}n F_1^n (y) +
		{  y_2 \over  n } [ F(y_3) - \gamma \, {\ln n }  \,   \, {2\pi |\Omega_0 | \over  T^2}] , \\ \vartheta_1 (h_1) &= n^{-1} \,  \ell (h_0+h_1, D (h_0+h_1), D^2 (h_0 + h_1)) \\
		\vartheta_2 (h_1)  &=  \gamma \ell_1 [h_0+ h_1]  -  \gamma \ell_1 [h_0] \\  &+  q(h_0+h_1 ,D(h_0 + h_1),D^2 (h_0 + h_1))\, - \,  q(h_0,Dh_0,D^2 h_0) .
	\end{aligned}$$ 
	The following properties can be immediately checked 
	$$\begin{aligned}
		\|E_\gamma^1\|_{C^\alpha (\Sigma_0)}  \le&  \frac  C n , \\ \|\vartheta_1 (h_1)\|_{C^\alpha (\Sigma_0)} \le& \frac Cn (\gamma + \|h_1\|_{C^{2,\alpha} (\Sigma_0)} )    , \\  \|\vartheta_2 (h_1)\|_{C^\alpha (\Sigma_0)} \le&  C[ \gamma  \|h_1\|_{C^{2,\alpha}}  + \|h_1\|^2_{C^{2,\alpha} (\Sigma_0)} ]
	\end{aligned}
	$$
	(see \eqref{chepalle}, \eqref{chepalle1}, \eqref{chepalle2}).
	Let us consider the equation \equ{p4} for $h_1$ in the region of functions $h_1$ with  $ \int_{\Sigma_0} h_1 =0$ and $\|h_1\|_{C^2_\alpha (\Sigma_0)}  \le \frac D n $ for some suitably large constant $D$. We check that the operator in the right-hand side of \equ{p4} applies this region into itself, and in addition, it is a contraction mapping. 
	We conclude that \equ{p4} has a unique solution with
	\be \label{hache1} \int_{\Sigma_0} h_1 =0, \quad {\mbox {and}} \quad \| h_1 \|_{C^{2, \alpha} (\Sigma_0)} \leq {C \over  n}.
	\ee

	We now have $ h= h_0 + h_1 $ defines an operator  $ h= h(\gamma)$.  Substituting this into the  equation \equ{p3}  we have reduced the full problem to 
	
	\be \label{pp1}
	\int_{\Sigma_0}  ( E^1_\gamma + \vartheta_1(h_1)  +  \vartheta_2(h_1))\nu_2 = 0 . 
	\ee
	Here we have used that because of the even symmetry in $y_2$ we have that 
	$$
	\int_{\Sigma_0}   (\gamma F_0^n (y_3)  + \gamma \, \ell_1[h_0] 
	+ n^{-1} \, \ell (h_0, D h_0, D^2 h_0)  + q(h_0,Dh_0,D^2 h_0)) \nu_2 = 0. 
	$$
	Then the full problem is reduced, after substitution into \equ{pp1} to finding $\gamma$ such that a relation of the following form holds
	$$ 
	\int_{\Sigma_0} (E_\gamma^1 + O(\frac \gamma n) )  \nu_2 = c_2
	(2 I_a- c_3 \gamma \log n  + o(1) ) 
	$$
	where the smaller order terms satisfy also small Lipschitz conditions.
	The choice of $\gamma $ such that this quantity equals zero immediately follows. The proof is concluded.  \qed

	\section{Appendix}\label{appe}
	
	\begin{proof}[Proof of Lemma \ref{LemmaIa}]

		It is straightforward to check that  $I_a >0$ for any  ${1\over 4} \leq a < {1\over 2}$. Since $f$ satisfies \eqref{eqf}, one can show that the following quantity 
		is conserved: 
		\be \label{conservedquantity}
		f^2  - \frac f{ \sqrt {1+ f'(s)^2}}  =  - a(1-a) .
		\ee
		An algebraic manipulation using \eqref{eqf}  and \eqref{conservedquantity} gives that
		$$
		\begin{aligned}
			I_a &= \int_0^{T\over 2} {f_* (s) \over 1+ (f'(s))^2} \, ds, \quad {\mbox {where}} \\
			f_*(t) &= f^2 \, (-1 + 4 (f')^2 ) + \, a \, (1-a) \, (3 + 2 (f')^2) .
		\end{aligned}
		$$
		Since $f(s)^2 \leq (1-a)^2$ for all $s \in [0,{T\over 2}]$ and $a \in [{1\over 4}, {1\over 2} ]$, we readily get
		$$
		f_*(s) > - f(s)^2 + 3a (1-a) \geq 0. \quad $$

		\medskip
		Let us now consider $a$ close to $0$.

		\medskip
		Let us consider the change of variable in the integral \eqref{Ia} given by
		$$
		s= z(t).
		$$
		Using the notation in \cite{Caldiroli Musso, Caldiroli Musso Iacopetti}, $s=z(t)$ is the unique diffeomorphism such that $z(0)=0$, $z' >0$ and
		$$
		y(t,\theta) = (x(t) \cos \theta, x(t) \sin \theta, z(t) ) , \quad x(t)= f(z(t))
		$$
		is a conformal parametrization of the unduloid $\Sigma$, i.e.
		$$
		\pp_t y \cdot \pp_\theta y = 0= |\pp_t y|^2 - |\pp_\theta y |^2.
		$$
		This relation is equivalent to
		\be \label{x}
		x^2= (x')^2 + (z')^2.
		\ee
		The functions $x=x(t)$ and $z=z(t)$ satisfy
		\be \label{eqx-z}
		\begin{aligned}
			x'' &= (1-2 a (1-a)) x - 2 x, \quad x(0)=1-a , \quad x'(0)=0 \\
			z'&= a (1-a) + x^2, \quad z(0) =0.
		\end{aligned}
		\ee
		The function $x$ is $2\tau$-periodic, $x (t+2\tau ) = x(t)$ for all $t$, 
		$x(-\tau) = x(\tau) = a$ and
		$z(\pm  \tau ) = \pm \frac T2$. 
		We refer to	\cite{Caldiroli Musso, Caldiroli Musso Iacopetti, Mazzeo Pacard}
		for the derivation of these facts.	
		
		\medskip
		Using \eqref{x}, we compute
		$$
		ds = z'(t) dt, \quad {df \over ds} = {z' \over x'}, \quad 1+ ({df \over ds})^2 = {x^2 + \left( a (1-a) + x^2 \right)^2}
		$$	
		and hence
		\be\label{Ia1}
		\begin{aligned}
			I_a &= \int_0^{\tau } A_a (t) B_a (t) \, dt \\
			A_a (t)&=  a (1-a) + x^2  \\
			B_a (t) &= - z' + 4 (x')^2 + {a (1-a) \over x^2} \, ( 3 (z')^2 + 2 (x')^2 ) 
		\end{aligned}
		\ee
		We claim that there exists $C>0$ such that for all $a >0$ small enough
		\be \label{xa}
		x(t) = {\rm sech } (t)\,  \left( 1 + a (-1 + t \tanh t) \right) + a^2\,  \varphi_a (t)
		\ee
		where
		\be \label{xa1}
		\| \varphi_a \|_{C^2_{\rm loc} (\R )} \leq C.
		\ee
		For the moment we assume the validity of \eqref{xa}-\eqref{xa1} and proceed with our argument to prove that $I_a >0$ for all $a>0$ small. 
		Using the notation introduced in \eqref{Ia1}, we differentiate $I_a$ with respect to $a$
		\begin{align*}
			\partial_a I_a = \int_0^\tau \left( \partial_a A_a \, B_a + A_a \, \partial_a B_a \right) \, dt + (A_a B_a) (\tau) \, \left( \partial_a \tau \right). 
		\end{align*}
		In Lemma 2.2 in \cite{Caldiroli Musso} it is proved that $\tau = -2 \log a + \tilde \tau (a)$ with $\tilde \tau (a)$ a smooth function which is uniformly bounded with its derivative as $a \to 0$. This fact, in combination with \eqref{xa}-\eqref{xa1}, give that
		$$
		(A_a B_a) (\tau) \, \left( \partial_a \tau \right) = O(a) \quad {\mbox {as}} \quad a \to 0.
		$$
		Let 
		\be \label{defx0}
		x_0 (t) = \sech t, \quad \varphi (t) = \sech t \, (-1 + t \tanh t).
		\ee
		Using \eqref{x} and \eqref{eqx-z}, we get
		\begin{align*}
			(\pp_a A_a )_{a=0} &= 1+ 2 x_0 \varphi \\
			(\pp_a B_a )_{a=0} &=- 9 x_0^2 +2 + 8 x_0 \varphi -20 x_0^3 \varphi .
		\end{align*}
		Since $I_0=0$, $A_0= x_0^2$ and $B_0 =   4 x_0^2 -5 x_0^4 $,  Taylor expanding $I_a$ in  
		\eqref{Ia1} we obtain
		\begin{align*}
			I_a &=  a \, \int_0^{\infty} \left( 1+ 2 x_0 \varphi \right) \left(  4 x_0^2 -5 x_0^4 \right) + x_0^2 \left( - 9 x_0^2 +2 + 8 x_0 \varphi -20 x_0^3 \varphi\right) \, dt \\
			&+ O(a^2) \quad {\mbox {as}} \quad a \to 0.
		\end{align*}
		We use that
		\begin{align*}
			\int \sech^2 (t) \, dt &= \tanh t + c, \quad \\
			\int \sech^4 (t) \, dt &= {1\over 3} \, (2 + \cosh(2 t)) \, \sech^2(t) \, \tanh(t)+c , \quad \\ \int \sech^6 (t) \, dt &={1\over 15}  \left(8 + 6 \cosh (2t) + \cosh (4t) \right)) \, \sech^4 (t) \, \tanh (t) +c \\
			\int \sech^4 (t) \,  t \, \tanh t  \, dt &= 
			{1\over 48}  (-12 t + 4 \sinh (2 t) + \sinh (4 t)) \sech^4(t) + c\\
			\int \sech^6 (t) \, t \, \tanh t  \, dt &= 
			{1\over 360}  (-60 t + 15 \sinh (2 t) + 6 \sinh (4 t) + \sinh (6 t)) \,  \sech^6 (t) + c
		\end{align*}
		to get
		\begin{align*}
			\int_0^\infty x_0^2 (t) \, dt &=1, \quad \int_0^\infty x_0^4 (t) \, dt = {2\over 3} , \quad \int_0^\infty x_0^6 (t) \, dt = {8 \over 15}\\
			\int_0^\infty x_0^4 (t) \,  t \, \tanh t  \, dt &= {1\over 6}, \quad  \int_0^\infty x_0^6 (t) \, t \, \tanh t  \, dt = {4 \over 45}.
		\end{align*}
		This gives
		\begin{align*}
			\int_0^{\infty} & \left( 1+ 2 x_0 \varphi \right) \left(  4 x_0^2 -5 x_0^4 \right) + x_0^2 \left( - 9 x_0^2 +2 + 8 x_0 \varphi -20 x_0^3 \varphi\right) \, dt\\
			&= \int_0^\infty \left( -30 \,  x_0^4 + 6 \, x_0^2 + 30 \, x_0^6 + 16 \, x_0^4 \, t \, \tanh t - 30 \, x_0^6 \, t \,\tanh t  \right) \, dt \\
			&= 2  >0 
		\end{align*}
		and hence $I_a >0$ for all $a>0$ small.
		
		\medskip
		We now prove \eqref{xa}-\eqref{xa1}. Observe that the function $x_0$ in \eqref{defx0} solves
		$$
		x''= x - 2 x^3, \quad {\mbox {in}} \quad \R, \quad x (0)=1, \quad x'(0)=0.
		$$
		The proof of \eqref{xa}-\eqref{xa1} is based on a formal linearization of \eqref{eqx-z} in $a$, at $a=0$. Hence we  
		consider the linearized operator around $x_0$
		$$
		L_0 (\varphi ) = \varphi'' - \varphi + 6 x_0^2 \varphi.
		$$
		A direct computation gives that the function $\varphi$ defined in \eqref{defx0} solves
		$$
		L_0 (\varphi) =-2 x_0 , \quad \varphi (0)=-1, \quad \varphi' (0)=0.
		$$
		Replacing the expression for $x$ in \eqref{xa} in \eqref{eqx-z}, we get that $\varphi_a$ solves
		\be\label{eqvarphia}
		\begin{aligned}
			L_0 (\varphi_a) &=2 a^2 x_0 - 2a^2 (1-a) \varphi - 6 x_0 a^2 \varphi^2 - 2 a^3 \varphi^3 + N (\varphi_a) \\
			\varphi_a (0)&=0, \quad \varphi_a' (0) = 0
		\end{aligned}
		\ee
		where
		$$
		N(\varphi_a )= 2 a (1-a) \varphi_a - 6x_0 [ (a \varphi + \varphi_a)^2 - a^2 \varphi^2] -2 [ (a\varphi  + \varphi_a)^3 - a^3 \varphi^3].
		$$
		The function $x_0' (t) = \tanh (t) \, \sech (t)$ solves $L_0 (x_0' ) =0$. Using the variation of parameter formula we check that, for a bounded function $h$, 
		\be \label{solode}
		\psi  (t) = \Theta (h) (t) := x_0' (t) \int_0^t {ds \over (x_0')^2 (s) } \int_0^s h(\eta ) x_0' (\eta ) \, d\eta
		\ee
		is the only solution to
		$$
		L_0 (\psi) = h, \quad \psi(0) =0, \quad \psi' (0)=0.
		$$
		Besides there exists a constant $C>0$ such that
		\be \label{apriori}
		\| \Theta (h) \|_{C^{2, \alpha}_{\rm loc} (\R)} \leq C \| h \|_{C^{0, \alpha}_{\rm loc} (\R)}.
		\ee
		In terms of the operator $\Theta$ the solution $\varphi_a$ to \eqref{eqvarphia} is a fixed point for the problem
		$$
		\varphi_a = {\mathcal A} (\varphi_a) := \Theta \left( 2 a^2 x_0 - 2a^2 (1-a) \varphi - 6 x_0 a^2 \varphi^2 - 2 a^3 \varphi^3 + N (\varphi_a) \right).
		$$
		Using Contraction Mapping, we show that the above problem has a unique fixed point in the set
		$$
		{\mathcal B} = \{ \psi \in C^{2, \alpha}_{\rm loc} (\R) \, : \, \| \psi \|_{C^{2, \alpha}_{\rm loc} (\R)} \leq r a^2\}.
		$$	
		For any $\psi \in {\mathcal B}$, a direct use of the explicit formula \eqref{solode} gives
		\begin{align*}
			\| {\mathcal A} (\psi ) \|_{C^{2, \alpha}_{\rm loc} (\R)} & \leq C \| 2 a^2 x_0 - 2a^2 (1-a) \varphi - 6 x_0 a^2 \varphi^2 - 2 a^3 \varphi^3 + N (\psi) \|_{C^{0, \alpha}_{\rm loc} (\R)} \\
			&\leq C \| 2 a^2 x_0 - 2a^2 (1-a) \varphi - 6 x_0 a^2 \varphi^2 - 2 a^3 \varphi^3  \|_{C^{0, \alpha}_{\rm loc} (\R)}  \\
			&+ C \| N (\psi) \|_{C^{0, \alpha}_{\rm loc} (\R)} \leq c_1 a^2,
		\end{align*}
		for some $c_1$. Taking $r>c_1$, this estimate shows that the map ${\mathcal A}$ send ${\mathcal B}$ into itself.
		
		Besides, for any $\psi_1$, $\psi_2 \in {\mathcal B}$, we see that
		\begin{align*}
			\| {\mathcal A} (\psi_1 ) - {\mathcal A} (\psi_2 )  \|_{C^{2, \alpha}_{\rm loc} (\R)} & \leq C \|  N (\psi_1 ) - N(\psi_2)  \|_{C^{0, \alpha}_{\rm loc} (\R)} \\
			&\leq C a (1-a) \| \psi_1 - \psi_2   \|_{C^{0, \alpha}_{\rm loc} (\R)}  \\
			&+ C \| x_0 a \varphi (\psi_1  - \psi_2) \|_{C^{0, \alpha}_{\rm loc} (\R)}   + C \| \psi_1^2 - \psi_2^2 \|_{C^{0, \alpha}_{\rm loc} (\R)} \\
			&\leq  C a \| \psi_1 - \psi_2   \|_{C^{0, \alpha}_{\rm loc} (\R)} \\
			& +  C \| \psi_1^2 + \psi_2^2 \|_{C^{0, \alpha}_{\rm loc} (\R)}
			\| \psi_1 - \psi_2 \|_{C^{0, \alpha}_{\rm loc} (\R)}
		\end{align*}
		for a constant $C$ independent of $a$, whose value changes from line to line, and within the same line. This gives that for any sufficiently $a>0$, ${\mathcal A}$ defines a contraction mapping of the set ${\mathcal B}$. This concludes the proof of \eqref{xa}-\eqref{xa1}.

	\end{proof}
	
	\subsection*{Acknowledgments}
	M.~del Pino has been supported by the Royal Society Research Professorship grant RP-R1-180114 and by  the  ERC/UKRI Horizon Europe grant  ASYMEVOL, EP/Z000394/1.
	M.~Musso has been supported by EPSRC research Grant EP/T008458/1. ANID Chile has supported A. Zúñiga under grant FONDECYT de Iniciaci\'on en Investigaci\'on  11201259.


\begin{thebibliography}{AAA}
		
		
		
		
		\bibitem{zuniga}  S. Alama, L. Bronsard, I. Topaloglu, A. Zuniga,  {\em A nonlocal isoperimetric problem with density perimeter.} Calc. Var. Partial Differential Equations 60 (2021), no. 1, Paper No. 1, 27 pp.
		
		\bibitem{bohr}
		N. Bohr,  {\em Neutron capture and nuclear constitution} Nature 137, 344–348 (1936).
		
		\bibitem{Bohr Wheeler} N. Bohr,  J. A. X  Wheeler, {\em The mechanism of nuclear fission.} Physical Review, 56(5):426 (1939). 
		
		\bibitem{Bonacini-Cristofri} M. Bonacini, R. Cristoferi,  {\em Local and global minimality results for a nonlocal isoperimetric problem on $R^N$.} SIAM J. Math. Anal. 46 (2014), no. 4, 2310–2349.
		
		
		\bibitem{Caldiroli Musso}
		P. Caldiroli, M. Musso,  {\em Embedded tori with prescribed mean curvature.} Adv. Math. 340 (2018), 406–458.
		
		\bibitem{Caldiroli Musso Iacopetti}
		P. Caldiroli, A. Iacopetti, M. Musso,  {\em On the non-existence of compact surfaces of genus one with prescribed, almost constant mean curvature, close to the singular limit.} Adv. Differential Equations 27 (2022), no. 3-4, 193–252. 
		
		
		\bibitem{Choksi Peletier}
		R. Choksi, M.A. Peletier, {\em 
			Small volume-fraction limit of the diblock copolymer problem: II. Diffuse-interface functional.} 
		SIAM J. Math. Anal. 43 (2011), no. 2, 739–763.
		
		\bibitem{choksi}
		R. Choksi, C. Muratov, I. Topaloglu, {\em An old problem resurfaces nonlocally: Gamow's liquid drops inspire today's research and applications.} Notices Amer. Math. Soc. 64 (2017), no. 11, 1275–1283.
		
		\bibitem{DeGiorgi}  E. De Giorgi, {\em Sulla proprietá isoperimetrica dell’ipersfera, nella classe degli insiemi
			aventi frontiera orientata di misura finita.} Atti Accad. Naz. Lincei. Mem. Cl. Sci. Fis.
		Mat. Nat. Sez. I (8) (1958), 33–44.
		
		\bibitem{figalli}   A. Figalli, N. Fusco, F. Maggi, V. Millot, M. Morini,  
		{\em Isoperimetry and stability properties of balls with respect to nonlocal energies.} 
		Comm. Math. Phys. 336 (2015), no. 1, 441–507.
		
		\bibitem{frank3} 
		L. Emmert, R. Frank, T. König,  
		{\em Liquid drop model for nuclear matter in the dilute limit.}
		SIAM J. Math. Anal. 52 (2020), no. 2, 1980–1999.
		
		\bibitem{frank1}
		R. Frank, R. Killip, P.T. Nam,  {\em Nonexistence of large nuclei in the liquid drop model.} Lett. Math. Phys. 106 (2016), no. 8, 1033–1036. 
		
		\bibitem{frank2} 
		R. Frank, P.T. Nam,
		{\em Existence and nonexistence in the liquid drop model}
		Preprint arXiv:2101.02163 
		
		
		
		\bibitem{frank4} 
		R. Frank,  {\em Non-spherical equilibrium shapes in the liquid drop model. } J. Math. Phys. 60 (2019), no. 7, 071506.
		
		\bibitem{Gamow}  G. Gamow, {\em Mass defect curve and nuclear constitution,} Proc. R. Soc. Lond. A 126 (1930), no. 803, pp. 632–644.
		
		\bibitem{julin} 
		V. Julin, {\em Isoperimetric problem with a Coulomb repulsive term.} Indiana Univ. Math. J. 63 (2014), no. 1, 77–89
		
		\bibitem{kapouleas}
		N. Kapouleas,  {\em Complete constant mean curvature surfaces in Euclidean three-space.} Ann. of Math. (2) 131 (1990), no. 2, 239–330. 
		
		
		
		\bibitem{Knupfer-Muratov}
		H. Knüpfer,  C. Muratov, 
		{\em On an isoperimetric problem with a competing nonlocal term II: The general case.} (English summary)
		Comm. Pure Appl. Math. 67 (2014), no. 12, 1974–1994.
		
		\bibitem{Lieb-Loss}   E. Lieb, M. Loss,   {\em Analysis.} Graduate Studies in Mathematics, 14. American Mathematical Society, Providence, RI, 1997. 
		
		\bibitem{maggi}
		F. Maggi, {\em Sets of finite perimeter and geometric variational problems.} An introduction to geometric measure theory. Cambridge Studies in Advanced Mathematics, 135. Cambridge University Press, Cambridge, 2012.
		
		\bibitem{Lu Otto}
		J. Lu, F. Otto, {\em Nonexistence of a minimizer for Thomas-Fermi-Dirac-von Weizsäcker model.} Comm. Pure Appl. Math. 67 (2014), no. 10, 1605–1617.
		
		\bibitem{Mazzeo Pacard}
		R. Mazzeo, F. Pacard, 
		{\em Constant mean curvature surfaces with Delaunay ends.}
		Comm. Anal. Geom. 9 (2001), no. 1, 169–237.
		
		\bibitem{pressley} 
		A. Pressley,  {\em Elementary differential geometry.} Springer Undergraduate Mathematics Series. Springer-Verlag London, Ltd., London, 2010.
		
		
		\bibitem{RenWei0}
		X. Ren, J. Wei,   {\em 
			Existence and stability of spherically layered solutions of the diblock copolymer equation}. 
		SIAM J. Appl. Math. 66 (2006), no. 3, 1080–-1099.
		
		\bibitem{RenWei1} 
		X. Ren, J. Wei,   {\em 
			A toroidal tube solution to a problem involving mean curvature and Newtonian potential.} 
		Interfaces Free Bound. 13 (2011), no. 1, 127–154.
		
		\bibitem{RenWei2} 
		X. Ren, J. Wei, {\em Double tori solution to an equation of mean curvature and Newtonian potential.} Calc. Var. Partial Differential Equations 49 (2014), no. 3-4, 987–1018.
		
		
		
		\bibitem{Riesz} 
		F. Riesz, Sur une in´egalit´e int´egrale. J. London Math. Soc. 5 (1930), 162–168.
		
		
		
		
		\bibitem{Xu-Du} 
		Z. Xu, Q. Du {\em 
			Bifurcation and fission in the liquid drop model: a phase-field approach.} 
		Preprint arXiv:2302.14449 
		
	\end{thebibliography}
\end{document}